\documentclass[reqno,12pt]{amsart} 
\usepackage{amssymb} 
 
\theoremstyle{plain} 
\newtheorem{theorem}{\indent\sc Theorem}[section] 
\newtheorem{lemma}[theorem]{\indent\sc Lemma}
\newtheorem{corollary}[theorem]{\indent\sc Corollary}
\newtheorem{proposition}[theorem]{\indent\sc Proposition}
\newtheorem{claim}[theorem]{\indent\sc Claim}
\theoremstyle{definition} 
\newtheorem{definition}[theorem]{\indent\sc Definition}
\newtheorem{remark}[theorem]{\indent\sc Remark}

\makeatletter
  
  \@addtoreset{equation}{section}
\makeatother
\begin{document}

\title[Ricci curvature]{Elliptic PDEs on compact Ricci limit spaces and applications} 

\author[Shouhei Honda]{Shouhei Honda} 



\address{ 
Faculty of Mathmatics \endgraf
Kyushu University \endgraf 
744, Motooka, Nishi-Ku, \endgraf 
Fukuoka 819-0395 \endgraf 
Japan
}
\email{honda@math.kyushu-u.ac.jp}

\address{ 
New address: \endgraf
Mathematical Institute \endgraf 
Tohoku University \endgraf
Sendai 980-8578 \endgraf 
Japan
}
\email{shonda@m.tohoku.ac.jp}

\maketitle
\centerline{\textit{Dedicated to the memory of Kentaro Nagao.}}
\begin{abstract}
In this paper we study elliptic PDEs on compact Gromov-Hausdorff limit spaces of Riemannian manifolds with lower Ricci curvature bounds.
In particular we establish continuities of geometric quantities, which include solutions of Poisson's equations, eigenvalues of Schr$\ddot{\text{o}}$dinger operators, generalized Yamabe constants and eigenvalues of the Hodge Laplacian, with respect to the Gromov-Hausdorff topology.
We apply these to the study of second-order differential calculus on such limit spaces.
\end{abstract}
\tableofcontents
\section{Introduction}
Let $(X, \upsilon)$ be a compact metric measure space, which means that $X$ is a compact metric space and $\upsilon$ is a Borel probability measure on $X$.
We say that $(X, \upsilon)$ is a \textit{smooth $n$-dimensional compact metric measure space} if $X$ is a smooth $n$-dimensional closed Riemannian manifold and $\upsilon=H^n/H^n(X)$, where $H^n$ is the $n$-dimensional Hausdorff measure.

Let $n \in \mathbf{N}_{\ge 2}$, $K \in \mathbf{R}$ and $d \in \mathbf{R}_{>0}$.
We denote by $M(n, K, d)$ the set of (isometry classes of) smooth $n$-dimensional compact metric measure spaces $(Y, \nu)$ with $\mathrm{diam}\,Y \le d$ and $\mathrm{Ric}_Y \ge K(n-1)$, where $\mathrm{diam}\,Y$ is the diameter and $\mathrm{Ric}_Y$ is the Ricci curvature. 

Let $\overline{M(n, K, d)}$ be the set of (measured) Gromov-Hausdorff limit compact metric measure spaces of sequences in $M(n, K, d)$.
See \cite{fu, gr} or subsection $2.2$ for the definition of the Gromov-Hausdorff convergence.
Note that $\overline{M(n, K, d)}$ is compact with respect to the Gromov-Hausdorff topology.
A compact metric measure space which belongs to $\overline{M(n, K, d)}$ is said to be a \textit{Ricci limit space} in this paper.

Assume $(X, \upsilon) \in \overline{M(n, K, d)}$ with $\mathrm{diam}\,X>0$.

The main purpose of this paper is to study elliptic PDEs on $(X, \upsilon)$.
In order to introduce the detail we first recall the Laplacian $\Delta^{\upsilon}$ on $X$ defined by Cheeger-Colding in \cite{ch-co3}.

Cheeger-Colding proved in \cite{ch-co3} that $(X, \upsilon)$ is rectifiable. 
As a corollary they constructed the canonical cotangent bundle $T^*X$ of $X$.
The fundamental properties include the following:
\begin{itemize}
\item Every fiber $T^*_xX$ is a finite dimensional Hilbert space. We denote the inner product by $\langle \cdot, \cdot \rangle$ for short.
\item For any open subset $U$ of $X$, $1<p<\infty$ and $f \in H^{1, p}(U)$, $f$ has the canonical differential $df(x) \in T^*_xX$ for a.e. $x \in U$, where $H^{1, p}(U)$ is the Sobolev space. In particular every Lipschitz function $g$ on $X$ is differentiable at a.e. $x \in X$ in this sense. 
\end{itemize}
See for instance subsections $2.3.2$ and $2.5$ for the details.

For an open subset $U$ of $X$, let us denote by $\mathcal{D}^2(\Delta^{\upsilon}, U)$ the set of $f \in H^{1, 2}(U)$ satisfying that there exists $g \in L^2(U)$ such that
\[\int_U\langle df, dh\rangle d\upsilon =\int_Ughd\upsilon\] 
for every Lipschitz function $h$ on $U$ with compact support.
Since $g$ is unique if it exists, we denote it by $\Delta^{\upsilon}f$.
\subsection{Poisson's equations}
Let us consider \textit{Poisson's equation} for given $g \in L^2(X)$:
\begin{equation}\label{0}
\Delta^{\upsilon}f=g
\end{equation}
on $X$.
We introduce a continuity of solutions of Poisson's equations with respect to the Gromov-Hausdorff topology.
\begin{theorem}\label{pois}
We have the following:
\begin{enumerate}
\item A solution $f \in \mathcal{D}^2(\Delta^{\upsilon}, X)$ of (\ref{0}) exists if and only if 
\begin{align}\label{55}
\int_Xgd\upsilon=0.
\end{align}
\item If (\ref{55}) holds, then there exists a unique solution $f \in \mathcal{D}^2(\Delta^{\upsilon}, X)$ of (\ref{0}) with
\[\int_Xfd\upsilon=0.\]
We denote $f$ by $(\Delta^{\upsilon})^{-1}g$.
\item Let $(X_i, \upsilon_i)$ Gromov-Hausdorff converges to $(X_{\infty}, \upsilon_{\infty})$ (we write it $(X_i, \upsilon_i) \stackrel{GH}{\to} (X_{\infty}, \upsilon_{\infty})$ for short in this paper) in $\overline{M(n, K, d)}$ with $\mathrm{diam}\,X_{\infty}>0$, and let $\{g_i\}_{i \le \infty}$ be an $L^2$-weak convergent sequence on $X_{\infty}$ of $g_i \in L^2(X_i)$ with 
\[\int_{X_i}g_id\upsilon_i=0.\]
Then $(\Delta^{\upsilon_i})^{-1}g_i, \nabla ((\Delta^{\upsilon_i})^{-1}g_i)$ $L^2$-converge strongly to $(\Delta^{\upsilon_{\infty}})^{-1}g_{\infty}, \nabla ((\Delta^{\upsilon_{\infty}})^{-1}g_{\infty})$ on $X_{\infty}$, respectively.
\end{enumerate}
\end{theorem}
See \cite{holp, KS, ks2} or subsection $2.5.2$ for the definition of the $L^p$-convergence with respect to the Gromov-Hausdorff topology.
It is worth pointing out that the regularity of Poisson's equations given in \cite{jiang} by Jiang yields that if $g \in L^{\infty}(X)$ with (\ref{55}) and $||g||_{L^{\infty}}\le L$, then 
\begin{align}\label{lipreg}
||\nabla (\Delta^{\upsilon})^{-1}g||_{L^{\infty}} \le C(n, K, d, L),
\end{align}
where $C(n, K, d, L)$ is a positive constant depending only on $n, K, d$ and $L$.  
See \cite[Theorem 3.1]{jiang}.

Let $B_r(x):=\{y \in X; d_X(x, y)<r\}$ and let $\overline{B}_r(x):=\{y \in X; d_X(x, y) \le r\}$, where $d_X(x, y)$ is the distance between $x$ and $y$ in $X$.
By combining Theorem \ref{pois} with the existence of a good cut-off function constructed by Cheeger-Colding in \cite{ch-co}, we have the following local version of $(4)$ of Theorem \ref{pois}:
\begin{theorem}\label{app3}
Let $(X_i, \upsilon_i) \stackrel{GH}{\to} (X_{\infty}, \upsilon_{\infty})$ in $\overline{M(n, K, d)}$,
let $R>0$, let $\{x_i\}_{i \le \infty}$ be a convergent sequence of $x_i \in X_i$, and let $f_{\infty} \in \mathcal{D}^2(\Delta^{\upsilon_{\infty}}, B_R(x_{\infty}))$.
Then for any $r<R$ with $X_{\infty} \neq \overline{B}_r(x_{\infty})$ and $L^2$-weak convergent sequence $\{g_i\}_{i \le \infty}$ on $B_r(x_{\infty})$ of $g_i \in L^2(B_r(x_i))$ with $g_{\infty}=\Delta^{\upsilon_{\infty}}f_{\infty}|_{B_r(x_{\infty})}$, there exist a subsequence $\{i(j)\}_j$ and a sequence $\{f_{i(j)}\}_j$ of $f_{i(j)} \in \mathcal{D}^2(\Delta^{\upsilon_{i(j)}}, B_r(x_{i(j)}))$ such that $g_{i(j)}=\Delta^{\upsilon_{i(j)}}f_{i(j)}$ and that $f_{i(j)}, \nabla f_{i(j)}$ $L^2$-converge strongly to $f_{\infty}, \nabla f_{\infty}$ on $B_r(x_{\infty})$, respectively.
\end{theorem}
Recall that Ding proved in \cite{di2} that the uniform limit function of a sequence of harmonic functions with respect to the Gromov-Hausdorff topology is also harmonic (see also \cite{holip, holp} for a different approach). 
As a corollary of Theorem \ref{app3} we have the converse:
\begin{corollary}\label{apphar}
Let $x \in X$, let $R>0$ and let $h$ be a harmonic function on $B_R(x)$, i.e., $h \in \mathcal{D}^2(\Delta^{\upsilon}, B_R(x))$ with $\Delta^{\upsilon}h=0$.
Then for every $r<R$, we see that $h|_{B_r(x)}$ is the uniform limit function of a sequence of harmonic functions with respect to the Gromov-Hausdorff topology in the following sense:
Let $\{(X_i, \upsilon_i)\}_i$ be a convergent sequence to $(X, \upsilon)$ in $\overline{M(n, K, d)}$ and let $\{x_i\}_i$ be a convergent sequence of $x_i \in X_i$ to $x$.
Then there exist a subsequence $\{i(j)\}_j$ and a sequence $\{h_{i(j)}\}_j$ of harmonic functions $h_{i(j)}$ on $B_r(x_{i(j)})$ such that $h_{i(j)}$ converges uniformly to $h$ on $B_r(x)$ (see subsection $2.3$ for the definition of the uniform convergence in this setting).
In particular if $h \ge 0$ on $B_R(x)$, then we have the following Cheng-Yau type gradient estimate \cite{ch-yau}:
\begin{align}\label{6554}
\mathrm{Lip} h(y)\le \frac{C(n)R^2(R|K|(n-1)+1)}{(R^2-r^2)R}h(y)
\end{align}
for every $y \in B_r(x)$, where 
\begin{align}\label{55m}
\mathrm{Lip}h(y):= \lim_{t \to 0}\left(\sup_{z \in B_t(y) \setminus \{y\}}\frac{|h(y)-h(z)|}{d_X(y, z)}\right).
\end{align}
\end{corollary}
Note that (\ref{6554}) follows directly from applying the original Cheng-Yau gradient estimate \cite{ch-yau} to $h_{i(j)}$ in the case that $(X_i, \upsilon_i) \in M(n, K,d)$ for every $i$
and that a Cheng-Yau type gradient estimate for harmonic functions on metric measure spaces in more general setting was already known via a different approach by Hua-Kell-Xia in \cite{hkz}.

We now introduce an application of Corollary \ref{apphar}.

In \cite{ho0} we gave the notion of a weakly second-order differential structure (or system) on rectifiable metric measure spaces and we knew that $(X, \upsilon)$
has such a structure associated with a convergent sequence $(X_i, \upsilon_i) \stackrel{GH}{\to} (X, \upsilon)$ of $(X_i, \upsilon_i) \in M(n, K, d)$. For example if $(X, \upsilon)$ is a smooth metric measure space, then a smooth coordinate system of $X$ gives an example of such systems.
See subsection $2.3$ for the precise definitions.

Let $\{(X_i^j, \upsilon_i^j)\}_{i<\infty, j \in \{1, 2\}}$ be two convergent sequences of $(X_i^j, \upsilon_i^j) \in M(n, K, d)$ to $(X, \upsilon)$ and let $\mathcal{A}^j_{2\mathrm{nd}}$ be the weakly second-order differential system on $(X, \upsilon)$ associated with $\{(X_i^j, \upsilon_i^j)\}_i$ for every $j \in \{1, 2\}$.

Let us consider the following natural question:

$\\ $
\textbf{Question $1$.} Are $\mathcal{A}_{2\mathrm{nd}}^1$ and $\mathcal{A}_{2\mathrm{nd}}^2$ compatible?
$\\ $

Note that we say that two weakly second-order differential systems $\mathcal{A}^1$ and $\mathcal{A}^2$ on $(X, \upsilon)$ are \textit{compatible} if $\mathcal{A}^1 \cup \mathcal{A}^2$ is also a weakly second-order differential system on $(X, \upsilon)$.
Note that the rectifiable version of this question has always a positive answer, i.e., for any two rectifiable systems $\mathcal{A}_{\mathrm{rec}}^1$ and  $\mathcal{A}_{\mathrm{rec}}^2$ on $(X, \upsilon)$ we easily see that  $\mathcal{A}_{\mathrm{rec}}^1 \cup \mathcal{A}_{\mathrm{rec}}^2$ is also a rectifiable system on $(X, \upsilon)$.

By Corollary \ref{apphar} we can give an answer to this question:
\begin{theorem}\label{221}
$\mathcal{A}_{2\mathrm{nd}}^1$ and $\mathcal{A}_{2\mathrm{nd}}^2$ are compatible.
In particular the weakly second-order differential structure and the Levi-Civita connection $\nabla^{g_X}$ on $(X, \upsilon)$ defined in \cite{ho0} are canonical.
\end{theorem}
Let $\nu$ be a Borel probability measure on $X$ with $(X, \nu) \in \overline{M(n, K, d)}$.
Then Cheeger-Colding proved in \cite{ch-co3} that $\upsilon$ and $\nu$ are mutually absolutely continuous on $X$.  
This gives that the weakly second-order differential structure on $(X, \upsilon)$ does not depend on the choice of limit measures and that it depends only on the (Riemannian) metric structure.
Thus in this paper we write the Levi-Civita connection on $(X, \upsilon)$ defined in \cite{ho0} by $\nabla^{g_X}$ in order to distinguish this from Gigli's Levi-Civita connection $\nabla^{\upsilon}$ defined in \cite{gigli}, where $g_X$ is the Riemannian metric of $X$.
See subsections $1.4$ and $2.5.4$ for the detail.
\subsection{Schr$\ddot{\text{o}}$dinger equations}
Let us discuss \textit{the Schr$\ddot{\text{o}}$dinger operator} with a potential $g \in L^q(X)$:
\begin{align}\label{778}
\Delta^{\upsilon}+g
\end{align}
on $X$, where $q>2$.

Fukaya conjectured in \cite{fu} that eigenvalues of the Laplacian behave continuously on $\overline{M(n, K, d)}$. 
Cheeger-Colding solved this conjecture in \cite{ch-co3}.
We will give a generalization of the continuity to Schr$\ddot{\text{o}}$dinger operators.

It is easy to check that if $(X, \upsilon)$ satisfies the $(2q/(q-2), 2)$-Sobolev inequality on $X$ for some pair $(A, B)$ of nonnegative constants $A, B$, then the spectrum of (\ref{778}) is discrete,
where we say that \textit{$(X, \upsilon)$ satisfies the $(q, p)$-Sobolev inequality on $X$ for $(A, B)$} if 
\[\left(\int_X|f|^qd\upsilon\right)^{p/q} \le A \int_X|df|^pd\upsilon+B\int_X|f|^pd\upsilon\]
for every Lipschitz function $f$ on $X$. See Proposition \ref{discrete}.
Then we write the spectrum by 
\[\lambda^g_{0}(X)\le \lambda^g_1(X) \le \lambda_{2}^g(X) \le \cdots.\]
\begin{theorem}\label{schro}
Let $2<p < \infty$ and let $(X_i, \upsilon_i) \stackrel{GH}{\to} (X_{\infty}, \upsilon_{\infty})$ in $\overline{M(n, K, d)}$ with $\mathrm{diam}\,X_{\infty}>0$.
Assume that there exist $A, B \ge 0$ such that for every $i<\infty$, $(X_i, \upsilon_i)$ satisfies the $(2p/(p-2), 2)$-Sobolev inequality on $X_i$ for $(A, B)$.
Then we have the following:
\begin{enumerate}
\item $(X_{\infty}, \upsilon_{\infty})$ satisfies the $(2p/(p-2), 2)$-Sobolev inequality on $X_{\infty}$ for $(A, B)$. In particular the spectrum of $\Delta^{\upsilon_{\infty}} + g_{\infty}$ is discrete.
\item For any $q>p/2$ and $L^q$-weak convergent sequence $\{g_i\}_{i \le \infty}$ on $X_{\infty}$ of $g_i \in L^q(X_i)$ we have
\begin{align}\label{schco}
\lim_{i \to \infty}\lambda_{k}^{g_i}(X_i)=\lambda_{k}^{g_{\infty}}(X_{\infty})
\end{align}
for every $k \ge 0$. Moreover for every $k$, if $f_{\infty} \in L^2(X_{\infty})$ is the $L^2$-weak limit on $X_{\infty}$ of 
a sequence $\{f_i\}_{i<\infty}$ of $\lambda_k^{g_i}(X_i)$-eigenfunctions $f_i$ of $\Delta^{\upsilon_i}+g_i$ with $||f_i||_{L^2}=1$, then we see that $f_{\infty}$ is a $\lambda_{k}^{g_{\infty}}(X_{\infty})$-eigenfunction of $\Delta^{\upsilon_{\infty}}+g_{\infty}$, that $f_i$ $L^{2p/(p-2)}$-converges strongly to $f_{\infty}$ on $X_{\infty}$
and that $d f_i$ $L^2$-converges strongly to $df_{\infty}$ on $X_{\infty}$.
\end{enumerate}
\end{theorem}
Combining this with the $(2n/(n-2), 2)$-Poincar\'e inequality given by Maheux-Saloff-Coste in \cite{mas} yields the following.
\begin{corollary}\label{horr}
Assume $n \ge 3$.
Let $q>n/2$ and let $(X_i, \upsilon_i) \stackrel{GH}{\to} (X_{\infty}, \upsilon_{\infty})$ in $\overline{M(n, K, d)}$ with $\mathrm{diam}\,X_{\infty}>0$. 
Then for every $L^q$-weak convergent sequence $\{g_i\}_{i \le \infty}$ on $X_{\infty}$ of $g_i \in L^q(X_i)$ 
we have
\begin{align}\label{108}
\lim_{i \to \infty}\lambda_{k}^{g_i}(X_i)=\lambda_{k}^{g_{\infty}}(X_{\infty})
\end{align}
for every $k \ge 0$.
\end{corollary}
Note that (\ref{108}) for $g_i \equiv 0$ corresponds to the continuity of eigenvalues of the Laplacian proved by Cheeger-Colding in \cite{ch-co3}
and that if we consider 
\[g_i =\frac{n-2}{4(n-1)}s_{X_i} (=:\hat{s}_{X_i})\]
when $(X_i, \upsilon_i)$ is a smooth $n$-dimensional compact metric measure space for every $i<\infty$, then Corollary \ref{horr} gives the continuity of eigenvalues of the conformal Laplacian, where
$s_{X_i}$ is the scalar curvature of $X_i$.
\subsection{Yamabe type equations}
Assume $n \ge 3$.
Let us consider the following quantity, called \textit{the generalized ($n$-) Yamabe constant of $(X, \upsilon)$} associated with $g \in L^q(X)$:
\begin{align}\label{yama7}
Y^g(X)=\inf_f \int_X\left(|df|^2+g|f|^2\right)d\upsilon,
\end{align}
where $f$ runs over all Lipschitz functions on $X$ with $||f||_{L^{2n/(n-2)}}=1$ and $1 \le q \le \infty$.
This is introduced by Akutagawa-Carron-Mazzeo in \cite{acm2} for more general setting.

We first recall the original Yamabe constant $\hat{Y}^{s_M}(M)$ of a closed $n$-dimensional Riemannian manifold $M$.
It is defined by
\begin{align}\label{mjnh}
\hat{Y}^{\hat{s}_M}(M):=\inf_f\int_M\left(|df|^2+\hat{s}_M|f|^2\right)dH^n,
\end{align}
where $f$ runs over all Lipschitz functions on $M$ with $||f||_{L^{2n/(n-2)}}=1$.

Then it is known that every positively valued minimizer $f$ of (\ref{mjnh}) satisfies the following equation:
\[\Delta f+\hat{s}_Mf-\hat{Y}^{\hat{s}_M}(M)|f|^{(n+2)/(n-2)}=0.\] 
We now recall the Yamabe problem on $M$ which means finding a minimizer of (\ref{mjnh}).
The following is well-known:
\begin{enumerate}
\item \textit{Aubin's inequality}
\[\hat{Y}^{\hat{s}_M}(M)\le \hat{Y}^{\hat{s}_{\mathbf{S}^n}}(\mathbf{S}^n)\]
holds.
\item For every $\epsilon>0$ there exists $B>0$ such that $(M, H^n)$ satisfies the $(2n/(n-2), 2)$-Sobolev inequality on $M$ for $((\hat{Y}^{\hat{s}_{\mathbf{S}^n}}(\mathbf{S}^n))^{-1}+\epsilon, B)$.
\item If $\hat{Y}^{\hat{s}_M}(M)<\hat{Y}^{\hat{s}_{\mathbf{S}^n}}(\mathbf{S}^n)$, then there exists a smooth positively valued minimizer of (\ref{mjnh}).
\item If $\hat{Y}^{\hat{s}_M}(M)=\hat{Y}^{\hat{s}_{\mathbf{S}^n}}(\mathbf{S}^n)$, then $M$ is conformally equivalent to $\mathbf{S}^n$.
In particular we can also find a smooth positively valued minimizer of $(\ref{mjnh})$.
\end{enumerate}
See \cite{au0, au, schoe, sc-ya, tru, yamabe} for the details.
Note that from above we see that the Yamabe problem is solvable in the smooth case and that by definition if $(X, \upsilon)$ is a smooth $n$-dimensional compact metric measure space and 
$g=\hat{s}_{X}$,
then we have 
\[Y^{\hat{s}_X}(X)=\left(H^n(X)\right)^{-2/n}\hat{Y}^{\hat{s}_X}(X).\]

We now turn to the nonsmooth case.

By combining Akutagawa-Carron-Mazzeo's works in \cite{acm2} with Cheeger-Colding's works in \cite{ch-co3}, we see that if $q>n/2$, $(X, \upsilon)$ satisfies
the $(2n/(n-2), 2)$-Sobolev inequality on $X$ for some $(A, B)$ and an Aubin type strict inequality
\begin{align*}\label{3124}
Y^g(X)<A^{-1}
\end{align*}
holds,
then there exists a minimizer $f \in H^{1, 2}(X)$ of (\ref{yama7}).
Thus we can regard this as a generalization of $(3)$ above. 

We now introduce a main result on generalized Yamabe constants.
It means that roughly speaking, a uniform Sobolev inequality and a uniform Aubin type strict inequality yield the continuity of generalized Yamabe constants with respect to the Gromov-Hausdorff topology:
\begin{theorem}\label{yamayama}
Let $q>n/2$, let $(X_i, \upsilon_i) \stackrel{GH}{\to} (X_{\infty}, \upsilon_{\infty})$ in $\overline{M(n, K, d)}$ with $\mathrm{diam}\,X_{\infty}>0$, and let $\{g_i\}_{i \le \infty}$ be an $L^q$-weak convergent sequence on $X_{\infty}$ of $g_i \in L^q(X_i)$.
Assume that there exist $A, B>0$ such that for every $i <\infty$, $(X_i, \upsilon_i)$ satisfies the $(2n/(n-2), 2)$-Sobolev inequality on $X_i$ for $(A, B)$ and that
\[\limsup_{i \to \infty}Y^{g_i}(X_i)<A^{-1}.\]
Then we have
\[\lim_{i \to \infty}Y^{g_i}(X_i)=Y^{g_{\infty}}(X_{\infty}).\]
In particular we have
\[Y^{g_{\infty}}(X_{\infty})<A^{-1}.\]
Moreover if $f_{\infty}$ is the $L^{2}$-weak limit on $X_{\infty}$ of a sequence
$\{f_i\}_{i<\infty}$ of minimizers $f_i \in H^{1, 2}(X_i)$ of $Y^{g_i}(X_i)$ with $||f_i||_{L^{2n/(n-2)}}=1$, then we see that $f_i$ $L^{2n/(n-2)}$-converges strongly to $f_{\infty}$ on $X_{\infty}$, that $df_i$ $L^2$-converges strongly to $df_{\infty}$ on $X_{\infty}$ and that $f_{\infty}$ is also a minimizer of $Y^{g_{\infty}}(X_{\infty})$.
\end{theorem}
As a corollary, 
we have the following continuity of \textit{almost nonpositive} generalized Yamabe constants:
\begin{corollary}\label{yamabe}
There exists $\delta:=\delta(n, K, d)>0$ such that the following holds:
Let $(X_i, \upsilon_i) \stackrel{GH}{\to} (X_{\infty}, \upsilon_{\infty})$ in $\overline{M(n, K, d)}$ with $\mathrm{diam}\,X_{\infty}>0$, let $q>n/2$, and let $\{g_i\}_{i \le \infty}$ be an $L^q$-weak convergent sequence on $X_{\infty}$ of $g_i \in L^q(X_i)$ with
\[\limsup_{i \to \infty}Y^{g_i}(X_i)\le \delta.\]
Then
\[\lim_{i \to \infty}Y^{g_i}(X_i)=Y^{g_{\infty}}(X_{\infty}).\]
\end{corollary}
\subsection{Hodge Laplace equations}
Let us consider \textit{the Hodge Laplacian}:
\[\Delta_{H, k}=\delta d + d \delta\]
acting on differential $k$-forms.
We first recall the notion of $RCD(K, \infty)$-spaces.

The notion of $RCD(K, \infty)$-spaces is introduced by Ambrosio-Gigli-Savar\'e in \cite{ags0} based on the study of $CD(K, \infty)$ spaces by Lott-Villani in \cite{lv} and Sturm in \cite{sturm1, sturm2}.
Roughly speaking,  a metric measure space said to be an $RCD(K, \infty)$ space if the Sobolev space $H_{1, 2}$ for functions is a Hilbert space and the Ricci curvature bounded from below by $K$. 
They developed the study of $RCD(K, \infty)$-spaces in \cite{ags0, ags}.
In particular by the stability with respect to the Gromov-Hausdorff topology proven in \cite{ags0} we knew that $(X, \upsilon)$ is an $RCD(K(n-1), \infty)$ space.

Gigli established in \cite{gigli} a second-order differential calculus on $RCD(K, \infty)$-spaces.
Let us consider the following question:

$\\ $
\textbf{Question $2$.} Is the second-order differential calculus on $(X, \upsilon)$ given in \cite{ho0} compatible with Gigli's one?

$\\ $
In order to give more precise statements, we introduce Gigli's notation and terminology in \cite{gigli} for the case of $(X, \upsilon)$ only.

In a similar way of the construction of $T^*X$ we can define the tangent bundle $TX$.
More generally, for any $r, s \in \mathbf{Z}_{\ge 0}$, the tensor bundle 
\[T^r_sX:=\bigotimes_{i=1}^rTX \otimes \bigotimes_{i=r+1}^{r+s} T^*X\]
is well-defined.
We denote the set of $L^p$-tensor fields of type $(r, s)$ on a Borel subset $A$ of $X$ by $L^p(T^r_sA)$.
Let 
\[\mathrm{Test}F(X):=\{f \in \mathcal{D}^2(\Delta^{\upsilon}, X); \Delta^{\upsilon}f \in H^{1, 2}(X),  df \in L^{\infty}(T^*X)\}.\]
By using this Gigli defined 
\begin{itemize}
\item the Sobolev spaces $W^{1, 2}_C(TX)$,$W^{1, 2}_d(\bigwedge^kT^*X)$ for vector fields, differential $k$-forms on $X$, respectively,
\item the covariant derivative $\nabla^{\upsilon}V \in L^2(T^1_1X)$ of $V \in W^{1, 2}_C(TX)$, and
\item the differential $d^{\upsilon}\omega \in L^2(\bigwedge^{k+1}T^*X)$ of $\omega \in W^{1, 2}_d(\bigwedge^kT^*X)$.
\end{itemize}
Note that we used a slight modified version of Gigli's covariant derivative because it is simple to use the modified version for discussing the compatibility between \cite{gigli} and \cite{ho0} (for example our manner is based on \cite{sakai}).
The original Gigli's covariant derivative of $V \in W^{1, 2}_C(TX)$ is in $L^2(T^2_0X)$, however it is essentially same to our terminology via the identification:
\[TX \otimes TX \cong TX \otimes T^*X\]
\begin{align}\label{hondahonda}u \otimes v \mapsto v\otimes u^*,
\end{align}
where $u^* \in T^*X$ is the dual element of $u \in T_X$ by the Riemannian metric $g_X$.
See Remark \ref{gigliremark}.

In a similar way of Gigli's manner in \cite{gigli} we can define the Sobolev space $W^{1, 2}_C(T^r_sX)$ for tensor fields and the covariant derivative $\nabla^{\upsilon}T \in L^2(T^r_{s+1}X)$ for $T \in W^{1, 2}_C(T^r_sX)$.
In particular define $W^{1, 2}_C(\bigwedge^kT^*X):=W^{1, 2}_C(T^0_kX) \cap L^2(\bigwedge^kT^*X)$ via the canonical embedding:
\[\bigwedge^kT^*X \hookrightarrow T^0_kX.\]
See subsection $2.5.4$ for the details.

Let $W^{2, 2}(X)$ be the set of $f \in H^{1, 2}(X)$ with $\nabla f \in W^{1, 2}_C(TX)$ (note that this holds if and only if $df \in W^{1, 2}_C(T^*X)$ holds), where $\nabla f \in L^{2}(TX)$ is the dual section of $df \in L^{2}(T^*X)$. Put 
\[\mathrm{TestForm}_k(X):=\left\{\sum_{i=1}^Nf_{0, i}df_{1, i} \wedge \cdots \wedge df_{k, i}; N \in \mathbf{N}, f_{j, i} \in \mathrm{Test}F(X)\right\}.\]
Gigli proved  $\mathcal{D}^2(\Delta^{\upsilon}, X) \subset W^{2, 2}(X)$ and $\mathrm{TestForm}_k(X) \subset W^{1, 2}_d(\bigwedge^kT^*X)$. 
We define the Hessian $\mathrm{Hess}_f^{\upsilon} \in L^2(T^0_2X)$ of $f \in W^{2, 2}(X)$ in Gigli's sense by
\[\mathrm{Hess}_f^{\upsilon}:=\nabla^{\upsilon}df.\]

Let us consider the following question.

$\\ $
\textbf{Question $3$.} When does $\mathrm{Hess}_f^{\upsilon}=\mathrm{Hess}_f^{g_X}$ hold?

$\\ $
Note that $\mathrm{Hess}_f^{g_X}:=\nabla^{g_X}df$ is the Hessian of a weakly twice differentiable function $f$ on $X$ defined in \cite{ho0}.

The following is an answer to this question:
\begin{theorem}\label{33766}
We have the following.
\begin{enumerate}
\item For any open subset $U$ of $X$ and $f \in \mathcal{D}^2(\Delta^{\upsilon}, U)$, we see that $f$ is weakly twice differentiable on $U$ with respect to the canonical weakly second-order differential structure of $(X, \upsilon)$ in the sense of Theorem \ref{221}.
Moreover if $(X, \upsilon)$ is the noncollapsed Gromov-Hausdorff limit of a sequence in $M(n, K, d)$, then
\begin{align}\label{34221}
\Delta^{\upsilon}f=\Delta^{g_X}f,
\end{align}
where $\Delta^{g_X}f:=-\mathrm{tr}(\mathrm{Hess}_f^{g_X})$ and $\mathrm{tr}$ is the trace.
\item For every $f \in \mathcal{D}^2(\Delta^{\upsilon}, X)$ we have
\[\mathrm{Hess}_f^{\upsilon}=\mathrm{Hess}_f^{g_X}.\]
\item If $(X, \upsilon)$ is the noncollapsed Gromov-Hausdorff limit of a sequence in $M(n, K, d)$, then  we have $\mathcal{D}^2(\Delta^{\upsilon}, X)=H^{2, 2}(X)$, where $H^{2, 2}(X)$ is the closure of $\mathrm{Test}F(X)$ in $W^{2, 2}(X)$ with respect to the $W^{2, 2}$-norm
\begin{align*}
||f||_{W^{2, 2}}:=\left(||f||_{L^2}^2+||\nabla f||_{L^2}^2+||\mathrm{Hess}_f^{\upsilon}||_{L^2}^2\right)^{1/2}.
\end{align*}
\end{enumerate}
\end{theorem}
Note that we proved in \cite{holp} that (\ref{34221}) holds for \textit{most} functions, that (\ref{34221}) is new even if $(X, \upsilon) \in M(n, K, d)$ and that Gigli proved in \cite{gigli} that $H^{2, 2}(X)$ coincides with the closure of $\mathcal{D}^2(\Delta^{\upsilon}, X)$ in $W^{2, 2}(X)$.
Theorem \ref{33766} gives a positive answer to \textbf{Question $2$} for functions.

Next let us discuss differential $k$-forms on $X$.

Let $\omega$ be a differential $k$-form on $X$ which is differentiable at a.e. $x \in X$ with respect to the canonical weakly second-order differential structure.
Recall that in \cite{ho0} we defined the differential $d\omega$ and the covariant derivative $\nabla^{g_X}\omega$.

Let us consider the following question:

$\\ $
\textbf{Question $4$.} When do $\nabla \omega=\nabla^{\upsilon}\omega$ and $d\omega=d^{\upsilon}\omega$ hold?

$\\ $
From the divergence theorem and Theorem \ref{33766} we easily see that this question has a positive answer in the following two cases: 
\begin{itemize}
\item  $(X, \upsilon) \in M(n, K, d)$ and $\omega \in C^{\infty}(\bigwedge^kT^*X)$.
\item  $\omega=df$ for some $f \in \mathcal{D}^2(\Delta^{\upsilon}, X)$.   
\end{itemize}
We introduce a generalization of these:
\begin{theorem}\label{aarrtt}
Assume that there exist a convergent sequence $\{(X_i, \upsilon_i)\}_i$ in $\overline{M(n, K, d)}$ of $(X_i, \upsilon_i) \in M(n, K, d)$ to $(X, \upsilon)$ and an $L^2$-strong convergent sequence $\{\omega_i\}_i$ of $\omega_i \in C^{\infty}(\bigwedge^kTX_i)$ to $\omega$ on $X$ with
\begin{align}\label{gyy}
\sup_i\int_{X_i}|\nabla \omega_i|^2d\upsilon_i<\infty.
\end{align}
Then we see that $\omega$ is differentiable at a.e. $x \in X$, that
 $\omega \in W^{1, 2}_C(\bigwedge^kT^*X) \cap W^{1, 2}_d(\bigwedge^kT^*X)$, that $\nabla\omega=\nabla^{\upsilon}\omega$ and that  $d\omega=d^{\upsilon}\omega$.
\end{theorem}
Note that in Theorem \ref{aarrtt}, if $k=1$, then the assumption (\ref{gyy}) can be replaced by a weaker condition:
\begin{align}\label{mmune}
\sup_{i}\int_{X_i}\left( |d\omega_i|^2+|\delta \omega_i|^2\right)d\upsilon_i<\infty.
\end{align}
See Theorems \ref{mnj}, \ref{techni2}, \ref{198183} and Proposition \ref{bounds}.

In order to give a sufficient condition for satisfying the assumption of Theorem \ref{aarrtt},
we consider the following question:

$\\ $
\textbf{Question $5$.} When is there a smooth approximation to $\omega$ in the following sense?
\begin{itemize}
\item There exist a convergent sequence $\{(X_i, \upsilon_i)\}_i$ in $\overline{M(n, K, d)}$ of $(X_i, \upsilon_i) \in M(n, K, d)$ to $(X, \upsilon)$ and a sequence $\{\omega_i\}_i$ of $\omega_i \in C^{\infty}(\bigwedge^kT^*X_i)$ such that $\omega_i, d\omega_i, \delta \omega_i$ $L^2$-converge strongly to $\omega, d^{\upsilon}\omega, \delta^{\upsilon}\omega$ on $X$, respectively (we call $\{\omega_i\}_i$ \textit{a smooth $W^{1, 2}_H$-approximation to $\omega$ with respect to $\{(X_i, \upsilon_i)\}_i$}).
\end{itemize}
We now recall the definition of the codifferential $\delta^{\upsilon}$ as above defined by Gigli in \cite{gigli}.

Let $\mathcal{D}^2(\delta^{\upsilon}_k, X)$ be the set of $\omega \in L^2(\bigwedge^kT^*X)$ satisfying that there exists $\eta \in L^2(\bigwedge^{k-1}T^*X)$ such that 
\begin{align}\label{codifferential}
\int_X\langle \omega, d^{\upsilon}\alpha \rangle d\upsilon=\int_X\langle \eta, \alpha \rangle d\upsilon
\end{align}
for every $\alpha \in \mathrm{TestForm}_{k-1}(X)$.
Since $\eta$ is unique if it exists because $\mathrm{TestForm}_{k-1}(X)$ is dense in $L^2(\bigwedge^{k-1}T^*X)$, we denote it by $\delta^{\upsilon}_k\omega$ or by $\delta^{\upsilon}\omega$ for short.
Note that $\mathcal{D}^2(\delta^{\upsilon}_k, X)$ is a Hilbert space equipped with the norm
\[||\omega||_{\delta_k^{\upsilon}}:=\left( ||\omega_k||_{L^2}^2+||\delta_k^{\upsilon}\omega||_{L^2}^2\right)^{1/2}.\]

Let $W^{1, 2}_H(\bigwedge^kT^*X):=W^{1, 2}_d(\bigwedge^kT^*X) \cap \mathcal{D}^2(\delta^{\upsilon}_k, X)$.
It is a Hilbert space equipped with the norm 
\[||\omega||_{W^{1, 2}_H}=\left(||\omega||_{L^2}^2+||d^{\upsilon}\omega ||_{L^2}^2+||\delta^{\upsilon} \omega||_{L^2}^2 \right)^{1/2}.\]
Gigli proved that $\mathrm{Test}\mathrm{Form}_k(X) \subset W^{1, 2}_H(\bigwedge^kT^*X)$.  
Let $H^{1, 2}_H(\bigwedge^kT^*X)$ be the closure of $\mathrm{Test}\mathrm{Form}_k(X)$ in $W^{1, 2}_H(\bigwedge^kT^*X)$.

We now give an answer to \textbf{Question $5$}:
\begin{theorem}\label{app6}
Let $\{(X_i, \upsilon_i)\}_i$ be a convergent sequence in $\overline{M(n, K, d)}$ to $(X, \upsilon)$. 
Assume $\omega \in H^{1, 2}_H(T^*X)$.
Then there exist a subsequence $\{i(j)\}_j$ and a sequence $\{\omega_{i(j)}\}_j$ of $\omega_{i(j)} \in \mathrm{Test}\mathrm{Form}_1(X_{i(j)})$ such that $\omega_{i(j)}, d^{\upsilon_{i(j)}}\omega_{i(j)}, \delta^{\upsilon_{i(j)}} \omega_{i(j)}$ $L^2$-converge strongly to $\omega, d^{\upsilon}\omega, \delta^{\upsilon}\omega$ on $X$, respectively.
Moreover if $(X_{i}, \upsilon_{i}) \in M(n, K, d)$ for every $i<\infty$, then we can choose $\{\omega_{i(j)}\}_j$ as $C^{\infty}$-differential $1$-forms. Therefore there exists a smooth $W^{1, 2}_H$-approximation to $\omega$ with respect to $\{(X_{i(j)}, \upsilon_{i(j)})\}_{j}$.
In particular the assumption of Theorem \ref{aarrtt} for $\omega$ holds. 
\end{theorem}
By the density of $\mathrm{TestForm}_kX$ in $L^2(\bigwedge^kT^*X)$, since $H^{1, 2}_H(\bigwedge^kT^*X)$ is dense in $L^2(\bigwedge^kT^*X)$, we see that 
\textbf{Questions $4$} and $5$ have positive answers for \textit{most} $1$-forms.
Note that we can also establish an approximation for differential $k$-forms which is similar to Theorem \ref{app6} in some weak sense.
See Remark \ref{ggtttt}.
These give also positive answers to \textbf{Question $2$} for differential forms.

Next we discuss the Hodge Laplacian $\Delta_{H, k}^{\upsilon}$ on $X$ defined by Gigli in \cite{gigli}. 

Let $\mathcal{D}^2(\Delta_{H, k}^{\upsilon}, X)$ be the set of $\omega \in W^{1, 2}_H(\bigwedge^kT^*X)$ satisfying that there exists $\eta \in L^2(\bigwedge^kT^*X)$ such that
\begin{align}\label{mkoiop}\int_X\langle \eta, \alpha\rangle d\upsilon=\int_X\langle d^{\upsilon}\omega, d^{\upsilon}\alpha\rangle d\upsilon+\int_X\langle \delta^{\upsilon}\omega, \delta^{\upsilon}\alpha \rangle d\upsilon
\end{align}
for every $\alpha \in \mathrm{TestForm}_k(X)$.
Since $\eta$ is unique if it exists by the same reason above 
we denote $\eta$ by $\Delta_{H, k}^{\upsilon}\omega$ (we will use the notation $\Delta_{H, k}\omega:=\Delta_{H, k}^{\upsilon}\omega$ if $(X, \upsilon) \in M(n, K, d)$ and $\omega \in C^{\infty}(\bigwedge^kT^*X)$ for brevity).
Note that Gigli established in \cite{gigli} the Hodge theorem for harmonic forms by restricting the domain of $\Delta_{H, k}^{\upsilon}$ to $\mathcal{D}^2(\Delta_{H, k}^{\upsilon}, X) \cap H_H^{1, 2}(\bigwedge^kT^*X)$.

Let us consider the following question.

$\\ $
\textbf{Question $6$.} How do eigenvalues of the Hodge Laplacian behave with respect to the Gromov-Hausdorff topology?

$\\ $
Note that in the case of $0$-forms, i.e., functions, the continuity of eigenvalues of the Laplacian acting on functions proved by Cheeger-Colding in \cite{ch-co3} (i.e., Theorem \ref{schro} in the case that $g_i \equiv 0$) gives the complete answer to this question.  

However it seems that in general it is difficult to study the behavior of eigenvalues of the Hodge Laplacian with respect to the Gromov-Hausdorff topology than that in the case of functions.
See for instance \cite{ac, cc90, cc00, lott, lott2, pere, taka}.
In particular Perelman showed in \cite{pere} that there exist a sequence $\{(X_i, \upsilon_i)\}_{i<\infty}$ in $M(4, 0, 1)$ and the noncollapsed Gromov-Hausdorff limit $(X_{\infty}, \upsilon_{\infty})$ of them such that 
\[b_2(X_i) \to \infty,\] 
where $b_2(X_i)$ is the second Betti number of $X_i$.
This shows 
that even for a noncollapsed sequence, it is not easy to control the behavior of harmonic $2$-forms with respect to the Gromov-Hausdorff topology under lower Ricci curvature bounds only. 
Note that for the first Betti number $b_1$, by Gallot and Gromov in \cite{gallot, gromov}, it is known a uniform bound:
\[b_1(X) \le C(n, K, d)\]
for every $(X, \upsilon) \in M(n, K, d)$.

We can give an answer to this question.
It means that for $1$-forms, the $L^2$-strong limit of a sequence of eigenforms is also an eigenform in some weak sense. 
\begin{theorem}\label{eigenfcont}
Let $\{\lambda_i\}_{i<\infty}$ be a bounded sequence in $\mathbf{R}$, let $\{(X_i, \upsilon_i)\}_{i < \infty}$ be a sequence in $M(n, K, d)$, let $(X_{\infty}, \upsilon_{\infty})$ be the Gromov-Hausdorff limit of them with $\mathrm{diam}\,X_{\infty}>0$, let $\{\omega_i\}_{i <\infty}$ be a sequence of $\lambda_i$-eigenforms $\omega_i \in C^{\infty}(T^*X_i)$ with $||\omega_i||_{L^2(X_i)}=1$, i.e., 
\begin{align*}
\Delta_{H, 1}^{\upsilon_i}\omega_i=\lambda_i\omega_i
\end{align*}
and let $\omega_{\infty}$ be the $L^2$-strong limit on $X_{\infty}$ of them.
Then we see that the limit
\[\lim_{i \to \infty}\lambda_i\]
exists and that
 $\omega_{\infty} \in \mathcal{D}^2(\Delta^{\upsilon_{\infty}}_{H, 1}, X_{\infty})$ with 
\[\Delta_{H, 1}^{\upsilon_{\infty}}\omega_{\infty}=\left(\lim_{i \to \infty}\lambda_i\right)\omega_{\infty}.\]
\end{theorem}
Note that in Theorem \ref{eigenfcont}, by using a mean value inequality by Li-Tam given in \cite{LT} we will prove a uniform $L^{\infty}$-estimate for $\omega_i$:
\[\sup_{i}||\omega_i||_{L^{\infty}}<\infty.\]
In particular we see that $\omega_{\infty} \in L^{\infty}(\bigwedge^kT^*X_{\infty})$ and that $\omega_{\infty}$ is the $L^p$-strong limit of $\{\omega_i\}_{i <\infty}$ for every $p \in (1, \infty)$.
See Proposition \ref{aa33}.
We will also prove a similar continuity of eigenvalues of the Hodge Laplacian for differential $k$-forms with an additional assumption.
See Theorem \ref{hodgelaplacian} and Corollary \ref{harmonicha}.

By Theorems \ref{aarrtt} and \ref{eigenfcont} we can easily see that the following final question is important:

$\\ $
\textbf{Question $7$.} Is there a nice compactness for the $L^2$-strong convergence?

$\\ $
Note that in \cite{KS, KS2, holp} we knew the $L^2$-weak compactness with respect to the Gromov-Hausdorff topology which means that every $L^2$-bounded sequence has an $L^2$-weak convergent subsequence (this also holds in the $L^p$-case for every $1<p<\infty$).

The following is an answer to this question:
\begin{theorem}\label{3ew3}
Let $\{(X_i, \upsilon_i)\}_{i<\infty}$ be a sequence in $M(n, K, d)$, let $(X_{\infty}, \upsilon_{\infty})$ be the noncollapsed Gromov-Hausdorff limit of them and let $\{\omega_i\}_{i<\infty}$ be a sequence of $\omega_i \in C^{\infty}(\bigwedge^kT^*X_i)$ with $||\omega_i||_{L^2(X_i)}=1$ and (\ref{gyy}).
Then there exist a subsequence $\{i(j)\}_j$ and $\omega_{\infty} \in L^2(\bigwedge^kT^*X_{\infty})$ such that $\omega_{i(j)}$ $L^2$-converges strongly to $\omega_{\infty}$ on $X_{\infty}$.
\end{theorem}
In Theorem \ref{3ew3} if we consider the case of $k=0$, i.e., functions, then we have the same conclusion without the noncollapsed assumption.
This is a Rellich type compactness in the Gromov-Hausdorff setting proven in \cite{holp, KS}.
See Theorem \ref{srell}.
However, in the case of differential forms, the noncollapsed assumption is essential. See Remark \ref{harharhar}.
\subsection{Organization of the paper}
In Section $2$ we introduce our notation and terminology. 

In Section $3$ we give several new results on the $L^p$-convergence. 
In particular we discuss the stabilities of Poincar\'e and Sobolev inequalities with respect to the Gromov-Hausdorff topology.
We also give generalizations of Fatou's lemma and Sobolev embeddings to the Gromov-Hausdorff setting.
They play crucial roles to prove Theorem \ref{yamayama}.

In Section $4$ we discuss Poisson's equations. In particular we prove the results stated in subsection $1.1$.
As applications we also prove Theorems \ref{33766} and \ref{app6}.

In Section $5$ we study Schr$\ddot{\text{o}}$dinger operators and generalized Yamabe constants.
We also give a generalization of Theorem \ref{yamayama}. 
See Theorem \ref{contya}.

In Section $6$ we establish a Rellich type compactness for tensor fields with respect to the Gromov-Hausdorff topology which is a generalization of Theorem \ref{3ew3}.
See Corollary \ref{nnbbmm}.
This plays a crucial role to prove results stated in subsection $1.4$ and gives a positive answer to \textbf{Question $7$} for tensor fields.
As an application we show that the noncollapsed Gromov-Hausdorff limit of a sequence of compact K$\ddot{\text{a}}$hler manifolds is also K$\ddot{\text{a}}$hler in some weak sense.
See Theorem \ref{uuhhff}.

Section $7$ is mainly devoted to the proofs of the remained results stated in subsection $1.4$.
We also establish new Bochner inequalities.
See Theorem \ref{boch} and Corollary \ref{expb}.  

It is worth pointing out that we will introduce a new `test class'
\begin{align}\label{newtest}
\widetilde{\mathrm{Test}}F(X):=\{f \in \mathcal{D}^2(\Delta^{\upsilon}, X); \Delta^{\upsilon}f\,\mathrm{is}\,\mathrm{a}\,\mathrm{Lipschitz}\,\mathrm{function}\,\mathrm{on}\,X\}.
\end{align}
Note that Theorem \ref{pois} yields that $\widetilde{\mathrm{Test}}F(X)$ is a linear subspace of $\mathrm{Test}F(X)$.

By using this test class instead of $\mathrm{Test}F(X)$, we will define new Sobolev spaces, $\widetilde{H}^{2, 2}(X), \widetilde{W}^{1, 2}_C(TX)$, and so on in the same manner of Gigli \cite{gigli}.
By using Theorem \ref{pois} we can check that these new Sobolev spaces (i.e., `$\sim$-versions') behave nicely with respect to the Gromov-Hausdorff topology.

On the other hand, by definition,  we can easily check the following trivial relationship between them and Gigli's one:
\begin{align}\label{HW1}
\widetilde{H} \subset H \subset W \subset \widetilde{W}.
\end{align}
A key result in proofs of theorems stated in subsection $1.4$ is to establish that the original Gigli's Sobolev spaces and the $\sim$-versions are coincide, i.e.
\[\widetilde{H} = H, W = \widetilde{W}.\]
In particular Gigli's original Sobolev spaces behave nicely with respect to the Gromov-Hausdorff topology.
See Theorems \ref{techni2} and \ref{198183}.

\textbf{Acknowledgments.}
The author is grateful to Kazuo Akutagawa for his suggestions on the Yamabe problem and giving him a long informal seminar on the topic at Tokyo Institute of Technology.
He would like to express his appreciation to Shin-ichi Ohta and Takao Yamaguchi for giving him a long informal seminar on the paper at Kyoto University and giving valuable suggestions on the paper.
He wishes to thank the all participants in these informal seminars.
He would like to express his appreciation to Akito Futaki for giving valuable comments on the preliminary version of the paper.
A revision of the paper was written 
during the stay of 
the Junior Hausdorff Trimester Program
on ``Optimal Transportation” in Hausdorff Research Institute for Mathematics.
He would like to thank HIM for warm hospitality.  
He is also grateful to the referee for careful reading, pointing out errors in subsection 3.3 of the previous version, and valuable suggestions which are greatly improved the overall mathematical quality of the paper.
This paper is dedicated to the memory of Kentaro Nagao.
I would like to express my deep respect to him for his kindness of heart, his strength of mind and his enthusiasm to mathematics during his short life.
This work was supported  by Grant-in-Aid for Young Scientists (B) $24740046$.
\section{Preliminaries}
In this section we fix our notation and terminology on metric measure geometry. We also recall several results on Ricci limit spaces.
\subsection{Metric measure spaces}
Let $X$ be a metric space. 
We say that \textit{$X$ is proper} if every bounded closed subset of $X$ is compact. 
We also say that \textit{$X$ is a geodesic space} if for any $p, q \in X$ there exists an isometric embedding $\gamma : [0, d_X(p, q)] \to X$ such that $\gamma (0)=p$ and $\gamma (d_X(p, q))=q$ hold (we call $\gamma$ \textit{a minimal geodesic from $p$ to $q$}).

We denote the space of Lipschitz functions on $X$ by $\mathrm{LIP}(X)$, the space of locally Lipschitz functions on $X$ by $\mathrm{LIP}_{\mathrm{loc}}(X)$, and the space of $f \in \mathrm{LIP}(X)$ with compact support by $\mathrm{LIP}_c(X)$.
For every $f \in \mathrm{LIP}(X)$, let us denote the Lipschitz constant of $f$ by $\mathbf{Lip}f$, i.e., 
\[\mathbf{Lip}f:= \sup_{a \neq b}\frac{|f(a)-f(b)|}{d_X(a, b)}.\]
We sometimes denote the local Lipschitz constant $\mathrm{Lip}f(x)$ of $f$ by $|\nabla f|(x)$ (see (\ref{55m})).

Let $\upsilon$ be a Borel measure on $X$.
In this paper we say that a pair $(X, \upsilon)$ is a \textit{metric measure space} if $0<\upsilon (B_r(x))<\infty$ for any $x \in X$ and $r>0$.
For two metric measure spaces $\{(Z_i, \nu_i)\}_{i=1, 2}$, we say that \textit{$(Z_1, \nu_1)$ is isometric to $(Z_2, \nu_2)$} if there exists an
isometry $\phi: Z_1 \to Z_2$ (as metric spaces) such that $\nu_1(A_1)=\nu_2(\phi(A_1))$ for every Borel subset $A_1$ of $Z_1$. 

Let $A$ be a Borel subset of $X$, let $Y$ be a metric space and let $G$ be a Borel map from $A$ to $Y$.
We say that  $G$ is \textit{differentiable at a.e. $a \in A$} if there exists a countable collection $\{A_i\}_{i}$ of Borel subsets $A_i$ of $A$ such that $\upsilon(A \setminus \bigcup_iA_i)=0$ and that $G|_{A_i}$ is Lipschitz for every $i$.
It is important that if $X, Y$ are Riemannian manifolds and $A$ is an open subset of $X$, then this notion coincides with that in the ordinary sense. 
\subsection{Gromov-Hausdorff convergence}
\subsubsection{Compact case}
Let $\{X_i\}_{i \le \infty}$ be a sequence of compact metric spaces and let $\{\upsilon_i\}_i$ be a sequence of Borel probability measures $\upsilon_i$ on $X_i$.
We say that $(X_i, \upsilon_i)$ Gromov-Hausdorff converges to $(X_{\infty}, \upsilon_{\infty})$ if there exist a sequence $\{\phi_i\}_i$ of Borel maps $\phi_i: X_i \to X_{\infty}$ and a sequence $\{\epsilon_i\}_i$ of $\epsilon_i >0$ such that the following four conditions hold:
\begin{itemize}
\item $\lim_{i \to \infty}\epsilon_i=0$.
\item $\left| d_{X_i}(x, y)-d_{X_{\infty}}(\phi_i(x), \phi_i(y))\right|<\epsilon_i$ for any $i$ and $x, y \in X_i$.
\item $X_{\infty}=B_{\epsilon_i}(\phi_i(X_i))$, where $B_{\epsilon}(A)$ is the $\epsilon$-open neighborhood of $A$.
\item We have
\[\lim_{i \to \infty}\upsilon_i(B_r(z_i))=\upsilon_{\infty}(B_r(z_{\infty}))\]
for any sequence $\{z_{i}\}_{i \le \infty}$ of $z_i \in X_i$ with $\phi_i(z_i) \to z_{\infty}$ in $X_{\infty}$ (we denote it $z_i \stackrel{GH}{\to} z_{\infty}$ for short and we call it a convergent sequence) and $r>0$.
\end{itemize} 
Then we denote it $(X_i, \upsilon_i) \stackrel{GH}{\to} (X_{\infty}, \upsilon_{\infty})$ for short.
\subsubsection{General case}
Let $\{X_i\}_{i \le \infty}$ be a sequence of proper metric spaces, let $\{x_i\}_{i \le \infty}$ be a sequence of $x_i \in X_i$ and let $\{\upsilon_i\}_{i \le \infty}$ be a sequence of Borel measures $\upsilon_i$ on $X_i$.
We say that \textit{$(X_i, x_i, \upsilon_i)$ Gromov-Hausdorff converges to $(X_{\infty}, x_{\infty}, \upsilon_{\infty})$} if there exist 
 sequences $\{R_i, \epsilon_i\}_i$ of $R_i, \epsilon_i >0$ and a sequence $\{\phi_i\}_i$ of Borel maps $\phi_i: B_{R_i}(x_i) \to X_{\infty}$ such that the following five conditions hold:
\begin{itemize} 
\item $\lim_{i \to \infty}\epsilon_i=0$ and $\lim_{i \to \infty}R_i=\infty$.
\item $\left| d_{X_i}(x, y)-d_{X_{\infty}}(\phi_i(x), \phi_i(y))\right|<\epsilon_i$ for any $i$ and $x, y \in B_{R_i}(x_i)$.
\item $B_{R_{i}}(x_{\infty})\subset B_{\epsilon_i}(\phi_i(B_{R_i}(x_i)))$ holds for every $i<\infty$.
\item $\phi_i(x_i) \to x_{\infty}$ in $X_{\infty}$. We also denote it by $x_i \stackrel{GH}{\to} x_{\infty}$ and we call it a convergent sequence.
\item We have
\[\lim_{i \to \infty}\upsilon_i(B_r(z_i))=\upsilon_{\infty}(B_r(z_{\infty}))\]
for any convergent sequence $\{z_{i}\}_{i \le \infty}$ of $z_i \in B_{R_i}(x_i)$ and $r>0$.
\end{itemize}
Then we also denote it by $(X_i, x_i, \upsilon_i) \stackrel{GH}{\to} (X_{\infty}, x_{\infty}, \upsilon_{\infty})$ for short.

Assume $(X_i, x_i, \upsilon_i) \stackrel{GH}{\to} (X_{\infty}, x_{\infty}, \upsilon_{\infty})$.
Let $\{C_i\}_{i\le \infty}$ be a sequence of subsets $C_i$ of $X_i$ satisfying that there exists $L>0$ such that $C_i \subset B_L(x_i)$ for every $i \le \infty$.
Then we denote $\limsup_{i \to \infty}C_i \subset C_{\infty}$ if for every $\epsilon>0$ there exists $i_0$ such that $\phi_i(C_i)\subset B_{\epsilon}(C_{\infty})$ for every $i \ge i_0$.
We also denote $C_{\infty} \subset \liminf_{i \to \infty}C_i$ if for every $\epsilon>0$ there exists $i_0$ such that $C_{\infty}\subset B_{\epsilon}(\phi_i(C_i))$ for every $i \ge i_0$.
We say that \textit{a subset $K$ of $X_{\infty}$ is a (Gromov-Hausdorff) limit of $\{C_i\}_{i<\infty}$ (with respect to the convergence $(X_i, x_i, \upsilon_i) \stackrel{GH}{\to} (X_{\infty}, x_{\infty}, \upsilon_{\infty})$)} if $\limsup_{i \to \infty}C_i \subset K$ and $K \subset \liminf_{i \to \infty}C_i$.
It is easy to check the following:
\begin{itemize}
\item If $K_1$ and $K_2$ are limits of $\{C_i\}_{i<\infty}$, then $\overline{K_1}=\overline{K_2}$, where $\overline{K_i}$ is the closure of $K_i$ in $X_{\infty}$. 
\item There exists a subsequence $\{i(j)\}_j$ such that a limit of $\{C_{i(j)}\}_j$ exists.
\end{itemize} 
From above we see that the compact limit of $\{C_i\}_i$ is unique if it exists.
Thus we denote it by $\lim_{i \to \infty}C_i$.

Assume that $C_{\infty}$ is a limit of $\{C_i\}_i$.
Let $Y$ be a metric space.
For a sequence $\{f_i\}_{i \le \infty}$ of continuous maps $f_i: C_i \to Y$, we say that \textit{$f_i$ converges uniformly to $f_{\infty}$ on $C_{\infty}$}
if for every $\epsilon>0$ there exist $i_0$ and $\delta>0$ such that 
$d_Y(f_{\infty}(y), f_i(x))<\epsilon$ for any $i \ge i_0$, $x \in C_i$ and $y \in C_{\infty}$ with $d_{X_{\infty}}(\phi_i(x), y)<\delta$.
See also subsection $2.2$ in \cite{holip}.

Let us denote by $M(n, K)$ the set of (isometry classes of)  pointed proper metric measure spaces $(X, x, H^n/H^n(B_1(x)))$, where $X$ is an $n$-dimensional complete Riemannian manifold with $\mathrm{Ric}_{X}\ge K(n-1)$.
We denote by $\overline{M(n, K)}$ the set of Gromov-Hausdorff limits of sequences in $M(n, K)$.
Note that $\overline{M(n, K)}$ is compact with respect to the Gromov-Hausdorff topology. See \cite{ch-co1, fu, gr}.

In this paper we also call a pointed metric measure space which belongs to $\overline{M(n, K)}$ \textit{a Ricci limit space} for short.
\subsection{Rectifiable metric measure spaces and weakly second-order differential structure}
\subsubsection{Euclidean spaces}
Let $A$ be a Borel subset of $\mathbf{R}^k$, let $F=(f_1, \ldots, f_m)$ be a Lipschitz map from $A$ to $\mathbf{R}^m$ and let $y \in \mathrm{Leb}\,A:=\{ a \in A; \lim_{r \to 0}H^k(A \cap B_r(a))/H^k(B_r(a))=1\}$.
Then we say that \textit{$F$ is differentiable at $y$} if there exists a Lipschitz map $\hat{F}$ from $\mathbf{R}^k$ to $\mathbf{R}^m$ such that $\hat{F}|_A\equiv F$ and that $\hat{F}$ is  differentiable at $y$.
Note that if $F$ is differentiable at $y$, then the Jacobi matrix 
\[J(\hat{F})(y)\]
of $\hat{F}$ does not depend on the choice of such $\hat{F}$.
Thus we denote it by 
\[J(F)(y)=(\partial f_i/\partial x_j(y))_{ij}.\]
Note that by Rademacher's theorem \cite{rad} we see that $F$ is differentiable at a.e. $x \in A$.
This is compatible with the similar notion introduced in subsection $2.1$.

Let $G$ be a Borel map from $A$ to $\mathbf{R}^m$.
Assume that $G$ is differentiable at a.e. $x \in A$.
Note that from above we easily see that $J(G)(x)$ is well-defined for a.e. $x \in A$.
We say that $G$ is \textit{weakly twice differentiable on $A$} if $J(G)$ is differentiable at a.e. $a \in A$. 
It is important that if $A$ is open and $G$ is a $C^{1, 1}$-map, then $G$ is weakly twice differentiable on $A$.

Let $T=\sum_{\lambda \in \Lambda} T_{\lambda}\bigotimes_{i=1}^r\nabla x_{\lambda(i)} \otimes \bigotimes_{i=r+1}^{r+s} dx_{\lambda(i)}$ be a tensor field of type $(r, s)$ on $A$, where $\Lambda$ is the set of maps from $\{1, \ldots, r+s\}$ to $\{1, \ldots, k\}$.
We say that $T$ is a \textit{Borel tensor field on $A$} if  $T_{\lambda}$ is a Borel function for every $\lambda$.
We also say that $T$ is \textit{differentiable at a.e. $a \in A$} if $T_{\lambda}$ is differentiable at a.e.  $a \in A$ for every $\lambda$.

For two Borel tensor fields $\{T_i\}_{i=1,2}$ of type $(r, s)$ on $A$, we say that \textit{$T_1$ is equivalent to $T_2$ on $A$} if $T_1(a)=T_2(a)$ holds for a.e. $a \in A$.
Let us denote by $[T]$ the equivalent class of $T$, denote by $\Gamma_{0}(T^r_sA)$ the set of equivalent classes of Borel tensor fields of type $(r, s)$, and denote
by $\Gamma_{1}(T^r_sA)$ the set of $[T] \in \Gamma_0(T^r_sA)$ represented by a Borel tensor field $T$ of type $(r, s)$ which is differentiable at a.e. $a \in A$.
We often write $T=[T]$ for brevity. 
See subsection $3.1$ in \cite{ho0} for the details of this subsection.
\begin{remark}
Throughout this paper we make no distinction between two objects which are coincide at a.e. for simplicity.
For instance if a function $f$ on a metric measure space $(X, \upsilon)$ satisfies that there exists a Borel subset $A$ of $X$ such that $\upsilon(X \setminus A)=0$ and that $f|_A$ is Lipschitz, then we also say that $f$ is Lipschitz on $X$.
\end{remark}
\subsubsection{Rectifiable metric measure spaces.}
Let $(X, \upsilon)$ be a metric measure space.
We say that \textit{$(X, \upsilon)$ is rectifiable} if there exist $m \in \mathbf{N}$,  a collection $\{C_{i}^l\}_{1 \le l \le m, i \in \mathbf{N}}$ of Borel subsets $C_i^l$ of $X$, and a collection $\{ \phi_{i}^l\}_{l, i}$ of bi-Lipschitz embedding maps $\phi_{i}^l: C_{i}^l \hookrightarrow \mathbf{R}^l$ such that the following three conditions hold:
\begin{itemize}
\item $\upsilon(X \setminus \bigcup_{l,i}C_{i}^l)=0$.
\item For any $i, l$ and $x \in C_{i}^l$ we have
\[0< \liminf_{t \to 0}\left(\frac{\upsilon (B_t(x))}{t^l}\right) \le \limsup_{t \to 0}\left(\frac{\upsilon (B_t(x))}{t^l}\right)<\infty.\] 
\item For any $l$, $x \in \bigcup_{i \in \mathbf{N}}C_{i}^l$ and $0 < \delta < 1$ there exists $i \in \mathbf{N}$ such that $x \in C_{i}^l$ and
\[\max \left\{\mathbf{Lip}\phi_i^l, \mathbf{Lip}(\phi_i^l)^{-1}\right\} \le 1+\delta.\]
\end{itemize}
See \cite[Definition $5.3$]{ch-co3} and the condition $iii)$ of page $60$ in \cite{ch-co3} (see also \cite{fed}).

We say that a family $\mathcal{A}:=\{(C_{i}^l, \phi_{i}^l)\}_{l, i}$ as above is  a \textit{rectifiable coordinate system of $(X, \upsilon)$} and that each $(C_i^l, \phi_i^l)$ is an \textit{$l$-dimensional rectifiable coordinate patch of $\mathcal{A}$}.

Assume that $(X, \upsilon)$ is rectifiable.
We introduce several fundamental properties of rectifiable metric measure spaces which include a generalization of Rademacher's theorem to such spaces:
\begin{theorem}\cite{ch1, ch-co3}\label{29292}
There exist a topological space $T^*X$ and a Borel map  $\pi:T^*X \rightarrow X$ such that the following hold.
\begin{enumerate}
\item $\upsilon(X \setminus \pi (T^*X))=0$.
\item For every $w \in \pi(T^*X)$, $(\pi)^{-1}(w) (=:T^*_wX)$ is a finite dimensional real Hilbert space. We denote the inner product by $\langle \cdot, \cdot \rangle_w$. Let $|v|:= \sqrt{\langle v, v\rangle_w}$ for every $v \in T^*_wX$.
\item For every $f \in \mathrm{LIP}_{\mathrm{loc}}(X)$, there exist a Borel subset $V$ of $X$, and a Borel map $df: V \to T^*X$ such that $\upsilon(X \setminus V)=0$, that $\pi \circ df\equiv id_V$ and that $|df|(w)=\mathrm{Lip}f(w)$ for every $w \in V$.
\end{enumerate} 
\end{theorem}
See Section $6$ in \cite{ch-co3} or page $458-459$ of \cite{ch1} (or subsection $2.3$ in \cite{holp}) for the details.

Let $A$ be a Borel subset of $X$.
As an important corollary of Theorem \ref{29292} we see that if a function $f$ on $A$ which is differentiable at a.e. $a \in A$, then $df(a) \in T^*_aX$ is well-defined for a.e. $a \in A$.
Let us denote by $\Gamma_0(A)$  the set of Borel functions on $A$ and denote by $\Gamma_1(A)$ the set of $f \in \Gamma_0(A)$ which is differentiable at a.e. $a \in A$.

Moreover for any $r, s \in \mathbf{Z}_{\ge 0}$ we can define the ($L^{\infty}$-)vector bundle
\[\pi^r_s: \bigotimes_{i=1}^r TX \otimes \bigotimes_{j=1}^sT^*X \to X.\]
For convenience we use the following notation: 
\[T^r_sA :=(\pi^r_s)^{-1}(A).\]
We call a Borel measurable section $T: A \to T^r_sA$ \textit{a Borel tensor field of type $(r, s)$ on $A$}. 
Let $\Gamma_{0}(T^r_sA)$ be the space of Borel tensor fields of type $(r, s)$ on $A$.
It is important that for every $T \in \Gamma_{0}(T^r_sA)$, each restriction $T|_{C_i^l \cap A}$ of $T$ to $C_i^l \cap A$ can be regarded as in $\Gamma_{0}(T^r_s\phi_i^l(C_i^l \cap A))$.

We also denote by $\langle \cdot, \cdot \rangle$ the canonical metric on each fiber of $T^r_sX$ which is defined by that of $T^*X$ for short. 
In particular we call the canonical metric on $TX$ the \textit{Riemannian metric of} $(X, \upsilon)$ and denote it sometimes by $g_{X}$.

For every $1 \le p \le \infty$, let 
\[L^p(T^r_sA):=\{T \in \Gamma_{0}(T^r_sA); |T| \in L^p(A)\}.\]
Note that $L^p(T^r_sA)$ is a Banach space equipped with the $L^p$-norm and that $g_{X} \in L^{\infty}(T^0_2X)$.
For any $V \in \Gamma_{0}(TA)$ and $f \in \Gamma_1(A)$, let $\nabla f:=(df)^* \in \Gamma_{0}(TA)$ and let $V(f):=\langle V, \nabla f\rangle \in \Gamma_0(A)$, where $^*$ is the canonical isometry $T^*_xX \cong T_xX$ defined by the Riemannian metric $g_X$.

Let $U$ be an open subset of $X$.
Let us denote by $\mathcal{D}^2(\mathrm{div}^{\upsilon}, U)$ the set of $T \in L^2(TU)$ satisfying that there exists a unique $h \in L^2(U)$ such that 
\[-\int_{U}fhd\upsilon=\int_{U}\langle \nabla f, T\rangle d\upsilon\]
holds for every $f \in \mathrm{LIP}_c(U)$.
Write $\mathrm{div}^{\upsilon}T:=h$.
We also denote the set of $\omega \in L^2(T^*U)$ satisfying  $\omega^* \in \mathcal{D}^2(\mathrm{div}^{\upsilon}, U)$ by $\mathcal{D}^2(\delta^{\upsilon}, U)$.
Then write $\delta^{\upsilon} \omega:=-\mathrm{div}^{\upsilon}\omega^*$.
See subsection $2.3$ in \cite{ho0} for the details of this subsection.
\subsubsection{Weakly second-order differential structure.}
Let $(X, \upsilon)$ be a rectifiable metric measure space and let $\mathcal{A}:=\{(C_i^l, \phi_i^l)\}_{i, l}$ be a rectifiable coordinate system of $(X, \upsilon)$.

We say that \textit{$\mathcal{A}$ is a weakly second-order differential system of $(X, \upsilon)$} if $\phi_i^l \circ (\phi_j^l)^{-1}$ is weakly twice differentiable on $\phi_j^l(C_i^l \cap C_j^l)$ for any $i, j$.

Assume that $\mathcal{A}$ is a weakly second-order differential system of $(X, \upsilon)$.

Let $A$ be a Borel subset of $X$.
We say that a Borel tensor field \textit{$T \in \Gamma_{0}(T^r_sA)$ is differentiable at a.e. $a \in A$ (with respect to $\mathcal{A}$)} if each $T|_{C_i^l \cap A} $ (recall that it can be regarded as in $\Gamma_{0}(T^r_s\phi_i^l(C_i^l \cap A)))$ is in $\Gamma_{1}(T^r_s\phi_i^l(C_i^l \cap A))$.
Let us denote by $\Gamma_1(T^r_sA; \mathcal{A})$ the set of $T \in \Gamma_0(T^r_sA)$ which is differentiable at a.e. $a \in A$.
We often write $\Gamma_1(T^r_sA):=\Gamma_1(T^r_sA; \mathcal{A})$ for brevity.

We say that a function $f \in \Gamma_0(A)$ is \textit{weakly twice differentiable on $A$ (with respect to $\mathcal{A}$)} if $f \in \Gamma_1(A)$ and $df \in \Gamma_1(T^*A)$.  
Let us denote by $\Gamma_2(A)=\Gamma_2(A; \mathcal{A})$ the set of weakly twice differentiable functions on $A$.
Note that for any $U, V \in \Gamma_1(TA)$, the Lie bracket $[U, V] \in \Gamma_{0}(TA)$ is well-defined by satisfying
\[[U, V]f=U(V(f))-V(U(f))\]
for every $f \in \Gamma_2(A)$.

Similarly we can define $\Gamma_1(\bigwedge^kT^*A)$ and the differential $d\omega \in \Gamma_0(\bigwedge^{k+1}T^*A)$ of $\omega \in \Gamma_1(\bigwedge^kT^*A)$.

Assume $g_X \in \Gamma_1(T^0_2X)$.  
We are now in a position to introduce a main result of \cite{ho0}.
\begin{theorem}\cite{ho0}\label{levi3}
There exists the Levi-Civita connection $\nabla^{g_X}$ on $X$ uniquely in the following sense:  
\begin{enumerate}
\item $\nabla^{g_X}$ is a map from $\Gamma_{0} (TX) \times \Gamma_{1} (TX)$ to $\Gamma_{0} (TX)$. Let $\nabla^{g_X}_UV:=\nabla^{g_X}(U, V)$.
\item $\nabla^{g_X}_U(V+W)=\nabla^{g_X}_UV + \nabla^{g_X}_U W$ for any $U \in \Gamma_{0} (TX)$ and  $V, W \in \Gamma_{1} (TX)$.
\item $\nabla^{g_X}_{fU+hV}W =f\nabla^{g_X}_UW + h\nabla^{g_X}_VW$ for any $U, V \in \Gamma_{0} (TX)$, $W \in \Gamma_{1} (TX)$ and $f, h \in \Gamma_0(X)$.
\item $\nabla^{g_X}_U(fV)=U(f)V + f\nabla^{g_X}_UV$ for any $U \in \Gamma_{0} (TX)$, $V \in \Gamma_{1} (TX)$ and $f \in \Gamma_1(X)$.
\item $\nabla^{g_X}_UV - \nabla^{g_X}_VU=[U, V]$ for any $U, V \in \Gamma_{1} (TX)$.
\item $Ug_X(V, W) = {g_X}( \nabla^{g_X}_UV, W) + g_X(V, \nabla^{g_X}_UW)$ for any $U \in \Gamma_{0} (TX)$ and  $V, W \in \Gamma_{1} (TX)$.   
\end{enumerate}
\end{theorem}
Note that $\nabla^{g_X}$ is local in the following sense: 
\begin{itemize}
\item The Levi-Civita connection induces the map $\nabla^{g_X}|_A:\Gamma_{0} (TA) \times \Gamma_{1} (TA) \to \Gamma_{0} (TA)$ by letting $\nabla^{g_X}|_A(U, V):=\nabla^{g_X}_{1_AU}(1_AV)$.
\end{itemize}
Thus we use the same notation: $\nabla^{g_X}=\nabla^{g_X}|_A$ for brevity.

Next we introduce several key notions in this paper.
\begin{proposition}\cite{ho0}\label{hess2}
Let $f \in \Gamma_2(A)$, let $\omega \in \Gamma_1(T^*A)$ and let $W \in \Gamma_{1}(TA)$.
Then there exist uniquely
\begin{enumerate}
\item the covariant derivative $\nabla^{g_X} \omega \in \Gamma_{0}(T^0_2A)$ of $\omega$ such that $\nabla^{g_X} \omega(U, V)=g_{X}(\nabla_{V}^{g_X}\omega^*, U)$ for any $U, V \in \Gamma_{0}(TA)$;
\item the Hessian $\mathrm{Hess}^{g_X}_f :=\nabla^{g_X}df \in \Gamma_{0}(T^0_2A)$ of $f$; 
\item the divergence $\mathrm{div}^{g_X}\,W:= \mathrm{tr} (\nabla^{g_X}W^*) \in \Gamma_0(A)$ of $W$; 
\item the codifferential $\delta^{g_X}\omega := - \mathrm{div}^{g_X}\omega^* \in \Gamma_0(A)$ of $\omega$;
\item the Laplacian $\Delta^{g_X} f:=-\mathrm{div}^{g_X}\,(\nabla^{g_X}f)=\delta^{g_X}(df)=- \mathrm{tr} (\mathrm{Hess}_{f}^{g_X}) \in \Gamma_0(A)$ of $f$.
\end{enumerate}
Moreover we have the following:
\begin{itemize}
\item[(a)] $\mathrm{Hess}^{g_X}_f(x)$ is symmetric for a.e. $x \in A$.
\item[(b)] $\mathrm{div}^{g_X}\,(hW)=h\mathrm{div}^{g_X} W +{g_X}(\nabla h, W)$ for every $h \in \Gamma_1(A)$. 
\item[(c)] $\Delta^{g_X} (fh)=h\Delta^{g_X} f-2{g_X}(\nabla f, \nabla h)+f\Delta^{g_X} h$ for every $h \in \Gamma_2(A)$.
\end{itemize}
\end{proposition}
More generally we can also define the \textit{covariant derivative of tensor fields} $\nabla^{g_X}:\Gamma_1(T^r_sA) \to \Gamma_{0}(T^r_{s+1}A)$ in the ordinary way of Riemannian geometry (c.f. \cite{sakai}),
i.e., for every $T \in \Gamma_1(T^r_sA)$,  $\nabla^{g_X}T \in \Gamma_0(T^r_{s+1}A)$ is defined by satisfying that
\begin{align}\label{huuh}
&\left\langle \nabla^{g_X}T, \bigotimes_{i=1}^rV_i \otimes \bigotimes_{j=1}^{s+1}\omega_j\right\rangle \nonumber \\
&=\omega_{s+1}^*\left(\left\langle T, \bigotimes_{i=1}^rV_i \otimes \bigotimes_{j=1}^{s}\omega_j \right\rangle \right) \nonumber \\
&-\sum_{i=1}^r\left\langle T, V_1 \otimes \cdots \otimes V_{i-1} \otimes \nabla^{g_X}_{\omega_{s+1}^*}V_i \otimes V_{i+1} \otimes \cdots \otimes V_{r} \otimes \bigotimes_{j=1}^{s}\omega_j \right\rangle \nonumber \\
&-\sum_{j=1}^{s}\left\langle T, \bigotimes_{i=1}^rV_i \otimes \omega_1 \otimes \cdots \otimes \omega_{j-1} \otimes \left(\nabla^{g_X}_{\omega_{s+1}^*}\omega_j^*\right)^* \otimes \omega_{j+1} \otimes \cdots \otimes \omega_{s}\right\rangle
\end{align}
for any $V_i \in \Gamma_1(TA)$ and $\omega_j \in \Gamma_1(T^*A)$.
For any $T \in \Gamma_1(T^r_sA)$ and $V \in \Gamma_0(TA)$ we define $\nabla^{g_X}_VT \in \Gamma_0(T^r_sA)$ by satisfying that
\[\left\langle \nabla^{g_X}_VT, \bigotimes_{i=1}^rV_i \otimes \bigotimes_{j=1}^s\omega_j \right\rangle=\left\langle \nabla^{g_X}T, \bigotimes_{i=1}^rV_i \otimes \bigotimes_{j=1}^s\omega_j \otimes V^*\right\rangle\]
for any $V_i \in \Gamma_0(TA)$ and $\omega_j \in \Gamma_0(T^*A)$.
Then it is easy to check that
\[\nabla^{g_X}g_X\equiv 0\]
and that
\begin{align}\label{dna}
d\omega (V_0, \ldots, V_k)= \sum_{i=0}^k(-1)^i(\nabla_{V_i}^{g_X}\omega)(V_0, \ldots, V_{i-1}, V_{i+1}, \ldots, V_k)
\end{align}
for any $\omega \in \Gamma_1(\bigwedge ^kT^*A)$ and $V_i \in \Gamma_0(TA)$.

Let $\hat{\mathcal{A}}$ be a weakly second-order differential system of $(X, \upsilon)$.
It is trivial that if $\mathcal{A}$ and $\hat{\mathcal{A}}$ are compatible (see subsection $1.1$ for the definition), then the notions introduced here are compatible, i.e., for instance we see that 
$\Gamma_2(A; \mathcal{A})=\Gamma_2(A; \hat{\mathcal{A}})$, $\Gamma_1(T^r_sA; \mathcal{A})=\Gamma_1(T^r_sA;  \hat{\mathcal{A}})$ and so on.
\subsection{Sobolev and Poincar\'e inequalities}
Let $(X, \upsilon)$ be a metric measure space and let $p, q \in [1, \infty)$.
\begin{definition}\label{muneri}
Let $U$ be an open subset of $X$.
\begin{enumerate}
\item We say that \textit{$(X, \upsilon)$ satisfies the $(q, p)$-Sobolev inequality on $U$ for a pair $(A, B)$ of some $A, B \ge 0$}
if 
\begin{align}\label{sobos}
\left(\int_U|f|^{q}d\upsilon \right)^{p/q}\le A \int_U|\mathrm{Lip} f|^pd\upsilon+B\int_U|f|^pd\upsilon
\end{align}
holds for every $f \in \mathrm{LIP}_c(U)$. 
\item We say that \textit{$(X, \upsilon)$ satisfies the $(q, p)$-Poincar\'e inequality on  $U$ for some $\tau \ge 0$}
if 
\begin{align}\label{poinc}
\left(\frac{1}{\upsilon (B_r(x))}\int_{B_r(x)}\left|f-\frac{1}{\upsilon (B_r(x))}\int_{B_r(x)}fd\upsilon \right|^qd\upsilon\right)^{1/q}\le \tau r \left(\frac{1}{\upsilon (B_r(x))}\int_{B_r(x)}|\mathrm{Lip} f|^pd\upsilon \right)^{1/p}
\end{align}
holds for any $x \in U$, $r>0$ with $B_r(x) \subset U$, and $f \in \mathrm{LIP}_{\mathrm{loc}}(U)$. 
\end{enumerate}
\end{definition}
\begin{remark}\label{85}
In Definition \ref{muneri}
let us consider the case that $U=B_R(z)$ for some $R>0$ and $z \in X$.

By using a cut-off function, it is not difficult to check that if (\ref{poinc}) holds for any $f \in \mathrm{LIP}_c(B_R(z))$, $x \in B_R(z)$ and $r>0$ with $B_r(x) \subset B_R(z)$, then $(X, \upsilon)$ satisfies the $(q, p)$-Poincar\'e inequality on $B_R(z)$ for $\tau$.
In particular (\ref{poinc}) holds for any $f \in \mathrm{LIP}(X)$, $x \in B_R(z)$ and $r>0$ with $B_r(x) \subset B_R(z)$ if and only if $(X, \upsilon)$ satisfies the $(q, p)$-Poincar\'e inequality on $B_R(z)$ for $\tau$.
\end{remark}
\begin{remark}\label{akl}
Let $U$ be a bounded open subset of $X$, let $\tau, R>0$ and let $z \in X$.
Then it is not difficult to check the following.
\begin{enumerate}
\item If $q\le p$, then $(X, \upsilon)$ satisfies the $(q, p)$-Sobolev inequality on  $U$ for $(0, \upsilon(U)^{p/q-1})$.
\item If $(X, \upsilon)$ satisfies the $(q, p)$-Sobolev inequality on $U$ for some $(A, B)$, then for every $\hat{q} \le q$, $(X, \upsilon)$ satisfies the $(\hat{q}, p)$-Sobolev inequality on $U$ for $(\upsilon(U)^{(q-\hat{q})/(q\hat{q})}A, \upsilon(U)^{(q-\hat{q})/(q\hat{q})}B)$. 
\item If $(X, \upsilon)$ satisfies the $(q, p)$-Poincar\'e inequality on $B_R(z)$ for $\tau$, then $(X, \upsilon)$ satisfies the $(q, p)$-Sobolev inequality on $B_R(z)$ for $(2^{p-1}\tau^p R^p\upsilon (B_R(z))^{p/q-1}, 2^{p-1}\upsilon (B_R(z))^{p/q-1})$.
\end{enumerate}
\end{remark}
We introduce the following Poincar\'e inequality given by Maheux-Saloff-Coste in \cite[Th\'eor\`eme $1.1$]{mas} (c.f. \cite[Theorem $3.21$]{he2} and \cite[Theorem $5.1$]{HK1}):
\begin{theorem}\cite{mas}\label{sobo}
Let $R>0$ and let $(X, x, \upsilon) \in M(n, K)$.
Then for any $1\le p<n$ and $p \le q \le np/(n-p)$, $(X, \upsilon)$ satisfies the $(q, p)$-Poincar\'e inequality on $B_R(x)$ for $C_1e^{C_2(1+\sqrt{|K|}R)}$,
where $C_1:=C_1(p, q)>0$ and $C_2:=C_2(n)>0$.
\end{theorem}
\subsection{Ricci limit spaces}
\subsubsection{Rectifiability, weakly second-order differential structure and dimension}
Let $(X, x, \upsilon) \in \overline{M(n, K)}$ with $\mathrm{diam}\,X>0$.
We say that a pointed proper metric measure space \textit{$(Y, y, \nu)$ is a tangent cone of $(X, \upsilon)$ at a point $z \in X$} if there exists a sequence $\{r_i\}_i$ of $r_i>0$ with $r_i \to 0$ such that $(X, z, \upsilon/\upsilon (B_{r_i}(z)), r^{-1}_id_X) \stackrel{GH}{\to} (Y, y, \nu)$.
For every $1 \le k \le n$, let us denote by $\mathcal{R}_k$ the set of $z \in X$ satisfying that every tangent cone $(Y, y, \nu)$ of $(X, \upsilon)$ at $z$ is isometric to $(\mathbf{R}^k, 0_k, H^k/H^k(B_1(0_k)))$, where $0_k=(0, \ldots, 0) \in \mathbf{R}^k$. 
Let
\[\mathcal{R}:=\bigcup_{k=1}^n\mathcal{R}_k.\]
We first introduce several Cheeger-Colding's results:
\begin{theorem}\cite{ch-co1, ch-co2, ch-co3}\label{fundpr}
We have the following:
\begin{enumerate}
\item For any $1<p<\infty$ and open subset $U$ of $X$, the Sobolev space $H^{1, p}(U)$ is well-defined.
Moreover if $U=B_R(z)$ for some $z \in X$ and $R>0$, then $\mathrm{LIP}_{\mathrm{loc}}(B_R(z)) \cap H^{1, p}(B_R(z))$ is dense in $H^{1, p}(B_R(z))$.
\item $(X, \upsilon)$ is rectifiable. Moreover for every open subset $U$ of $X$ we see that $H^{1, p}(U) \subset \Gamma_1(U)$ 
 and that
\[||f||_{H_{1, p}}=\left(||f||_{L^p}^p+||df||_{L^p}^p\right)^{1/p}\]
for every $f \in H^{1, p}(U)$.
In particular for every $f \in H^{1, p}(U) \cap \mathrm{LIP}_{\mathrm{loc}}(U)$ we have $||f||_{H_{1, p}}=\left(||f||_{L^p}^p+||\mathrm{Lip}f||_{L^p}^p\right)^{1/p}$.
\item The bilinear form
\[\int_X\langle df, dg \rangle d\upsilon\]
on $H^{1, 2}(X)$ gives a canonical Dirichlet form on $L^2(X)$ (see \cite{fuku}).
\item $\upsilon(X \setminus \mathcal{R})=0$.
\item The following four conditions (called \textit{noncollapsed conditions}) are equivalent:
\begin{enumerate}
\item The Hausdorff dimension of $X$ is equal to $n$.
\item $\mathcal{R}_n \neq \emptyset$.
\item $\mathcal{R}_i =\emptyset$ for every $i<n$.
\item $\upsilon=H^n/H^n(B_1(x))$.
\end{enumerate}
\end{enumerate}
\end{theorem} 
Next we introduce a Colding-Naber's result:
\begin{theorem}\cite{co-na1}\label{4455}
There exists a unique $k$ such that 
\[\upsilon (\mathcal{R} \setminus \mathcal{R}_k)=0.\]
We call $k$ \textit{the dimension of $X$} and we denote it by $\mathrm{dim}\,X$.
\end{theorem}
By Theorem \ref{4455} with an argument similar to the proof of $(2)$ of Theorem \ref{fundpr} we have the following (see also \cite{holip}):
\begin{itemize}
\item For every convergent sequence $\{(X_i, x_i, \upsilon_i)\}$ of $(X_i, x_i, \upsilon_i) \in M(n, K)$ to $(X, x, \upsilon)$, there exist
\begin{itemize}
\item a subsequence $\{i(j)\}_j$,
\item a rectifiable system $\mathcal{A}=\{(C_l, \phi_l)\}_l$ satisfying that $(C_l, \phi_l)$ is $k$-dimensional for every $l$,
\item a collection $\{B_{r_l}(y_{l})\}_{l}$ of $B_{r_l}(y_{l}) \subset X$ with $C_l \subset B_{r_l}(y_l)$,
\item a collection $\{h_l\}_l$ of $C(n, K)$-Lipschitz harmonic maps $h_l: B_{r_l}(y_{l}) \to \mathbf{R}^k$ (which means that $h_{l, j}$ is harmonic for every $j$, where $h_l:=(h_{l, 1}, \ldots, h_{l, k})$) with 
\[h_l|_{C_l}=\phi_l.\]
\item sequences $\{y_{i(j), l}\}_{j<\infty, l}$ of $y_{i(j), l} \in X_{i(j)}$ with $y_{i(j), l} \stackrel{GH}{\to} y_l$ as $j \to \infty$ and
\item sequences $\{h_{i(j), l}\}_{j<\infty, l}$ of $C(n, K)$-Lipschitz harmonic maps $h_{i(j), l}: B_{r_l}(y_{i(j), l}) \to \mathbf{R}^k$ satisfying that $h_{i(j), l}$ converges uniformly to $h_{l}$ on $B_{r_l}(y_l)$ as $j \to \infty$.
\end{itemize}
\end{itemize}
A main result of \cite{ho0} is the following.
\begin{theorem}\cite{ho0}\label{2n}
For $\mathcal{A}$ as above, we see that $\mathcal{A}$ is a weakly second-order differential system of $(X, \upsilon)$.
We say that $\mathcal{A}$ is a \textit{weakly second-order differential system of $(X, \upsilon)$ associated with $\{(X_{i(j)}, x_{i(j)}, \upsilon_{i(j)})\}_j$} or more simply, a \textit{harmonic rectifiable system of $(X, \upsilon)$ associated with $\{(X_{i(j)}, x_{i(j)}, \upsilon_{i(j)})\}_j$}. 
\end{theorem}
\subsubsection{$L^p$-convergence}
Let $R>0$, let $\{p_i\}_{i \le \infty}$ be a convergent sequence in $(1, \infty)$ and let $(X_i, x_i, \upsilon_i) \stackrel{GH}{\to} (X_{\infty}, x_{\infty}, \upsilon_{\infty})$ in $\overline{M(n, K)}$ with $\mathrm{diam}\,X_{\infty}>0$. 
We first introduce a generalization of the notion of the $L^p$-convergence for functions with respect to the Gromov-Hausdorff topology defined by Kuwae-Shioya in \cite{KS, KS2}.
\begin{definition}\cite{KS, KS2, holp}\label{lpf}
Let $\{f_i\}_{i\le \infty}$ be a sequence of $f_i \in L^{p_i}(B_R(x_i))$. 
\begin{enumerate}
\item We say that \textit{$f_i$ $\{L^{p_i}\}_i$-converges weakly to $f_{\infty}$ on $B_R(x_{\infty})$} if the following two conditions hold:
\begin{enumerate} 
\item $\sup_i||f_i||_{L^{p_i}(B_R(x_i))}<\infty$.
\item For any convergent sequence $\{z_i\}_{i \le \infty}$ of $z_i \in B_R(x_{i})$ and $r>0$ with $B_r(z_{\infty}) \subset B_R(x_{\infty})$,  we have 
\[\lim_{i \to \infty}\int_{B_r(z_i)}f_i d\upsilon_i=\int_{B_r(z_{\infty})}f_{\infty} d\upsilon_{\infty}.\]
Moreover if $p_i \equiv p$, then we call this an \textit{$L^p$-weak convergence (with respect to the Gromov-Hausdorff topology)}.
\end{enumerate}
\item We say that \textit{$f_i$ $\{L^{p_i}\}_i$-converges strongly to $f_{\infty}$ on $B_R(x_{\infty})$} if the following two conditions hold:
\begin{enumerate}
\item $f_i$ $\{L^{p_i}\}_i$-converges weakly to $f_{\infty}$ on $B_R(x_{\infty})$.
\item $\limsup_{i \to \infty}||f_i||_{L^{p_i}(B_R(x_{i}))}\le ||f_{\infty}||_{L^{p_{\infty}}(B_R(x_{\infty}))}$.
\end{enumerate}
Moreover if $p_i \equiv p$, then we also call this an \textit{$L^p$-strong convergence (with respect to the Gromov-Hausdorff topology)}.
\end{enumerate}
\end{definition}
Note that if $p_i \equiv p$, then the notions above are equivalent to that defined by Kuwae-Shioya in \cite{KS, KS2}.
In particular if $(X_i, x_i, \upsilon_i) \equiv (X, x, \upsilon)$, then these coincide with that in the ordinary sense.
See \cite[Remark $3.77$]{holp}.
\begin{remark}\label{equivlp}
Let $\{f_i\}_{i\le \infty}$ be a sequence of $f_i \in C^0(B_R(x_i))$.
Assume that $\sup_{i \le \infty}||f_i||_{L^{\infty}}<\infty$ and that $\{f_i\}_{i<\infty}$ is asymptotically uniformly equicontinuous on $B_R(x_{\infty})$, i.e., for every $\epsilon>0$ there exist $i_0 \in \mathbf{N}$ and $\delta>0$ such that for any $i \ge i_0$ and $\alpha, \beta \in B_R(x_i)$ with $d_{X_i}(\alpha, \beta)<\delta$ we have $|f_i(\alpha)-f_i(\beta)|<\epsilon$ (see \cite[Definition $3.2$]{holp}).
Then the following three conditions are equivalent:
\begin{itemize}
\item $f_i$ converges uniformly to $f_{\infty}$ on $B_R(x_{\infty})$.
\item $f_i$ $\{L^{p_i}\}_i$-converges weakly to $f_{\infty}$ on $B_R(x_{\infty})$ for some convergent sequence $\{p_i\}_i$ in $(1, \infty)$.
\item $f_i$ $\{L^{p_i}\}_i$-converges strongly to $f_{\infty}$ on $B_R(x_{\infty})$ for every convergent sequence $\{p_i\}_i$ in $(1, \infty)$.
\end{itemize}
See \cite[Remark $3.8$ and Proposition $3.32$]{holp}.
\end{remark}
Next we consider the case of tensor fields.
We denote by $r_z$ the distance function from $z$.
\begin{definition}\cite{holp}\label{lpt}
Let $r, s \in \mathbf{Z}_{\ge 0}$ and let $\{T_i\}_{i \le \infty}$ be a sequence of $T_i \in L^{p_i}(T^r_sB_R(x_i))$.
\begin{enumerate}
\item We say that \textit{$T_i$ $\{L^{p_i}\}_i$-converges weakly to $T_{\infty}$ on $B_R(x_{\infty})$} if the following two conditions hold:
\begin{enumerate}
\item $\sup_i||T_i||_{L^{p_i}(B_R(x_i))}<\infty$.
\item For any convergent sequence $\{z_i\}_{i \le \infty}$ of $z_i \in B_R(x_i)$, 
 sequences $\{z_{i, j}\}_{i \le \infty, 1 \le j \le r+s}$ of $z_{i, j} \in X_i$ with $z_{i, j} \stackrel{GH}{\to} z_{\infty, j}$ as $i \to \infty$,
and $r>0$ with $B_r(z_{\infty}) \subset B_R(x_{\infty})$, we have 
\[\lim_{i \to \infty}\int_{B_r(z_i)}\left\langle T_i, \bigotimes _{j=1}^r\nabla r_{z_{i, j}} \otimes \bigotimes_{j=r+1}^{r+s}dr_{z_{i, j}}\right\rangle d\upsilon_i=\int_{B_r(z_{\infty})}\left\langle T_{\infty}, \bigotimes _{j=1}^r\nabla r_{z_{\infty, j}} \otimes \bigotimes_{j=r+1}^{r+s}dr_{z_{\infty, j}} \right\rangle d\upsilon_{\infty}.\]
\end{enumerate}
\item We say that \textit{$T_i$ $\{L^{p_i}\}_i$-converges strongly to $T_{\infty}$ on $B_R(x_{\infty})$} if the following two conditions hold:
\begin{enumerate}
\item $T_i$ $\{L^{p_i}\}_i$-converges weakly to $T_{\infty}$ on $B_R(x_{\infty})$.
\item $\limsup_{i \to \infty}||T_i||_{L^{p_i}(B_R(x_{i}))}\le ||T_{\infty}||_{L^{p_{\infty}}(B_R(x_{\infty}))}$.
\end{enumerate}
\end{enumerate}
\end{definition}
The fundamental properties of these convergence include the following (see \cite{holp}):
\begin{itemize}
\item Every $\{L^{p_i}\}_i$-bounded sequence has an $\{L^{p_i}\}_i$-weak convergent subsequence.
\item $L^{p_i}$-norms are lower semicontinuous with respect to the $\{L^{p_i}\}_i$-weak convergence.
\item Every $\{L^{p_i}\}_i$-weak (or $\{L^{p_i}\}_i$-strong, respectively) convergent sequence is also an $\{L^{q_i}\}_i$-weak (or $\{L^{p_i}\}_i$-strong, respectively) convergent sequence for every convergent sequence $\{q_i\}_{i \le \infty}$ in $(1, \infty)$ with $q_i \le p_i$ for every $i$.
\end{itemize}
We also proved in \cite{holp} that Riemannian metrics of Ricci limit spaces behave $L^p$-weak continuously with respect to the Gromov-Hausdorff topology for every $1<p<\infty$ (see also Remark \ref{riemlp}).
As a corollary we knew the following.
\begin{theorem}\cite{holp}
The function $\mathrm{dim}: \overline{M(n, K)} \to \{0, 1, \ldots, n\}$ is lower semicontinuous, where $\mathrm{dim}\,X:=0$ when $X$ is a single point.
\end{theorem}
We say that a convergent sequence $(X_i, x_i, \upsilon_i) \stackrel{GH}{\to} (X_{\infty}, x_{\infty}, \upsilon_{\infty})$ in $\overline{M(n, K)}$ is \textit{noncollapsed} if 
\[\lim_{i \to \infty}\mathrm{dim}\,X_i=\mathrm{dim}\,X_{\infty}.\]
Note that this notion is well-known if $(X_i, x_i, \upsilon_i) \in M(n, K)$ for every $i<\infty$ (see Theorem \ref{fundpr}).

A main result of \cite{holp} is the following which is a Rellich type compactness for functions with respect to the Gromov-Hausdorff topology.
\begin{theorem}\cite{holp}\label{srell}
Let $\{f_i\}_{i<\infty}$ be a sequence of $f_i \in H^{1, p_i}(B_R(x_i))$ with $\sup_i||f_i||_{H^{1, p_i}}<\infty$.
Then there exist  $f_{\infty} \in H^{1, p_{\infty}}(B_R(x_{\infty}))$ and a subsequence $\{f_{i(j)}\}_j$ such that 
$f_{i(j)}$ $\{L^{p_{i(j)}}\}_j$-converges strongly to $f_{\infty}$ on $B_R(x_{\infty})$ and that $\nabla f_{i(j)}$ $\{L^{p_{i(j)}}\}_j$-converges weakly to $\nabla f_{\infty}$ on $B_R(x_{\infty})$. 
\end{theorem}
See \cite[Remark $3.77$ and Theorem $4.9$]{holp}.
\subsubsection{Generalized Yamabe constants}
Let $(X, \upsilon) \in \overline{M(n, K, d)}$ with $\mathrm{diam}\,X>0$, let $2<p \le \infty$ and let $g \in L^{p/2}(X)$.
We define \textit{the $p$-Yamabe constant $Y^g_p(X)$ of $(X, \upsilon)$ associated with $g$ (in the sense of Akutagawa-Carron-Mazzeo)} by 
\begin{align}\label{yama}Y^g_p(X):=\inf_{f}\int_{X}\left(|\nabla f|^2+g|f|^2\right)d\upsilon,
\end{align}
where $f$ runs over all $f \in \mathrm{LIP}(X)$ with $||f||_{L^{2p/(p-2)}}=1$. 

Note that by definition we have the following:
\begin{itemize}
\item $Y^g_n(X)=Y^g(X)$ ($Y^g(X)$ is defined in subsection $1.3$).
\item $Y^{g}_{\infty}(X)=\lambda^g_0(X)$ if $(X, \upsilon)$ satisfies the $(2p/(p-2), 2)$-Sobolev inequality on $X$ for some $(A, B)$.
\item It follows from the  H$\ddot{\text{o}}$lder inequality that
\begin{align}\label{ybound}
-\left(\int_X|g|^{p/2}d\upsilon \right)^{2/p}\le Y^g_p(X)\le \int_Xg d\upsilon.
\end{align}
\end{itemize}
By Theorem \ref{fundpr} and \cite[Theorem $3.1$]{acm2} we have the following:
\begin{theorem}\cite{acm2}\label{acmex}
Assume that there exist $A, B>0$ such that $(X, \upsilon)$ satisfies the $(2p/(p-2), 2)$-Sobolev inequality on $X$ for $(A, B)$ and that  
\[Y^g_p(X)<A^{-1}.\]
Then there exists a minimizer $f \in H^{1, 2}(X)$ of (\ref{yama}).
\end{theorem}
Note that they discussed a regularity of minimizers of (\ref{yama}) in \cite{acm2}.
\subsubsection{Gigli's Sobolev spaces on Ricci limit spaces}
Let $(X, \upsilon) \in \overline{M(n, K, d)}$ with $\mathrm{diam}\,X>0$.
We first recall Gigli's Sobolev space $W^{1, 2}_C(TX)$ for vector fields on $X$ defined in \cite{gigli}:
\begin{itemize}
\item Let $W^{2, 2}(X)$ be the set of $f \in H^{1, 2}(X)$ satisfying that there exists $A \in L^2(T^0_2X)$ such that 
\begin{align}\label{78}
&2\int_{X}g_0\left\langle A, dg_1 \otimes dg_2 \right\rangle d\upsilon \nonumber \\
&=\int_X\left(-\langle \nabla f, \nabla g_1 \rangle \mathrm{div}^{\upsilon}(g_0\nabla g_2)-\langle \nabla f, \nabla g_2\rangle \mathrm{div}^{\upsilon}(g_0\nabla g_1)- g_0 \left\langle \nabla f, \nabla \left\langle \nabla g_1, \nabla g_2 \right\rangle\right\rangle \right)d\upsilon
\end{align}
for any $g_i \in \mathrm{Test}F(X)$.
Since $A$ is unique if it exists because 
\[\mathrm{Test}T^0_2X:=\left\{ \sum_{i=1}^Ng_{0, i}dg_{1, i} \otimes dg_{2, i}; N \in \mathbf{N}, g_{i, j} \in \mathrm{Test}F(X)\right\}\]
is dense in $L^2(T^0_2X)$, we denote $A$ by $\mathrm{Hess}^{\upsilon}_f$.
Moreover we see that $\mathrm{Test}F(X) \subset W^{2, 2}(X)$ and that $W^{2, 2}(X)$ is a Hilbert space equipped with the norm
\[||f||_{W^{2, 2}}:= \left(||f||_{H^{1, 2}}^2+||\mathrm{Hess}^{\upsilon}_f||_{L^2}^2\right)^{1/2}.\]
\item Let $W^{1, 2}_C(TX)$ be the set of $V \in L^2(TX)$ satisfying that there exists $T \in L^2(T^1_1X)$ such that 
\[\int_Xg_0\left\langle T, \nabla g_1 \otimes d g_2 \right\rangle d\upsilon=-\int_X\left(\langle V, \nabla g_1\rangle \mathrm{div}^{\upsilon}(g_0\nabla g_2)+g_0\left\langle \mathrm{Hess}^{\upsilon}_{g_1}, V^* \otimes d g_2\right\rangle \right)d\upsilon\]
for any $g_i \in \mathrm{Test}F(X)$.
Since $T$ is unique if it exists, we denote it by $\nabla^{\upsilon}V$.
Moreover we see that $W^{1, 2}_C(TX)$ is also a Hilbert space equipped with the norm
\[||V||_{W^{1, 2}_C}:=\left( ||V||_{L^2}^2+||\nabla^{\upsilon}V||_{L^2}^2\right)^{1/2}.\]
\end{itemize}
See subsection $3.2$ in \cite{gigli} for the detail.
Note that as we mentioned in subsection $1.4$, we use a slight modified version of the original Gigli's definition of $W^{1, 2}_C(TX)$ in this paper.
\begin{remark}\label{gigliremark}
The original definition of $W^{1, 2}_C(TX)$ by Gigli in \cite{gigli} is as follows:
Let $W^{1, 2}_C(TX)$ be the set of $V \in L^2(TX)$ satisfying that there exists $\hat{T} \in L^2(T^2_0X)$ such that
\[\int_Xg_0\langle \hat{T}, \nabla g_1 \otimes \nabla g_2\rangle d\upsilon=-\int_X\left(\langle V, \nabla g_2\rangle\mathrm{div}^{\upsilon}(g_0\nabla g_1)+g_0\langle \mathrm{Hess}_{g_2}^{\upsilon}, V^*\otimes dg_1\rangle \right)d\upsilon\]
for any $g_i \in \mathrm{Test}F(X)$.
Thus $\hat{T}$ coincides with $\nabla^{\upsilon}V$ via the identification (\ref{hondahonda}).
\end{remark}
For any $V \in W^{1, 2}_C(TX)$ and $W \in L^{p}(TX)$ we define $\nabla^{\upsilon}_{W}V \in L^{2p/(p+2)}(TX)$ by satisfying
\[\left\langle \nabla^{\upsilon}_{W}V, Z \right\rangle=\left\langle \nabla^{\upsilon}V, Z \otimes W^*\right\rangle\]
for every $Z \in \Gamma_0(TX)$, where $p \in [2, \infty]$. Compare with \cite[(3.4.6)]{gigli}.
For any $V, W \in W^{1, 2}_C(TX)$ we define the Lie bracket $[V, W]^{\upsilon} \in L^1(TX)$ in Gigli's sense by
\[[V, W]^{\upsilon}:=\nabla^{\upsilon}_VW-\nabla^{\upsilon}_WV.\]
For convenience we use the following notation:
\[\nabla^r_sh:=\nabla h_1 \otimes \cdots \otimes \nabla h_r \otimes dh_{r+1} \otimes \cdots \otimes dh_{r+s} \in \Gamma_0(T^r_sX)\]
for any $h_i \in \Gamma_1(X)$, where $h=(h_1, \ldots, h_{r+s})$.

We are now in a position to define the Sobolev space $W^{1, 2}_C(T^r_sX)$ for tensor fields in the same manner of \cite{gigli} by Gigli:
\begin{definition}
Let $W^{1, 2}_C(T^r_sX)$ be the set of $T \in L^2(T^r_sX)$ satisfying that there exists $S \in L^2(T^r_{s+1}X)$ such that 
\begin{align}\label{appp}
\int_{X}g_0\left\langle S, \nabla^r_{s+1}g\right\rangle d\upsilon &=-\int_{X}\left(\left\langle T, \nabla^{r}_sg^{(r+s+1)}\right\rangle \mathrm{div}^{\upsilon}(g_0\nabla g_{r+s+1})+\sum_{j=1}^{r+s}g_0\left\langle T^*, R_{j}^X(g)\right\rangle \right)d\upsilon
\end{align}
holds for any $g_{j} \in \mathrm{Test}F(X)$, where $g:=(g_{1}, \ldots, g_{r+s+1})$, $g^{(r+s+1)}:=(g_{1}, \ldots, g_{r+s})$,
\begin{align}\label{0098}
R_{j}^X(g):=\nabla g_{1} \otimes \cdots \otimes \nabla g_{j-1} \otimes \left(\nabla ^{\upsilon}_{\nabla g_{r+s+1}}\nabla g_{j}\right) \otimes \nabla g_{j+1} \otimes \cdots \otimes \nabla g_{r+s+1}
\end{align}
and $^*$ is the canonical isometry: 
\[T^r_sX \cong T^{r+s}_0X\]
\[\bigotimes_{i=1}^{r} v_i \otimes \bigotimes_{j=1}^sw_j \mapsto \bigotimes_{i=1}^{r} v_i \otimes \bigotimes_{j=1}^sw_j^*.\] 
Since $S$ is unique if it exists we write it by $\nabla^{\upsilon}T$.
\end{definition}
Note that if $(X, \upsilon)$ is a smooth compact metric measure space and $T \in C^{\infty}(T^r_sX)$, then $\nabla T=\nabla^{\upsilon}T$, where $\nabla T$ is the covariant derivative as a smooth tensor field in the manner of \cite{sakai}.

By an argument similar to the proof of \cite[Theorem $3.4.2$]{gigli} we see that $W^{1, 2}_C(T^r_sX)$ is a Hilbert space equipped with the norm
\[||T||_{W^{1, 2}_C}:=\left(||T||_{L^2}^2+||\nabla^{\upsilon}T||_{L^2}^2\right)^{1/2},\]
that 
\[\mathrm{Test}T^r_sX:=\left\{ \sum_{k=1}^N h_0^k \nabla^{r}_sh^k; N \in \mathbf{N}, \,\,h^k=(h^k_1, \ldots, h^k_{r+s}),\,\,h^k_i \in \mathrm{Test}F(X)\right\}\]
is a linear subspace of $W^{1, 2}_C(T^r_sX)$ and that $\{(T, \nabla^{\upsilon}T); T \in W^{1, 2}_C(T^r_sX)\}$ is a closed subset of $L^2(T^r_sX) \times L^2(T^r_{s+1}X)$.
Let $H^{1, 2}_C(T^r_sX)$ be the closure of $\mathrm{Test}T^r_sX$ in $W^{1, 2}_C(T^r_sX)$.

We also recall Gigli's exterior derivative $d^{\upsilon}$:
\begin{itemize}
\item Let $W^{1, 2}_d(\bigwedge^kT^*X)$ be the set of $\omega \in L^2(\bigwedge^kT^*X)$ satisfying that there exists $\eta \in L^2(\bigwedge^{k+1}T^*X)$ such that 
\begin{align}\label{rrer}
\int_X \eta (V_0, \ldots, V_k) d\upsilon &=\sum_i(-1)^{i+1}\int_X\omega (V_0, \ldots, V_{i-1}, V_{i+1}, \ldots V_k)\mathrm{div}^{\upsilon}V_id\upsilon \nonumber \\
&+\sum_{i<j}(-1)^{i+j}\int_X\omega ([V_i, V_j]^{\upsilon}, V_0, \ldots, V_{i-1}, V_{i+1}, \ldots, V_{j-1}, V_{j+1}, \ldots, V_k)d\upsilon
\end{align}
for any $V_i \in \mathrm{Test}T^1_0X$. Since $\eta$ is unique if it exists we denote it by $d^{\upsilon}\omega $.
Moreover we see that $\mathrm{TestForm}_kX \subset W^{1, 2}_d(\bigwedge^kT^*X)$ and that $W^{1, 2}_d(\bigwedge^kT^*X)$ is a Hilbert space equipped with the norm
\[||\omega||_{W^{1, 2}_d}:=\left(||\omega||_{L^2}^2+||d^{\upsilon}\omega||_{L^2}^2\right)^{1/2}.\]
Let $H^{1, 2}_d(\bigwedge^kT^*X)$ be the closure of $\mathrm{TestForm}_kX$ in $W^{1, 2}_d(\bigwedge^kT^*X)$.
\end{itemize} 
See subsection $1.4$ for the definition of the codifferential $\delta^{\upsilon}_k$ and the Hodge Laplacian $\Delta^{\upsilon}_{H, k}$ for differential $k$-forms.
Note that $\delta^{\upsilon}_1$ coincides with $\delta^{\upsilon}$ introduced in subsection $2.3.2$. See \cite[$(3.5.13)$]{gigli}.
\begin{remark}\label{appremark}
Recall a mollified heat flow from \cite[(3.2.3)]{gigli},
\begin{align}\label{mollified}
\tilde{h}_{\epsilon}f:=\frac{1}{\epsilon}\int_0^{\infty}h_sf\phi (s\epsilon^{-1})ds,
\end{align}
where $h_s$ is the heat flow and $\phi$ is a nonnegatively valued smooth function on $(0, 1)$ with compact support and
\[\int_0^1\phi(s) ds=1.\]
By using this Gigli proved an existence of the following approximation:
For every $f \in \mathrm{LIP}(X)$ there exists a sequence $\{f_i\}_i$ of $f_i \in \mathrm{Test}F(X)$ such that $\sup_{i}\mathbf{Lip}f_i<\infty$, that $f_i \to f$ in $H^{1, 2}(X)$ and that $\Delta^{\upsilon}f_i \in L^{\infty}(X)$ for every $i$.

By using this approximation, it is not difficult to check the following:
\begin{itemize}
\item For every $f \in W^{2, 2}(X)$, we see that (\ref{78}) holds for any $g_1, g_2 \in \mathrm{Test}F(X)$ and $g_0 \in \mathrm{LIP}(X)$. 
\item For every $T \in W^{1, 2}_C(T^r_sX)$, we see that (\ref{appp}) holds for any $\{g_{i}\}_{1 \le i \le r+s+1} \subset \mathrm{Test}F(X)$ and $g_0 \in \mathrm{LIP}(X)$.
\end{itemize}
\end{remark}
\section{$L^p$-convergence revisited}
In this section we prove several new results on the $L^p$-convergence  with respect to the Gromov-Hausdorff topology.
They play crucial roles in the proofs of results introduced in Section $1$.
\subsection{Stabilities of Sobolev and Poincar\'e inequalities}
In this subsection we give two `stabilities'. 
The first one is the following which is a stability of Sobolev inequalities with respect to the Gromov-Hausdorff topology.
\begin{theorem}\label{stasob}
Let $R>0$, let $\{q_i\}_{i \le \infty}$ and $\{p_i\}_{i \le \infty}$ be convergent sequences in $[1, \infty)$, let $\{\hat{A}_i, \hat{B}_i\}_{i<\infty}$ be bounded sequences in $[0, \infty)$, and let $(X_i, x_i, \upsilon_i) \stackrel{GH}{\to} (X_{\infty}, x_{\infty}, \upsilon_{\infty})$ in $\overline{M(n, K)}$ with $\mathrm{diam}\,X_{\infty}>0$.
Assume that for every $i<\infty$, $(X_i, \upsilon_i)$ satisfies the $(q_i, p_i)$-Sobolev inequality on $B_R(x_i)$ for $(\hat{A}_i, \hat{B}_i)$.
Then we see that $(X_{\infty}, \upsilon_{\infty})$ satisfies the $(q_{\infty}, p_{\infty})$-Sobolev inequality on $B_R(x_{\infty})$ for $(\liminf_{i \to \infty}\hat{A}_i, \liminf_{i \to \infty}\hat{B}_i)$.
\end{theorem}
\begin{proof}
Without loss of generality we can assume that the limits 
$\lim_{i \to \infty}\hat{A}_i$ and $\lim_{i \to \infty}\hat{B}_i$
exist.

Let $f_{\infty} \in \mathrm{LIP}_c(B_R(x_{\infty}))$ and let $r \in (0, R)$ with $\mathrm{supp}\,f_{\infty} \subset B_r(x_{\infty})$.
By applying \cite[Theorem $4.2$]{holip} for $A_i:=\overline{B}_R(x_i) \setminus B_r(x_i)$ and $A_{\infty}:=\overline{B}_R(x_{\infty}) \setminus B_r(x_{\infty})$ there, without loss of generality we can assume that there exists a sequence $\{f_i\}_{i} \in \mathrm{LIP}_c(B_R(x_{i}))$ such that $f_i, \nabla f_i$ $L^r$-converge strongly to $f_{\infty}, \nabla f_{\infty}$ on $B_R(x_{\infty})$ for every $1<r<\infty$, respectively.
Since
\[\left(\int_{B_R(x_i)}|f_i|^{q_i}d\upsilon_i\right)^{p_i/q_i}\le \hat{A}_i \int_{B_R(x_i)}|\nabla f_i|^{p_i}d\upsilon_i + \hat{B}_i\int_{B_R(x_i)}|f_i|^{p_i}d\upsilon_i \]
for every $i<\infty$, letting $i \to \infty$ completes the proof.
\end{proof}
See \cite{he2, he1, he-va} for the sharp Sobolev inequality on Riemannian manifolds.

Similarly, by Remark \ref{85}, we have the following stability of Poincar\'e inequalities with respect to the Gromov-Hausdorff topology.
See \cite{ch1} and \cite{keith2} for related works. 
\begin{theorem}\label{stapo}
Let $R>0$, let $\{q_i\}_{i \le \infty}$ and $\{p_i\}_{i \le \infty}$ be convergent sequences in $[1, \infty)$, let $\{\tau_i\}_i$ be a bounded sequence in $[0, \infty)$,
and let $(X_i, x_i, \upsilon_i) \stackrel{GH}{\to} (X_{\infty}, x_{\infty}, \upsilon_{\infty})$ in $\overline{M(n, K)}$ with $\mathrm{diam}\,X_{\infty}>0$.
Assume that for every $i<\infty$, $(X_i, \upsilon_i)$ satisfies the $(q_i, p_i)$-Poincar\'e inequality on $B_R(x_i)$ for $\tau_i$.
Then we see that $(X_{\infty}, \upsilon_{\infty})$ satisfies the $(q_{\infty}, p_{\infty})$-Poincar\'e inequality on $B_R(x_{\infty})$ for $\liminf_{i \to \infty}\tau_i$.
\end{theorem}
\subsection{Sobolev embeddings}
The main purpose of this subsection is to introduce two `embedding' theorems.
In order to introduce them we first give the following:
\begin{proposition}\label{ssrt}
Let $r, s \in \mathbf{Z}_{\ge 0}$, let $R>0$, let $1<p<q<\infty$, let $(X_i, x_i, \upsilon_i) \stackrel{GH}{\to} (X_{\infty}, x_{\infty}, \upsilon_{\infty})$ in $\overline{M(n, K)}$, and let $\{T_i\}_{i \le \infty}$ be an $L^p$-strong convergent sequence on $B_R(x_{\infty})$ of $T_i \in L^q(T^r_sB_R(x_i))$ with $\sup_i||T_i||_{L^q}<\infty$.
Then we see that $T_i$ $L^t$-converges strongly to $T_{\infty}$ on $B_R(x_{\infty})$ for every $1<t<q$.
\end{proposition} 
\begin{proof}
By \cite[Propositions $3.22$ and $3.56$]{holp} there exist a sequence $\{T_{i, j}\}_{i \le \infty, j<\infty}$ of $T_{i, j} \in L^{\infty}(T^r_sB_R(x_i))$ such that $T_{i, j}$ $L^t$-converges strongly to $T_{\infty, j}$ on $B_R(x_{\infty})$ for any $t \in (1, \infty)$ and $j$, that $T_{\infty, j} \to T_{\infty}$ in $L^q(T^r_sB_R(x_{\infty}))$, and that 
\[\sup_i||T_{i, j}||_{L^{\infty}}<\infty\]
for every $j$ (see also \cite[Proposition 3.69]{holp}).
Thus there exists a subsequence $\{j(i)\}_i$ of $\mathbf{N}$ such that $\hat{T}_{i}:=T_{i, j(i)}$ $L^q$-converges strongly to $T_{\infty}$ on $B_R(x_{\infty})$.
For every $i<\infty$, let
\[\epsilon_i:=\int_{B_R(x_i)}|T_i-\hat{T}_i|^pd\upsilon_i+i^{-1}\]
and let $A_i:=\{y \in B_R(x_i); |T_i-\hat{T}_i|^p(y) \ge \epsilon_i^{1/2}\}$.
Note that $\epsilon_i \to 0$ as $i \to \infty$.
Then we have
\begin{align*}
\upsilon_i(A_i)\le \epsilon_i^{-1/2}\int_{A_i}|T_i-\hat{T}_i|^pd\upsilon_i \le \epsilon_i^{-1/2}\int_{B_R(x_i)}|T_i-\hat{T}_i|^pd\upsilon_i \le \epsilon_i^{1/2} \to 0
\end{align*}
as $i \to \infty$.
Therefore
\begin{align*}
\int_{B_R(x_i)}|T_i-\hat{T}_i|^td\upsilon_i &= \int_{A_i}|T_i-\hat{T}_i|^td\upsilon_i+\int_{B_R(x_i) \setminus A_i}|T_i-\hat{T}_i|^td\upsilon_i \\
&\le \left(\upsilon_i(A_i)\right)^{(q-t)/q}\left(\int_{B_R(x_i)}|T_i-\hat{T}_i|^qd\upsilon_i\right)^{t/q} + \upsilon_i(B_R(x_i))\epsilon_i^{t/(2p)} \to 0 
\end{align*}
as $i \to \infty$.
This completes the proof.
\end{proof}
We now recall that in \cite{HK1} Haj\l asz-Koskela gave a generalization of the Rellich-Kondrachov compactness to the metric measure space setting in the following sense:
If a metric measure space $(X, \upsilon)$ satisfies a doubling condition and the $(q, p)$-Poincar\'e inequality on a ball $B_R(x) \subset X$ for some $\tau$, then we see that $H^{1, p}(B_R(x)) \hookrightarrow L^{q}(B_R(x))$ and that the embedding $H^{1, p}(B_R(x)) \hookrightarrow L^r(B_R(x))$ is compact for every $r<q$.

We now give a generalization of this to the Gromov-Hausdorff setting which is the first embedding theorem:
\begin{theorem}\label{tt}
Let $R >0$, let $\{q_i\}_{i \le \infty}$ and $\{p_i\}_{i \le \infty}$ be convergent sequences in $(1, \infty)$, 
let $(X_i, x_i, \upsilon_i) \stackrel{GH}{\to} (X_{\infty}, x_{\infty}, \upsilon_{\infty})$ in $\overline{M(n, K)}$ with $\mathrm{diam}\,X_{\infty}>0$, and let $\{f_i\}_{i < \infty}$ be a sequence of $f_i \in H^{1, p_i}(B_R(x_i))$ with $\sup_{i}||f_i||_{H^{1, p_i}(B_R(x_i))}<\infty$.
Assume 
\begin{align}\label{whty}
\sup_{i<\infty}||f_i||_{L^{q_i}(B_R(x_i))}<\infty.
\end{align}
Then there exist $f_{\infty} \in L^{q_{\infty}}(B_R(x_{\infty}))$ and a subsequence $\{f_{i(j)}\}_j$ such that $f_{i(j)}$ $\{L^{q_i}\}_i$-converges weakly to $f_{\infty}$ on $B_R(x_{\infty})$ and that $f_{i(j)}$ $L^r$-converges strongly to $f_{\infty}$ on $B_R(x_{\infty})$ for every $1<r<q_{\infty}$.
\end{theorem}
\begin{proof}
It is a direct consequence of Theorem \ref{srell} and Proposition \ref{ssrt}
\end{proof}
The following is the second embedding theorem:
\begin{theorem}\label{sobo3}
Let $\{q_i\}_{i \le \infty}$ and $\{p_i\}_{i \le \infty}$ be convergent sequences in $(1, \infty)$,
let $(X_i, \upsilon_i) \stackrel{GH}{\to} (X_{\infty}, \upsilon_{\infty})$ in $\overline{M(n, K, d)}$ with $\mathrm{diam}\,X_{\infty}>0$, and let
$\{f_i\}_{i \le \infty}$ be a sequence of $f_i \in H^{1, p_i}(X_i)$ satisfying that $f_i$ $\{L^{p_i}\}_{i}$-converges weakly to $f_{\infty}$ on $X_{\infty}$ and that $\nabla f_i$ $\{L^{p_i}\}_i$-converges strongly to $\nabla f_{\infty}$ on $X_{\infty}$.
Assume that there exist $A, B>0$ such that for every $i<\infty$, $(X_i, \upsilon_i)$ satisfies the $(q_i, p_i)$-Sobolev inequality on $X_i$ for $(A, B)$.
Then we see that $f_i$ $\{L^{r_i}\}_i$-converges strongly to $f_{\infty}$ on $X_{\infty}$, where $r_i=\max\{p_i, q_i\}$. 
\end{theorem}
\begin{proof}
Theorem \ref{srell} yields that $f_i$ $\{L^{p_i}\}_i$-converges strongly to $f_{\infty}$ on $X_{\infty}$.
Thus it suffices to check that this is also an $\{L^{q_i}\}_i$-strong convergence.

By Theorem \ref{fundpr} and an argument similar to the proof of Theorem \ref{stasob} without loss of generality we can assume that there exist sequences $\{f_{i, j}\}_{i \le \infty, j<\infty}$ of $f_{i, j} \in \mathrm{LIP}(X_i)$ such that $f_{i, j}, \nabla f_{i, j}$ $L^r$-converge strongly to $f_{\infty, j}, \nabla f_{\infty, j}$ on $X_{\infty}$ as $i \to \infty$ for any $r, j$, respectively and that
$f_{\infty, j} \to f_{\infty}$ in $H^{1, p_{\infty}}(X_{\infty})$ as $j \to \infty$.

Note
\begin{align}\label{hou}
\left( \int_{X_i}\left|f_i-f_{i, j}\right|^{q_i}d\upsilon_i \right)^{p_i/q_i} \le A \int_{X_i}\left|\nabla (f_i-f_{i, j})\right|^{p_i}d\upsilon_i + B\int_{X_i}|f_i -f_{i, j}|^{p_i}d\upsilon_i.
\end{align}
Since the right hand side of (\ref{hou}) goes to $0$ as $i \to \infty$ and $j \to \infty$,  we have
\[\lim_{j \to \infty}\left(\limsup_{i \to \infty}\int_{X_i}\left|f_i-f_{i, j}\right|^{q_i}d\upsilon_i\right)=0.\]
This completes the proof.
\end{proof}
\subsection{Fatou's lemma in the Gromov-Hausdorff setting}
Throughout this subsection we always consider the following setting:
\begin{itemize}
\item Let $R>0$.
\item Let $(X_i, x_i, \upsilon_i) \stackrel{GH}{\to} (X_{\infty}, x_{\infty}, \upsilon_{\infty})$ in $\overline{M(n, K)}$ with $\mathrm{diam}\,X_{\infty}>0$.
\item Let $\{f_i\}_{i \le \infty}$ be a sequence of $f_i \in L^1_{\mathrm{loc}}(B_R(x_i))$, where $L^1_{\mathrm{loc}}(B_R(x_i))$ is the space of locally $L^1$-functions on $B_R(x_i)$.
\end{itemize}
The main purpose of this subsection is to introduce a generalization of Fatou's lemma to the Gromov-Hausdorff setting.
In order to give the precise statement we introduce the following.
\begin{definition}\label{co}
\begin{enumerate}
\item We say that \textit{$\{f_i\}_i$ is $L^1_{\mathrm{loc}}$-weakly lower semicontinuous at a point $z_{\infty} \in B_R(x_{\infty})$} if  for every $\epsilon>0$ there exist $r_0>0$ and a convergent sequence $\{z_{i}\}_{i \le \infty}$ of $z_{i} \in B_R(x_{i})$ such that 
\[\liminf_{i \to \infty}\left(\frac{1}{\upsilon_{i}(B_r(z_{i}))}\int_{B_r(z_{i})}f_{i}d\upsilon_{i}\right)- \frac{1}{\upsilon_{\infty}(B_r(z_{\infty}))}\int_{B_r(z_{\infty})}f_{\infty}d\upsilon_{\infty} \ge -\epsilon\] 
holds for every $r<r_0$.
\item We say that \textit{$\{f_i\}_i$ is $L^1_{\mathrm{loc}}$-weakly upper semicontinuous at a point $z_{\infty} \in B_R(x_{\infty})$} if  for every $\epsilon>0$ there exist $r_0>0$ and a convergent sequence $\{z_{i}\}_{i \le \infty}$ of $z_{i} \in B_R(x_{i})$ such that 
\[\limsup_{i \to \infty}\left(\frac{1}{\upsilon_{i}(B_r(z_{i}))}\int_{B_r(z_{i})}f_{i}d\upsilon_{i}\right)- \frac{1}{\upsilon_{\infty}(B_r(z_{\infty}))}\int_{B_r(z_{\infty})}f_{\infty}d\upsilon_{\infty} \le \epsilon\] 
holds for every $r<r_0$.
\item We say that \textit{$\{f_i\}_i$ is $L^1_{\mathrm{loc}}$-weakly continuous at a point $z_{\infty} \in B_R(x_{\infty})$} if  for every $\epsilon>0$ there exist $r_0>0$ and a convergent sequence $\{z_{i}\}_{i \le \infty}$ of $z_{i} \in B_R(x_{i})$ such that 
\[\limsup_{i \to \infty}\left|\frac{1}{\upsilon_{i}(B_r(z_{i}))}\int_{B_r(z_{i})}f_{i}d\upsilon_{i}- \frac{1}{\upsilon_{\infty}(B_r(z_{\infty}))}\int_{B_r(z_{\infty})}f_{\infty}d\upsilon_{\infty}\right| \le \epsilon\] 
holds for every $r<r_0$.
\end{enumerate}
\end{definition}
Compare with \cite[Definition $3.4$]{holp}.

The following is a generalization of Fatou's lemma to the Gromov-Hausdorff setting.
The proof is essentially same to that of \cite[Proposition $3.5$]{holp}, however we give a proof for reader's convenience:
\begin{proposition}\label{fatou1}
Assume that $f_i \ge 0$ for every $i\le \infty$ and that $\{f_i\}_i$ $L^1_{\mathrm{loc}}$-weakly lower semicontinuous at a.e. $z_{\infty} \in B_R(x_{\infty})$.
Then we have
\[\liminf_{i \to \infty}\int_{B_R(x_i)}f_id\upsilon_i \ge \int_{B_R(x_{\infty})}f_{\infty}d\upsilon_{\infty}.\]
\end{proposition}
\begin{proof}
Without loss of generality we can assume that $\sup_{i<\infty}||f_i||_{L^1(B_R(x_i))}<\infty$.

Let $K$ be a Borel subset of $B_R(x_{\infty})$ satisfying that $\upsilon_{\infty}(B_R(x_{\infty}) \setminus K)=0$ and that $\{f_i\}_i$  is $L^1_{\mathrm{loc}}$-weakly lower semicontinuous at every $z_{\infty} \in K$.

Fix $\epsilon>0$.
By a standard covering argument (c.f. \cite{le} or \cite[Proposiiton $2.2$]{holip}) and our assumption, there exist a  pairwise disjoint countable collection 
$\{\overline{B}_{r_k}(w_{\infty, k})\}_k$ and sequences $\{w_{i, k}\}_{i \le \infty, k}$ of $w_{i, k} \in B_R(x_i)$ with $w_{i, k} \stackrel{GH}{\to} w_{\infty, k}$ as $i \to \infty$ such that $\overline{B}_{5r_k}(w_{\infty, k}) \subset B_R(x_{\infty})$, that $w_{\infty, k} \in K$, that 
$K \setminus \bigcup_{k=1}^N\overline{B}_{r_k}(w_{\infty, k}) \subset \bigcup_{k=N+1}^{\infty}\overline{B}_{5r_k}(w_{\infty, k})$ for every $N \in \mathbf{N}$, and that
\[\liminf_{i \to \infty}\left(\frac{1}{\upsilon_{i}(B_{r_k}(w_{i, k}))}\int_{B_{r_k}(w_{i, k})}f_{i}d\upsilon_{i}\right) \ge \frac{1}{\upsilon_{\infty}(B_{r_k}(w_{\infty, k}))}\int_{B_{r_k}(w_{\infty, k})}f_{\infty}d\upsilon_{\infty}-\epsilon.\]

Let $N_0 \in \mathbf{N}$ with $\sum_{k=N_0 +1}^{\infty}\upsilon_{\infty}(B_{5r_k}(w_{\infty, k}))<\epsilon$ and let $K^{\epsilon}:=K \cap \bigcup_{k=1}^{N_0}\overline{B}_{r_k}(w_{\infty, k})$.
Then we have 
\begin{align*}
\int_{K^{\epsilon}}f_{\infty}d\upsilon_{\infty}\le \sum_{k=1}^{N_0}\int_{B_{r_k}(w_{\infty, k})}f_{\infty}d\upsilon_{\infty}&\le \sum_{k=1}^{N_0}\left(\int_{B_{r_k}(w_{i, k})}f_id\upsilon_{i} + 2\epsilon \upsilon_{i}(B_{r_k}(w_{i, k}))\right) \\
&\le \int_{B_R(x_{i})}f_{i}d\upsilon_{i} +2\epsilon \upsilon_{i}(B_R(x_{i}))
\end{align*}
for every sufficiently large $i$.
Since $\upsilon_{\infty}(B_R(x_{\infty}) \setminus K^{\epsilon})<\epsilon$, by letting $i \to \infty$ and $\epsilon \to 0$, the dominated convergence theorem yields the assertion.
\end{proof}
Next we give a fundamental property of the $L^1_{\mathrm{loc}}$-weak upper semicontinuity:
\begin{proposition}\label{upperpro}
Assume that $f_{\infty} \in L^1(B_R(x_{\infty}))$, that $\{f_i\}_i$ is $L^1_{\mathrm{loc}}$-weakly upper semicontinuous at a.e. $z_{\infty} \in B_R(x_{\infty})$ and that there exists $p>1$ such that $\sup_{i < \infty}||f_i||_{L^p(B_R(x_i))}<\infty$.
Then we have
\[\limsup_{i \to \infty}\int_{B_R(x_i)}f_id\upsilon_i \le \int_{B_R(x_{\infty})}f_{\infty}d\upsilon_{\infty}.\]
\end{proposition}
\begin{proof}
Let $L:=\sup_{i < \infty}||f_i||_{L^p(B_R(x_i))}$, let $\epsilon>0$ and let $K$ be a Borel subset of $B_R(x_{\infty})$ satisfying that $\upsilon_{\infty}(B_R(x_{\infty})\setminus K)=0$ and that $\{f_i\}_i$ is $L^1_{\mathrm{loc}}$-weakly upper semicontinuous at every $z_{\infty} \in K$.
By an argument similar to the proof of Proposition \ref{fatou1}, there exist a pairwise disjoint countable collection 
$\{\overline{B}_{r_k}(w_{\infty, k})\}_k$ and sequences $\{w_{i, k}\}_{i \le \infty, k}$ of $w_{i, k} \in B_R(x_i)$ such that $w_{i, k} \stackrel{GH}{\to} w_{\infty, k}$ as $i \to \infty$, that $\overline{B}_{5r_k}(w_{\infty, k}) \subset B_R(x_{\infty})$, that $w_{\infty, k} \in K$, that 
$K \setminus \bigcup_{k=1}^N\overline{B}_{r_k}(w_{\infty, k}) \subset \bigcup_{k=N+1}^{\infty}\overline{B}_{5r_k}(w_{\infty, k})$ for every $N \in \mathbf{N}$, and that
\[\limsup_{i \to \infty}\left(\frac{1}{\upsilon_{i}(B_{r_k}(w_{i, k}))}\int_{B_{r_k}(w_{i, k})}f_{i}d\upsilon_{i}\right) \le \frac{1}{\upsilon_{\infty}(B_{r_k}(w_{\infty, k}))}\int_{B_{r_k}(w_{\infty, k})}f_{\infty}d\upsilon_{\infty}+\epsilon.\]

Let $N_0 \in \mathbf{N}$ with $\sum_{k=N_0 +1}^{\infty}\upsilon_{\infty}(B_{5r_k}(w_{\infty, k}))<\epsilon$ and
\[\left|\int_{B_R(x_{\infty})}f_{\infty}d\upsilon_{\infty} - \int_{\bigcup_{k=1}^{N_0}B_{r_k}(w_{\infty, k})}f_{\infty}d\upsilon_{\infty} \right| <\epsilon.\]
Then we have 
\begin{align*}
\int_{B_R(x_{\infty})}f_{\infty}d\upsilon_{\infty} &\ge \sum_{k=1}^{N_0}\int_{B_{r_k}(w_{\infty, k})}f_{\infty}d\upsilon_{\infty}-\epsilon \\
&\ge \sum_{k=1}^{N_0}\int_{B_{r_k}(w_{i, k})}f_{i}d\upsilon_{i}-\epsilon\left(1+\upsilon_{\infty}(B_R(w_{\infty, k})) \right) \\
\end{align*}
for every sufficiently large $i<\infty$.
On the other hand the H$\ddot{\text{o}}$lder inequality yields
\begin{align*}
\left| \int_{B_R(x_i)}f_id\upsilon_i- \sum_{k=1}^{N_0}\int_{B_{r_k}(w_{i, k})}f_{i}d\upsilon_{i} \right| &\le \int_{B_R(x_i)\setminus \bigcup_{k=1}^{N_0}B_{r_k}(w_{\infty, k})}|f_i|d\upsilon_i \\
&\le \left(\upsilon_i \left( B_R(x_i)\setminus \bigcup_{k=1}^{N_0}B_{r_k}(w_{i, k})\right)\right)^{(p-1)/p}||f_i||_{L^p(B_R(x_i))} \\
&\le \epsilon^{(p-1)/p}L
\end{align*}
for every sufficiently large $i<\infty$.
Thus 
\[\int_{B_R(x_{\infty})}f_{\infty}d\upsilon_{\infty}\ge \int_{B_R(x_i)}f_id\upsilon_i -\epsilon\left(1+\upsilon_{\infty}(B_R(w_{\infty, k})) \right)- \epsilon^{(p-1)/p}L.\]
Letting $i \to \infty$ and $\epsilon \to 0$ completes the proof.
\end{proof}
In order to introduce another upper semicontinuity of $L^1$-norms (Corollary \ref{fatou2}) we also consider the following notion:
\begin{definition}\label{ll}
We say that \textit{$f_i$ $w$-converges $f_{\infty}$ on $B_R(x_{\infty})$} if 
for any $\epsilon>0$ and subsequence $\{i(j)\}_j$ there exist a subsequence $\{j(k)\}_k$ of $\{i(j)\}_j$, and a sequence $\{A(j(k), \epsilon)\}_{k \le \infty}$ of compact subsets $A(j(k), \epsilon)$ of $B_R(x_{j(k)})$ such that the following three conditions hold:
\begin{enumerate}
\item $\upsilon_{j(k)}\left(B_R(x_{j(k)}) \setminus A(j(k), \epsilon)\right)\le \epsilon$ for every $k \le \infty$.
\item $A(\infty, \epsilon) \subset \liminf_{k \to \infty}A(j(k), \epsilon)$.
\item We have
\[\lim_{r \to 0}\left(\limsup_{l \to \infty}\left|\frac{1}{\upsilon_{k(l)}(B_r(z_{k(l)}))}\int_{B_r(z_{k(l)})}f_{k(l)}d\upsilon_{k(l)}-\frac{1}{\upsilon_{\infty}(B_r(z_{\infty}))}\int_{B_r(z_{\infty})}f_{\infty}d\upsilon_{\infty}\right|\right)=0\]
for any subsequence $\{k(l)\}_l$ of $\{j(k)\}_k$ and  convergent sequence $\{z_{k(l)}\}_{l \le \infty}$ of $z_{k(l)} \in A(k(l), \epsilon)$.
\end{enumerate}
\end{definition}
A fundamental property of this convergence is the following `linearity':
\begin{proposition}\label{li}
Let $\{g_i\}_{i \le \infty}$ be a sequence of $g_i \in L^1_{\mathrm{loc}}(B_R(x_i))$.
Then we have the following:
\begin{enumerate}
\item If $f_i$ w-converges to $f_{\infty}$ on $B_R(x_{\infty})$, then for every subsequence $\{i(j)\}_j$ of $\mathbf{N}$, there exists a subsequence $\{j(k)\}_k$ of $\{i(j)\}_j$ such that $f_{j(k)}$ $L^1_{\mathrm{loc}}$-converges to $f_{\infty}$ at a.e. $z_{\infty} \in B_R(x_{\infty})$.
\item If $f_{i}, g_i$ w-converge to $f_{\infty}, g_{\infty}$ on $B_R(x_{\infty})$, respectively,
then we see that $a_if_i+b_ig_i$ w-converges to $a_{\infty}f_{\infty}+b_{\infty}g_{\infty}$ on $B_R(x_{\infty})$ for any convergent sequences $\{a_i\}_{i \le \infty}, \{b_i\}_{i \le \infty}$ in $\mathbf{R}$.
\end{enumerate}
\end{proposition}
\begin{proof}
We first check (1).
Let $\{\epsilon_l\}_l$ be a sequence of positive numbers with $\epsilon_l \to 0$, and let $\{i(j)\}_j$ be a subsequence of $\mathbf{N}$.
There exist a subsequence $\{i_1(j)\}_j$ of $\{i(j)\}_j$ and a sequence $\{A(i_1(j), \epsilon_1)\}_{j \le \infty}$ as in Definition \ref{ll}.
Applying Definition \ref{ll} for $\{i_1(j)\}_j$ and $\epsilon_2$ yields that there exist a subsequence $\{i_2(j)\}_j$ of $\{i_1(j)\}_j$ and a sequence $\{A(i_2(j), \epsilon_2)\}_{j \le \infty}$ as in Definition \ref{ll}.
By iterating this argument we construct sequences of $\{i_k(j)\}_j$ and $\{A(i_k(j), \epsilon_k)\}_{j \le \infty}$ for every $k \ge 1$.

Let us consider a subsequence $\{i_l(l)\}_l$ and a Borel subset 
\[A:=\bigcup_k A(\infty, \epsilon_k).\]
Note that $\upsilon_{\infty}(B_R(x_{\infty}) \setminus A)=0$. 
Since $\{i_l(l)\}_l$ is a subsequence of $\{i_k(j)\}_j$ for every $k$, (2) of Definition \ref{ll} yields
\[A(\infty, \epsilon_k) \subset \liminf_{l \to \infty}A(i_l(l), \epsilon_k)\]
for every $k \ge 1$.
Thus for every $z_{\infty} \in A$ there exist $k \ge 1$ and a sequence $\{z_{i_l(l)}\}_l$ of $z_{i_l(l)} \in A(i_l(l), \epsilon_k)$ such that $z_{i_l(l)} \stackrel{GH}{\to} z_{\infty}$.
Then (3) of Definition \ref{ll} yields that $f_{i_l(l)}$ $L^1_{\mathrm{loc}}$-weakly converges to $f_{\infty}$ at a.e. $z_{\infty} \in B_R(x_{\infty})$.
This completes the proof of (1).

Next we prove $(2)$.

Let $\epsilon>0$ and let $\{i(j)\}_j$ be a subsequence of $\mathbf{N}$.
Then there exist a subsequence $\{j(k)\}_k$ of $\{i(j)\}_j$, and sequences $\{A_l(j(k), \epsilon)\}_{1 \le l \le 2, k \le \infty}$ of compact subsets $A_l(j(k), \epsilon)$ of $B_R(x_{j(k)})$ such that $\upsilon_{j(k)}(B_R(x_{j(k)}) \setminus A_l(j(k), \epsilon))\le \epsilon/4$ for any $k \le \infty$ and $l$, that $A_l(\infty, \epsilon) \subset \lim_{k \to \infty}A_l(j(k), \epsilon)$ for every $l$, and that 
\[\lim_{r \to 0}\left(\limsup_{k \to \infty}\left|\frac{1}{\upsilon_{j(k)}(B_r(z_{j(k)}))}\int_{B_r(z_{j(k)})}f_{j(k)}d\upsilon_{j(k)}-\frac{1}{\upsilon_{\infty}(B_r(z_{\infty}))}\int_{B_r(z_{\infty})}f_{\infty}d\upsilon_{\infty}\right|\right)=0\]
for any $l$ and  convergent sequence $\{z_{j(k)}\}_{k \le \infty}$ of $z_{j(k)} \in A_l(j(k), \epsilon)$.

Let $\hat{A}(j(k), \epsilon):=A_1(j(k), \epsilon) \cap A_2(j(k), \epsilon)$ for every $k \le \infty$.
Without loss of generality we can assume that the limit $\lim_{k \to \infty}\hat{A}(j(k), \epsilon) ( \subset \overline{B}_R(x_{\infty}))$ exists.
Note that
\[\upsilon_{j(k)}\left(B_R(x_{j(k)})\setminus \hat{A}(j(k), \epsilon)\right) \le \epsilon/2\]
for every $k \le \infty$.
Let 
\[A(j(k), \epsilon):= \begin{cases}
\hat{A}(j(k), \epsilon) & (k <\infty)\\
\lim_{l \to \infty}\hat{A}(j(l), \epsilon) \cap \hat{A}(\infty, \epsilon) & (k=\infty).
\end{cases}
\]
Then \cite[Proposition $2.3$]{holp} yields
\begin{align*}
&\upsilon_{\infty}\left( \overline{B}_R(x_{\infty}) \setminus A(\infty, \epsilon) \right) \\
&\le \upsilon_{\infty}\left( \overline{B}_R(x_{\infty}) \setminus \lim_{k \to \infty}\hat{A}(j(k), \epsilon) \right) + \upsilon_{\infty}\left(  \overline{B}_R(x_{\infty}) \setminus \hat{A}(\infty, \epsilon) \right)\\ 
& \le \liminf_{k \to \infty}\upsilon_{j(k)}\left(\overline{B}_R(x_{j(k)}) \setminus \hat{A}(j(k), \epsilon)\right) + \epsilon/2 \le \epsilon.
\end{align*}
Since it is easy to check 
\begin{align*}
\lim_{r \to 0}\Bigl(\limsup_{l \to \infty}\Bigl|\frac{1}{\upsilon_{k(l)}(B_r(z_{k(l)}))}\int_{B_r(z_{k(l)})}(a_{k(l)}f_{k(l)}+b_{k(l)}g_{k(l)})d\upsilon_{j(k)}\\
-\frac{1}{\upsilon_{\infty}(B_r(z_{\infty}))}\int_{B_r(z_{\infty})}(a_{\infty}f_{\infty}+b_{\infty}g_{\infty})d\upsilon_{\infty}\Bigl|\Bigl)=0
\end{align*}
for any subsequence $\{k(l)\}_l$ of $\{j(k)\}_k$, and convergent sequence $\{z_{k(l)}\}_{l\le \infty}$ of $z_{k(l)} \in A(k(l), \epsilon)$,
this completes the proof.
\end{proof}
The following is a key result to prove Theorem \ref{yamayama}.
\begin{corollary}\label{fatou2}
Let $\{g_i\}_{i \le \infty}$ be a sequence of $g_i \in L^1(B_R(x_i))$.
Assume that the following four conditions hold:
\begin{enumerate}
\item $f_i \in L^1(B_R(x_i))$ for every $i \le \infty$.
\item $f_{i}, g_i$ w-converge to $f_{\infty}, g_{\infty}$ on $B_R(x_{\infty})$, respectively.
\item $f_i \le g_i$ on $B_R(x_{\infty})$ for every $i<\infty$.
\item We have
\[\lim_{i \to \infty}\int_{B_R(x_i)}g_id\upsilon_i=\int_{B_R(x_{\infty})}g_{\infty}d\upsilon_{\infty}.\]
\end{enumerate}
Then we have 
\[\limsup_{i \to \infty}\int_{B_R(x_i)}f_id\upsilon_i \le \int_{B_R(x_{\infty})}f_{\infty}d\upsilon_{\infty}.\]
\end{corollary}
\begin{proof}
Let $\{i(j)\}_j$ be a subsequence of $\mathbf{N}$.
Proposition \ref{li} gives that there exists a subsequence $\{j(k)\}_k$ of $\{i(j)\}_j$ such that $g_{j(k)}-f_{j(k)}$ $L^1_{\mathrm{loc}}$-weakly converges $g_{\infty}-f_{\infty} (\ge 0)$ at a.e. $z_{\infty} \in B_R(x_{\infty})$. 
Thus since $\{i(j)\}_j$ is arbitrary, applying Proposition \ref{fatou1} to $\{g_{j(k)}-f_{j(k)}\}_{k \le \infty}$ completes the proof.
\end{proof}
The following is an important example of the $w$-convergence.
\begin{theorem}\label{nn}
Let $\{q_i\}_{i \le \infty}$ and $\{p_i\}_{i \le \infty}$ be convergent sequences in $(1, \infty)$, let $\{f_i\}_{i<\infty}$ be a sequence of $f_i \in H^{1, p_i}(B_R(x_i))$ with $\sup_{i<\infty}||f_i||_{H^{1, p_i}}<\infty$, and let $f_{\infty}$ be the $\{L^{p_i}\}_i$-weak limit on $B_R(x_{\infty})$ of $\{f_i\}_i$.
Assume that there exists $\tau>0$ such that for every $i<\infty$, $(X_i, \upsilon_i)$ satisfies the $(q_i, p_i)$-Poincar\'e inequality on $B_R(x_i)$ for $\tau$. 
Then we see that $|f_i|^{q_i}$ $w$-converges $|f_{\infty}|^{q_{\infty}}$ on $B_R(x_{\infty})$.
\end{theorem}
\begin{proof}
We use the following notation for convenience:
\begin{itemize}
\item Let us denote by $\Psi (\epsilon_1, \epsilon _2 ,\ldots ,\epsilon_k  ; c_1, c_2, \ldots, c_l)$
some positive valued function on $\mathbf{R}_{>0}^k \times \mathbf{R}^l $ satisfying 
\begin{align}\label{89890}
\lim_{\epsilon_1, \epsilon_2,\ldots ,\epsilon_k \to 0}\Psi (\epsilon_1, \epsilon_2,\ldots ,\epsilon_k  ; c_1, c_2 ,\ldots ,c_l)=0
\end{align} 
for fixed real numbers $c_1, c_2,\ldots ,c_l$.
\item For any $a, b \in \mathbf{R}$ and $\epsilon>0$,
\begin{align}\label{33we}
a=b \pm \epsilon \Longleftrightarrow |a-b|<\epsilon.
\end{align}
\end{itemize}
Note that by Theorem \ref{stapo} we see that $(X_{\infty}, \upsilon_{\infty})$ satisfies the $(q_{\infty}, p_{\infty})$-Poincar\'e inequality on $B_R(x_{\infty})$ for $\tau$.

Let $\epsilon>0$, let $L:=\sup_i||f_i||_{H^{1, p_i}}$ and let $\hat{A}(i, \epsilon)$ be the set of $w_i \in \overline{B}_{R-\epsilon}(x_i)$ with
\[\frac{1}{\upsilon_i(B_r(w_i))}\int_{B_r(w_i)}|\nabla f_i|^{p_i}d\upsilon_i \le \epsilon^{-p_i}\]
for every $0<r<\epsilon$. 
It is easy to check that $\hat{A}(i, \epsilon)$ is compact.
Then (the proof of) \cite[Lemma $3.1$]{holip} yields 
\[\upsilon_i(B_R(x_i) \setminus \hat{A}(i, \epsilon))\le \Psi(\epsilon; n, K, d, L, R)\]
for every $i \le \infty$.

Note that the $(q_i, p_i)$-Poincar\'e inequality yields
\[\frac{1}{\upsilon_i(B_r(z))}\int_{B_r(z)}|f_i|d\upsilon_i=\left(\frac{1}{\upsilon_i(B_r(z))}\int_{B_r(z)}|f_i|^{q_i}d\upsilon_i\right)^{1/q_i} \pm \tau r\epsilon^{-1}\]
for any $i\le \infty$, $z \in \hat{A}(i, \epsilon)$ and $0<r<\epsilon$.

Let $\{i(j)\}_j$ be a subsequence of $\mathbf{N}$.
There exists a subsequence $\{j(k)\}_k$ of $\{i(j)\}_j$ such that the limit  $\lim_{k \to \infty}\hat{A}(j(k), \epsilon)$ $(\subset \overline{B}_{R-\epsilon}(x_{\infty}))$ exists.
Let 
\[A(j(k), \epsilon):= \begin{cases}
\hat{A}(j(k), \epsilon) & (k <\infty)\\
\lim_{l \to \infty}\hat{A}(j(l), \epsilon) \cap \hat{A}(\infty, \epsilon) & (k=\infty).
\end{cases}
\]
Note that by \cite[Proposition $2.3$]{holip} we have
\[\upsilon_{\infty}\left(\overline{B}_R(x_{\infty}) \setminus A(\infty, \epsilon)\right)\le \Psi(\epsilon; n, K, d, L, R).\]

Let $\{k(l)\}_l$ be a subsequence of $\{j(k)\}_k$ and let $\{z_{k(l)}\}_{l \le \infty}$ be a convergent sequence of $z_{k(l)} \in A(k(l), \epsilon)$.

Note that Theorem \ref{tt} yields that $f_i$ $L^s$-converges strongly to $f_{\infty}$ on $B_R(x_{\infty})$ for every $s<q_{\infty}$.
Thus for every $r<\epsilon$ we have
\begin{align*}
&\left(\frac{1}{\upsilon_{\infty}(B_r(z_{\infty}))}\int_{B_r(z_{\infty})}|f_{\infty}|^{q_{\infty}}d\upsilon_{\infty}\right)^{1/q_{\infty}} \\
&= \frac{1}{\upsilon_{\infty}(B_r(z_{\infty}))}\int_{B_r(z_{\infty})}|f_{\infty}|d\upsilon_{\infty} \pm \tau r \epsilon^{-1} \\
&=\frac{1}{\upsilon_{k(l)}(B_r(z_{k(l)}))}\int_{B_r(z_{k(l)})}|f_{k(l)}|d\upsilon_{k(l)} \pm \tau r \epsilon^{-1} \\
&=\left(\frac{1}{\upsilon_{k(l)}(B_r(z_{k(l)}))}\int_{B_r(z_{k(l)})}|f_{k(l)}|^{q_{k(l)}}d\upsilon_{k(l)}\right)^{1/q_{k(l)}} \pm 2\tau r\epsilon^{-1}\\
\end{align*}
for every sufficiently large $l$.
This completes the proof.
\end{proof}
\section{Poisson's equations}
This section is devoted to the proofs of the results stated in subsection $1.1$. 
We also give related results.
In order to prove Theorem \ref{pois} we start this section by giving the following:
\begin{theorem}\label{convlap}
Let $(X_i, \upsilon_i) \stackrel{GH}{\to} (X_{\infty}, \upsilon_{\infty})$ in $\overline{M(n, K, d)}$ with $\mathrm{diam}\,X_{\infty}>0$, and let $\{f_i\}_{i<\infty}$ be a sequence of $f_i \in \mathcal{D}^2(\Delta^{\upsilon_i}, X_i)$ with 
\[\sup_{i<\infty}\int_{X_i}\left(|f_i|^2+|\Delta^{\upsilon_i}f_i|^2\right)d\upsilon_i<\infty.\]
Then there exist $f_{\infty} \in \mathcal{D}^2(\Delta^{\upsilon_{\infty}}, X_{\infty})$ and a subsequence $\{f_{i(j)}\}_j$ such that 
$f_{i(j)}, \nabla f_{i(j)}$ $L^2$-converge strongly to $f_{\infty}, \nabla f_{\infty}$ on $X_{\infty}$, respectively and that $\Delta^{\upsilon_{i(j)}}f_{i(j)}$ $L^2$-converges weakly to $\Delta^{\upsilon_{\infty}}f_{\infty}$ on $X_{\infty}$.
\end{theorem}
\begin{proof}
Note that since
\[\int_{X_i}|\nabla f_i|^2d\upsilon_i=\int_{X_i}f_i\Delta^{\upsilon_i}f_id\upsilon_i \le \left(\int_{X_i}|f_i|^2d\upsilon_i\right)^{1/2}\left(\int_{X_i}|\Delta^{\upsilon_i}f_i|^2d\upsilon_i\right)^{1/2},\]
we have $\sup_i||f_i||_{H_{1, 2}}<\infty$.

Thus by Theorem \ref{srell} without loss of generality we can assume that there exists $f_{\infty} \in H^{1, 2}(X_{\infty})$ such that $f_i$ $L^2$-converges strongly to $f_{\infty}$ on $X_{\infty}$ and that $\nabla f_i$ $L^2$-converges weakly to $\nabla f_{\infty}$ on $X_{\infty}$.
Then \cite[Theorem $4.1$]{holp} yields that $f_{\infty} \in \mathcal{D}^2(\Delta^{\upsilon_{\infty}}, X_{\infty})$ and that $\Delta^{\upsilon_i}f_i$ $L^2$-converges weakly to $\Delta^{\upsilon_{\infty}}f_{\infty}$ on $X_{\infty}$.
Thus it suffices to check that $\nabla f_i$ $L^2$-converges strongly to $\nabla f_{\infty}$ on $X_{\infty}$.

Let $g_i:=\Delta^{\upsilon_i}f_i$.
\begin{claim}\label{abc}
We have 
\begin{align*}
\int_{X_i}|\nabla h|^2d\upsilon_i-2\int_{X_i}g_ihd\upsilon_i \ge \int_{X_i}|\nabla f_i|^2d\upsilon_i-2\int_{X_i}g_if_id\upsilon_i
\end{align*}
for any $i\le \infty$ and $h \in H^{1, 2}(X_i)$.
\end{claim}
It follows from the identity
\[\int_{X_i}|\nabla h|^2d\upsilon_i-2\int_{X_i}g_ihd\upsilon_i = \int_{X_i}|\nabla f_i|^2d\upsilon_i-2\int_{X_i}g_if_id\upsilon_i + \int_{X_i}|\nabla (f_i-h)|^2d\upsilon_i.\]
By Theorem \ref{fundpr} and \cite[Theorem $4.2$]{holip} without loss of generality we can assume that there exists a sequence $\{h_{i}\}_{i <\infty}$ of $h_{i} \in H^{1, 2}(X_i)$ such that $h_i, \nabla h_i$ $L^2$-converge strongly to $f_{\infty}, \nabla f_{\infty}$ on $X_{\infty}$, respectively.
In particular we have
\begin{align}\label{poij}
\lim_{i \to \infty}\int_{X_i}g_ih_id\upsilon_i=\lim_{i \to \infty}\int_{X_i}g_if_id\upsilon_i=\int_{X_{\infty}}g_{\infty}f_{\infty}d\upsilon_{\infty}.
\end{align}
By Claim \ref{abc}, since
\begin{align}\label{dfg}
\int_{X_i}|\nabla h_i|^2d\upsilon_i-2\int_{X_i}g_ih_id\upsilon_i \ge \int_{X_i}|\nabla f_i|^2d\upsilon_i-2\int_{X_i}g_if_id\upsilon_i,
\end{align}
letting $i \to \infty$ in (\ref{dfg}) with (\ref{poij}) yields
\[\limsup_{i \to \infty}\int_{X_i}|\nabla f_i|^2d\upsilon_i \le \int_{X_{\infty}}|\nabla f_{\infty}|^2d\upsilon_{\infty}.\]
This completes the proof.
\end{proof}
We are now in a position to prove Theorem \ref{pois}.

\textit{Proof of Theorem \ref{pois}.}

We first prove $(1)$.

Note that for every $f \in \mathcal{D}^2(\Delta^{\upsilon}, X)$ we have
\[\int_{X}\Delta^{\upsilon}fd\upsilon=\int_{X}\langle \nabla 1, \nabla f\rangle d\upsilon=0.\]
Thus it suffices to check `if' part of $(1)$.

Let $\hat{H}^{1, 2}(X)$ be the closed subspace of $f \in H^{1, 2}(X)$ with
\[\int_{X}fd\upsilon=0.\]
By the $(2, 2)$-Poincar\'e inequality we have
\[\int_{X}|f|^2d\upsilon \le C(n, K, d)\int_{X}|\nabla f|^2d\upsilon\]
for every $f \in \hat{H}^{1, 2}(X)$.
Thus we see that
$||f||_{\hat{H}^{1, 2}}:=||\nabla f||_{L^2}$
is a norm on $\hat{H}^{1, 2}(X)$ and that $||\cdot||_{H^{1, 2}}$ and $||\cdot||_{\hat{H}^{1, 2}}$ are equivalent on $\hat{H}^{1, 2}(X)$.
Therefore the Riesz representation theorem yields that there exists a unique $f \in \hat{H}^{1, 2}(X)$ such that 
\begin{align*}\label{polll}
\int_{X}\langle \nabla f, \nabla h \rangle d\upsilon=\int_{X}ghd\upsilon
\end{align*}
holds for every $h \in \hat{H}^{1, 2}(X)$.
Let $h \in H^{1, 2}(X)$ and let
\[\hat{h}:=h-\int_{X}hd\upsilon.\]
Then since $\hat{h} \in \hat{H}^{1, 2}(X)$ we have
\begin{align*}
\int_{X}\langle \nabla f, \nabla h \rangle d\upsilon = \int_{X}\langle \nabla f, \nabla \hat{h} \rangle d\upsilon =\int_{X}g\hat{h}d\upsilon = \int_{X}ghd\upsilon.
\end{align*}
This completes the proof of $(1)$. 
Note that the argument above also gives $(2)$.

We next prove $(3)$.

Let $f_i:= (\Delta^{\upsilon_i})^{-1}g_i$.
The $(2, 2)$-Poincar\'e inequality yields
\begin{align}\label{star}
\int_{X_i}|f_i|^2d\upsilon_i &\le C(n, K, d)\int_{X_i}|\nabla f_i|^2d\upsilon_i \nonumber \\
&=  C(n, K, d)\int_{X_i}f_ig_id\upsilon_i \nonumber \\
&\le C(n, K, d)\left(\int_{X_i}|f_i|^2d\upsilon_i \right)^{1/2}\left(\int_{X_i}|g_i|^2d\upsilon_i \right)^{1/2}.
\end{align}
This gives $\sup_i||f_i||_{H^{1, 2}}<\infty$.

Thus by Theorem \ref{srell} without loss of generality we can assume that there exists $\hat{f}_{\infty} \in H^{1, 2}(X_{\infty})$ such that 
$f_i$ $L^2$-converges strongly to $\hat{f}_{\infty}$ on $X_{\infty}$ and that $\nabla f_i$ $L^2$-converges weakly to $\nabla \hat{f}_{\infty}$ on $X_{\infty}$.
In particular we have
\[\int_{X_{\infty}}\hat{f}_{\infty}d\upsilon_{\infty}=\lim_{i \to \infty}\int_{X_i}f_id\upsilon_i=0.\]
On the other hand, Theorem \ref{convlap} yields $\hat{f}_{\infty} \in \mathcal{D}^2(\Delta^{\upsilon_{\infty}}, X_{\infty})$ and  $\Delta^{\upsilon_{\infty}}\hat{f}_{\infty}=\Delta^{\upsilon_{\infty}}f_{\infty}$.
Thus $(2)$ yields $f_{\infty}=\hat{f}_{\infty}$.
This completes the proof of $(3)$.  $\,\,\,\,\,\,\,\,\,\,\,\,\,\,\,\,\,\,\,\,\Box$

We are now in a position to prove Theorem \ref{app6}.

\textit{Proof of Theorem \ref{app6}}

We only check the assertion in the case that $(X_i, \upsilon_i) \in M(n, K, d)$ for every $i<\infty$ only because the proof of the other case is similar.

Let $\epsilon>0$ and let $\{f_{\infty, l}, g_{\infty, l}\}_{1 \le l \le k}$ be a collection in $\mathrm{Test}F(X_{\infty})$ with
\[||\omega_{\infty}-\sum_{l=1}^kf_{\infty, l}dg_{\infty, l}||_{W^{1, 2}_H(X_{\infty})}<\epsilon.\]
Let $\{G_{\infty, l, m}\}_m$ be a sequence in $\mathrm{LIP}(X_{\infty})$ with $G_{\infty, l, m} \to \Delta^{\upsilon_{\infty}}g_{\infty, l}$ in $H^{1, 2}(X_{\infty})$ as $m \to \infty$ and
\[\int_{X_{\infty}}G_{\infty, l, m}d\upsilon_{\infty}=0.\]
Let $g_{\infty, l, m} := (\Delta^{\upsilon_{\infty}})^{-1}G_{\infty, l, m}$.
Theorem \ref{pois} with (\ref{lipreg}) yields that $g_{\infty, l, m} \in \mathrm{Test}F(X_{\infty})$ and that $g_{\infty, l, m} \to g_{\infty, l}$ in $H^{1, 2}(X_{\infty})$ as $m \to \infty$.

Note that by \cite[Remark $4.3$]{holp} (or \cite[($2.3.13$)]{gigli})
\begin{align}\label{new}
\delta^{\upsilon_{\infty}}(f_{\infty, l}d g_{\infty, l, m})=\langle df_{\infty, l}, dg_{\infty, l, m} \rangle-f_{\infty, l}\Delta^{\upsilon_{\infty}}g_{\infty, l, m}
\end{align}
and that \cite[Theorem $3.5.2$]{gigli} yields 
\begin{align}\label{8670}
d^{\upsilon_{\infty}}(f_{\infty, l}d g_{\infty, l, m})=df_{\infty, l} \wedge d g_{\infty, l, m}.
\end{align}
Thus we see that $f_{\infty, l}d g_{\infty, l, m} \to f_{\infty, l}d g_{\infty, l}$ in $W^{1, 2}_H(T^*X_{\infty})$ as $m \to \infty$.

Let $m \in \mathbf{N}$ with
\[||\sum_{l=1}^kf_{\infty, l}dg_{\infty, l}-\sum_{l=1}^kf_{\infty, l}dg_{\infty, l, m}||_{W_H^{1, 2}(X_{\infty})}<\epsilon.\]
By \cite[Theorem 4.2]{holip} without loss of generality we can assume that 
there exist sequences $\{f_{i, l}, G_{i, l, m}\}_{i<\infty, 1 \le l \le k}$ of $f_{i, l}, G_{i, l, m} \in \mathrm{LIP}(X_{i})$ such that the following hold:
\begin{itemize}
\item $\sup_{i, l}\left(\mathbf{Lip}f_{i, l} +\mathbf{Lip}G_{i, l, m}\right)<\infty$.
\item For any $i, l$,
\[\int_{X_i}G_{i, l, m}d\upsilon_i=0.\]
\item $f_{i, l}, \nabla f_{i, l}, G_{i, l, m}, \nabla G_{i, l, m}$ $L^2$-converge strongly to $f_{\infty, l}, \nabla f_{\infty, l},  G_{\infty, l, m}, \nabla G_{\infty, l, m}$ on $X_{\infty}$, respectively for every $l$.
\end{itemize}
Moreover by the smoothing via the heat flow (c.f. \cite{ags, grigo}) without loss of generality we can assume that $f_{i, l}, G_{i, l, m} \in C^{\infty}(X_{i})$.

Let $g_{i, l, m} := (\Delta^{\upsilon_i})^{-1}G_{i, l, m} \in C^{\infty}(X_i)$ (note that the smoothness of $g_{i, l, m}$ follows from the elliptic regularity theorem).
Note that Theorem \ref{pois} with (\ref{lipreg}) implies that $\sup_{i<\infty, l}\mathbf{Lip}g_{i, l, m}<\infty$ and that $g_{i, l, m}, \nabla g_{i, l, m}$ $L^2$-converge strongly to $g_{\infty, l, m}, \nabla g_{\infty, l, m}$ on $X_{\infty}$, respectively.

For $i \le \infty$, let 
\[\omega_{i, m}:=\sum_{l=1}^kf_{i, l}dg_{i, l, m}.\]
Then by (\ref{new}) and (\ref{8670}), $\omega_{i, m}, d\omega_{i, m}, \delta \omega_{i, m}$ $L^2$-converge strongly to $\omega_{\infty, m}, d^{\upsilon_{\infty}}\omega_{\infty, m}, \delta^{\upsilon_{\infty}} \omega_{\infty, m}$ on $X_{\infty}$, respectively.
This completes the proof. $\,\,\,\,\,\,\,\,\,\,\,\,\,\,\,\,\Box$ 

Similarly we have the following:
\begin{theorem}\label{kkii}
Let $(X_i, \upsilon_i) \stackrel{GH}{\to} (X_{\infty}, \upsilon_{\infty})$ in $\overline{M(n, K, d)}$ and let $\omega_{\infty} \in H^{1, 2}_d(\bigwedge^kT^*X_{\infty})$.
Then there exist a subsequence $\{i(j)\}_j$ and a sequence $\{\omega_{i(j)}\}_j$ of $\omega_{i(j)} \in \mathrm{Test}\mathrm{Form}_kX_{i(j)}$ such that $\omega_{i(j)}, d^{\upsilon_{i(j)}}\omega_{i(j)}$ $L^2$-converge strongly to $\omega_{\infty}, d^{\upsilon_{\infty}}\omega_{\infty}$ on $X_{\infty}$, respectively.
Moreover if $(X_{i}, \upsilon_{i}) \in M(n, K, d)$ for every $i<\infty$, then we can choose $\{\omega_{i(j)}\}_j$ as $C^{\infty}$-differential $k$-forms.
\end{theorem}
\begin{corollary}\label{contiad}
Let $(X_i, \upsilon_i) \stackrel{GH}{\to} (X_{\infty}, \upsilon_{\infty})$ in $\overline{M(n, K, d)}$ with $\mathrm{diam}\,X_{\infty}>0$, and let 
$\{\omega_i\}_{i\le \infty}$ be an $L^2$-weak convergent sequence on $X_{\infty}$ of $\omega_i \in L^2(\bigwedge^kT^*X_i)$.
Assume that $\omega_i \in \mathcal{D}^2(\delta^{\upsilon_i}, X_i)$ for every $i<\infty$ and that $\sup_{i<\infty}||\delta^{\upsilon_i}\omega_i||_{L^2}<\infty$.
Then we see that $\omega_{\infty} \in \mathcal{D}^2(\delta^{\upsilon_{\infty}}, X_{\infty})$ and that $\delta^{\upsilon_i}\omega_i$ $L^2$-converges weakly to $\delta^{\upsilon_{\infty}}\omega_{\infty}$ on $X_{\infty}$.
\end{corollary}
\begin{proof}
By the $L^2$-weak compactness, for every subsequence $\{i(j)\}_j$ there exist a subsequence $\{j(l)\}_l$ of $\{i(j)\}_j$ and the $L^2$-weak limit $\eta_{\infty} \in L^2(\bigwedge^{k-1}T^*X_{\infty})$  on $X_{\infty}$ of $\{\delta^{\upsilon_{j(l)}}\omega_{j(l)}\}_{l}$.

Let $\alpha_{\infty} \in \mathrm{Test}\mathrm{Form}_{k-1}(X_{\infty})$.
By Theorem \ref{kkii} without loss of generality we can assume that there exists a sequence $\{\alpha_{j(l)}\}_l$ of $\alpha_{j(l)} \in \mathrm{Test}\mathrm{Form}_{k-1}(X_{j(l)})$ such that $\alpha_{j(l)}, d^{\upsilon_{j(l)}}\alpha_{j(l)}$ $L^2$-converge strongly to $\alpha_{\infty}, d^{\upsilon_{\infty}}\alpha_{\infty}$ on $X_{\infty}$, respectively.

Since
\[\int_{X_{j(l)}}\langle \delta^{\upsilon_{j(l)}}\omega_{j(l)}, \alpha_{j(l)}\rangle d\upsilon_{j(l)}=\int_{X_{j(l)}}\langle \omega_{j(l)}, d^{\upsilon_{j(l)}}\alpha_{j(l)}\rangle d\upsilon_{j(l)}\]
for every $l<\infty$,
by letting $l \to \infty$ we have
\[\int_{X_{\infty}}\langle \eta_{\infty}, \alpha_{\infty}\rangle d\upsilon_{\infty}=\int_{X_{\infty}}\langle \omega_{\infty}, d^{\upsilon_{\infty}}\alpha_{\infty} \rangle d\upsilon_{\infty}.\]
In particular we see that $\omega_{\infty} \in \mathcal{D}^2(\delta^{\upsilon_{\infty}}, X_{\infty})$ and that $\delta^{\upsilon_{\infty}}\omega_{\infty}=\eta_{\infty}$.
Since $\{i(j)\}_j$ is arbitrary, this completes the proof. 
\end{proof}
In order to prove Theorem \ref{app3}, we establish the following key result:
\begin{theorem}\label{nju}
Let $(X, x, \upsilon) \in \overline{M(n, K)}$, let $R>0$ and let $\phi \in \mathcal{D}^2(\Delta^{\upsilon}, B_R(x)) \cap \mathrm{LIP}_c(B_R(x))$ with $\Delta^{\upsilon}\phi \in L^{\infty}(B_R(x))$.
Then 
for every $f \in \mathcal{D}^2(\Delta^{\upsilon}, B_R(x))$ we see that $\phi f \in \mathcal{D}^2(\Delta^{\upsilon}, X)$ and that  $\Delta^{\upsilon}(\phi f)=\phi \Delta^{\upsilon}f-2\langle \nabla \phi, \nabla f \rangle +f\Delta^{\upsilon}\phi$ in $L^2(X)$.
\end{theorem}
\begin{proof}
Let $f \in \mathcal{D}^2(\Delta^{\upsilon}, B_R(x))$.
We first prove $\phi f \in H^{1, 2}(X)$.

Let $r<R$ with $\mathrm{supp}\,\phi \subset B_r(x)$, and let $\phi \equiv 0$ on $X \setminus B_R(x)$.
By Theorem \ref{fundpr} there exists a sequence $\{f_{j}\}_j$ of  $f_{j} \in \mathrm{LIP}_{\mathrm{loc}}(B_R(x))$ with $f_{j} \to f$ in $H^{1, 2}(B_R(x))$. 
Note that $\phi f_{j} \in \mathrm{LIP}_{\mathrm{loc}}(X)$ and that since $\mathrm{supp}\,(\phi f_{j}) \subset B_r(x)$ we have
\begin{align*}
\mathrm{Lip}(\phi f_{j})(y)&\le |\phi (y)|\mathrm{Lip}f_{j}(y) + |f_{j}(y)|\mathrm{Lip}\phi (y) \\
&\le ||\phi ||_{L^{\infty}(B_r(x))}\mathbf{Lip}(f_{j}|_{B_r(x)}) + ||f_{j}||_{L^{\infty}(B_r(y))}\mathbf{Lip}\phi
\end{align*}
for every $y \in X$.
Therefore we have $\phi f_{j} \in \mathrm{LIP}(X)$.

Similarly we have
\[\int_{X}|\nabla (\phi f_{j})|^2d\upsilon \le 2||\phi ||_{L^{\infty}(B_r(x))}^2\int_{B_R(x)}|\nabla f_{j}|^2d\upsilon + 2(\mathbf{Lip}\phi )^2\int_{B_R(x)}|f_{j}|^2d\upsilon.\]
In particular we have $\sup_j ||\phi f_{j}||_{H^{1, 2}(X)}<\infty$.
Since $\phi f_{j} \to \phi f$ in $L^2(X)$, we have $\phi f \in H^{1, 2}(X)$.

Thus it suffices to check that for every $h \in \mathrm{LIP}_c(X)$ we have 
\begin{align}\label{counte}
\int_{X}\langle \nabla h, \nabla (\phi f)\rangle d\upsilon =\int_{X}h \left( \phi \Delta^{\upsilon}f-2\langle \nabla \phi, \nabla f \rangle +f\Delta^{\upsilon}\phi \right) d\upsilon.
\end{align}
Since $h \phi \in \mathrm{LIP}_c(B_R(x))$, we have
\begin{align}\label{firs}
\int_{X}h \phi \Delta^{\upsilon}fd\upsilon &=\int_{B_R(x)}h \phi \Delta^{\upsilon}f d\upsilon \nonumber \\
&=\int_{B_R(x)}\langle \nabla (h \phi ), \nabla f \rangle d\upsilon \nonumber \\
&=\int_{B_R(x)}\phi \langle \nabla h, \nabla f \rangle d\upsilon + \int_{B_R(x)}h \langle \nabla \phi, \nabla f \rangle d\upsilon \nonumber \\
&=\int_{X}\phi \langle \nabla h, \nabla f\rangle d\upsilon + \int_{X}h \langle \nabla \phi, \nabla f \rangle d\upsilon.
\end{align}
We need the following elementary claim:
\begin{claim}\label{jhhhf}
Let $F \in \mathcal{D}^2(\Delta^{\upsilon}, B_R(x))$ with compact support, i.e., there exists $r>0$ with $r<R$ such that $F \equiv 0$ on $B_R(x) \setminus B_r(x)$.
Then for every $G \in H^{1, 2}(B_R(x))$ we have
\[\int_{B_R(x)}\langle dG, dF\rangle d\upsilon = \int_{B_R(x)}G\Delta^{\upsilon}Fd\upsilon.\]
\end{claim}
The proof is as follows.

Let $\psi \in \mathrm{LIP}_c(B_R(x))$ with $\psi \equiv 1$ on $B_r(x)$.
Then since $\langle dG, dF \rangle = \langle d(\psi G), dF\rangle$ on $B_R(x)$, we have
\[\int_{B_R(x)}\langle dG, dF \rangle d\upsilon =\int_{B_R(x)}\langle d(\psi G), dF\rangle d\upsilon = \int_{B_R(x)}\psi G\Delta^{\upsilon}Fd\upsilon. \]
On the other hand, since $\Delta^{\upsilon}F \equiv 0$ on $B_R(x) \setminus \overline{B}_r(x)$ because $F \equiv 0$ on there, we have
\[\int_{B_R(x)}\psi G\Delta^{\upsilon}Fd\upsilon = \int_{B_R(x)}G\Delta^{\upsilon}Fd\upsilon.\]
This completes the proof of Claim \ref{jhhhf}.

By Claim \ref{jhhhf} we have
\begin{align}\label{seco}
\int_{X}fh \Delta^{\upsilon}\phi d\upsilon &=\int_{B_R(x)}fh\Delta^{\upsilon}\phi d\upsilon \nonumber \\
&=\int_{B_R(x)}\langle \nabla (fh), \nabla \phi \rangle d\upsilon  \nonumber \\
&=\int_{B_R(x)}h\langle \nabla f, \nabla \phi \rangle d\upsilon + \int_{B_R(x)}f\langle \nabla h, \nabla \phi \rangle d\upsilon \nonumber \\
&=\int_{X}h\langle \nabla f, \nabla \phi \rangle d\upsilon + \int_{X}f \langle \nabla h, \nabla \phi \rangle d\upsilon.
\end{align}
Therefore (\ref{firs}) and (\ref{seco}) give
\begin{align*}
&\int_{X}h \left(\phi \Delta^{\upsilon}f-2\langle \nabla f, \nabla \phi \rangle +f\Delta^{\upsilon}\phi \right)d\upsilon \\
&=\int_{X}\phi \langle \nabla h, \nabla f\rangle d\upsilon +\int_{X}f\langle \nabla h, \nabla \phi \rangle d\upsilon =\int_{X}\langle \nabla h, \nabla (\phi f)\rangle d\upsilon.
\end{align*}
This completes the proof.
\end{proof}
\begin{theorem}\label{app2}
Let $R>0$, let $(X_i, \upsilon_i) \stackrel{GH}{\to} (X_{\infty}, \upsilon_{\infty})$ in $\overline{M(n, K, d)}$ with $\mathrm{diam}\,X_{\infty}>0$,
let $\{x_i\}_{i \le \infty}$ be a convergent sequence of $x_i \in X_i$,
let $f_{\infty} \in \mathcal{D}^2(\Delta^{\upsilon_{\infty}}, B_R(x_{\infty}))$ and let $\{g_i\}_{i \le \infty}$ be an $L^2$-weak convergent sequence on $B_R(x_{\infty})$ of 
$g_i \in L^2(B_R(x_i))$ with $g_{\infty}=\Delta^{\upsilon_{\infty}}f_{\infty}$.
Then for every $0<r<R$ there exist a subsequence $\{i(j)\}_j$, a sequence $\{c_{i(j)}\}_{j<\infty}$ of $c_{i(j)} \in \mathbf{R}$ with $c_{i(j)} \to 0$, and a sequence $\{f_{i(j)}\}_j$ of $f_{i(j)} \in \mathcal{D}^2(\Delta^{\upsilon_{i(j)}}, B_r(x_{i(j)}))$ such that $\Delta^{\upsilon_{i(j)}}f_{i(j)}=g_{i(j)}+c_{i(j)}$ on $B_r(x_{i(j)})$ and that $f_{i(j)}, \nabla f_{i(j)}$ $L^2$-converge strongly to $f_{\infty}, \nabla f_{\infty}$ on $B_r(x_{\infty})$, respectively.
\end{theorem}
\begin{proof}
For simplicity we only prove the assertion in the case that $(X_i, \upsilon_i) \in M(n, K, d)$ for every $i<\infty$ (by using \cite[Corollary $4.29$]{holp} we can prove it in general case).

Let $0<r< s<R$.
By \cite[Theorem $6.33$]{ch-co},
without loss of generality we can assume that there exists a sequence $\{\phi_i\}_{i\le \infty}$ of $\phi_i \in \mathrm{LIP}(X_i)$ such that the following hold:
\begin{itemize}
\item $0 \le \phi_i \le 1$, $\mathrm{supp}\,\phi_i \subset B_s(x_i)$ and $\phi_i|_{B_r(x_i)}\equiv 1$ for every $i \le \infty$.
\item $\phi_i \in C^{\infty}(X_i)$ for every $i<\infty$ with $\sup_{i<\infty}\left(\mathbf{Lip}\phi_i + ||\Delta^{\upsilon_i}\phi_i||_{L^{\infty}}\right) <\infty$.
\item $\phi_i$ converges uniformly to $\phi_{\infty}$ on $B_R(x_{\infty})$.
\end{itemize}
By Theorem \ref{srell} we see that $\nabla \phi_i$ $L^2$-converges weakly to $\nabla \phi_{\infty}$ on $B_R(x_{\infty})$.
In particular \cite[Proposition $3.29$ and Theorem $4.1$]{holp} give that $\phi_{\infty} \in \mathcal{D}^2(\Delta^{\upsilon_{\infty}}, B_R(x_{\infty}))$, that $\Delta^{\upsilon_{\infty}}\phi_{\infty} \in L^{\infty}(B_R(x_{\infty}))$ and that $\Delta^{\upsilon_i}\phi_i$ $L^p$-converges weakly to $\Delta^{\upsilon_{\infty}}\phi_{\infty}$ on $B_R(x_{\infty})$ for every $1<p<\infty$. 

On the other hand  by Theorem \ref{fundpr}, we see that $f_{\infty}|_{B_s(x_{\infty})}$ is the limit of a sequence in $\mathrm{LIP}(B_s(x_{\infty}))$ with respect to the $H^{1, 2}$-norm.
Thus by 
\cite[Theorem $4.2$]{holip}, without loss of generality we can assume that there exists a sequence $\{\tilde{f}_i\}_{i<\infty}$ of $\tilde{f}_i \in \mathrm{LIP}(B_s(x_i))$ such that $\tilde{f}_i, \nabla \tilde{f}_i$ $L^2$-converge strongly to $f_{\infty}, \nabla f_{\infty}$ on $B_s(x_{\infty})$, respectively.

Let 
\[\hat{g}_{\infty}:=\Delta^{\upsilon_{\infty}}(\phi_{\infty}f_{\infty})=\phi_{\infty}\Delta^{\upsilon_{\infty}}f_{\infty}-2\langle \nabla \phi_{\infty}, \nabla f_{\infty} \rangle +f_{\infty}\Delta^{\upsilon_{\infty}}\phi_{\infty} \in L^2(X_{\infty})\]
and let
\[\hat{g}_{i}:=\phi_{i}g_{i}-2\langle \nabla \phi_{i}, \nabla \tilde{f}_{i}\rangle + \tilde{f}_{i}\Delta^{\upsilon_{i}}\phi_{i}+c_{i} \in L^2(X_{i}),\]
where $c_{i}$ is the constant satisfying
\[\int_{X_{i}}\hat{g}_{i}d\upsilon_{i}=0.\]
It is easy to check that $\hat{g}_{i}-c_i$ $L^2$-converges weakly to $\hat{g}_{\infty}$ on $X_{\infty}$.
Since
\[\int_{X_{\infty}}\hat{g}_{\infty}d\upsilon_{\infty}=0,\]
we have $c_i \to 0$.

Let $f_i:=(\Delta^{\upsilon_i})^{-1}\hat{g}_i$ for every $i<\infty$.
Theorem \ref{pois} yields that $f_{i}, \nabla f_{i}$ $L^2$-converge strongly to $\phi_{\infty}f_{\infty}, \nabla (\phi_{\infty}f_{\infty})$ on $X_{\infty}$, respectively.
Since 
\begin{align*}
\Delta^{\upsilon_i}f_{i}|_{B_r(x_i)} &= \hat{g}_{i}|_{B_r(x_i)}\\
&=\left(\phi_{i}g_{i} -2\langle \nabla \tilde{f}_{i}, \nabla \phi_{i}\rangle + \tilde{f}_{i}\Delta^{\upsilon_{i}}\phi_{i}+c_{i}\right)|_{B_r(x_i)}\\
&=g_{i}|_{B_r(x_i)} +c_{i},
\end{align*}
this completes the proof. 
\end{proof}
We are now in a position to prove Theorem \ref{app3}.

\textit{Proof of Theorem \ref{app3}.}

Without loss of generality we can assume that $g_i \in L^2(B_R(x_i))$ for every $i\le \infty$ and that $\{g_i\}_{i \le \infty}$ is an $L^2$-weak convergent sequence on $B_R(x_{\infty})$.

We use the same notation as in the proof of Theorem \ref{app2} for convenience. 
For every sufficiently large $i\le \infty$ let 
\[a_i:=-\frac{\upsilon_i(B_r(x_i))}{\upsilon_i(X_i \setminus B_r(x_i))}\]
and let $h_i:=1_{B_r(x_i)} + a_i 1_{X_i \setminus B_r(x_i)}$.
Note
\[\int_{X_i}h_id\upsilon_i=0.\]
Let $k_i := (\Delta^{\upsilon_i})^{-1}h_i$.
Since $h_i$ $L^2$-converges strongly to $h_{\infty}$ on $X_{\infty}$,
Theorem \ref{pois} yields that $k_i, \nabla k_i$ $L^2$-converge strongly to $k_{\infty}, \nabla k_{\infty}$ on $X_{\infty}$, respectively.
Since $\Delta^{\upsilon_i} k_i|_{B_r(x_i)}\equiv 1$, we see that $\Delta^{\upsilon_{i}}(f_{i}-c_{i}k_i)|_{B_r(x_i)}=g_{i}|_{B_r(x_i)}$ and that 
$f_{i}-c_{i}k_i, \nabla (f_{i}-c_{i}k_i)$ $L^2$-converge strongly to $f_{\infty}, \nabla f_{\infty}$ on $B_r(x_{\infty})$, respectively.
This completes the proof.
$\,\,\,\,\,\,\,\,\,\,\,\,\,\,\,\,\,\,\,\,\Box$

Note that Corollary \ref{apphar} is a direct consequence of Theorem \ref{app3}.
\begin{corollary}\label{associ}
Let $1<p< \infty$, let $R>0$, let $(X_i, \upsilon_i) \stackrel{GH}{\to} (X_{\infty}, \upsilon_{\infty})$ in $\overline{M(n, K, d)}$ with $\mathrm{diam}\,X_{\infty}>0$, let $\{x_i\}_{i \le \infty}$ be a convergent sequence of $x_i \in X_i$, and let $\{T_i\}_{i \le \infty}$ be a sequence of $T_{i} \in L^p(T^r_sB_R(x_i))$ with $\sup_i||T_i||_{L^p}<\infty$.
Assume that for any subsequence $\{i(j)\}_j$, convergent sequence $\{y_{i(j)}\}_{j \le \infty}$ of $y_{i(j)} \in B_R(x_{i(j)})$, $r>0$ with $\overline{B}_r(y_{\infty}) \subset B_R(x_{\infty})$, and uniform convergent sequence $\{h_{i(j)}\}_{j \le \infty}$ on $B_r(y_{\infty})$ of harmonic maps $h_{i(j)}: B_r(y_{i(j)}) \to \mathbf{R}^{r+s}$,  we have
\[\lim_{j \to \infty}\int_{B_t(y_{i(j)})}\left\langle T_{i(j)}, \nabla^r_sh_{i(j)}\right\rangle d\upsilon_{i(j)}=\int_{B_t(y_{\infty})}\left\langle T_{\infty}, \nabla^r_sh_{\infty}\right\rangle d\upsilon_{\infty}\]
for every $t<r$.
Then we see that $T_i$ $L^p$-converges weakly to $T_{\infty}$ on $B_R(x_{\infty})$.
\end{corollary}
\begin{proof}
It follows directly from Corollary \ref{apphar}, Theorem \ref{2n} and \cite[Proposition $3.71$]{holp}.
\end{proof}
From now on we turn to the proof of Theorem \ref{221}.
\begin{proposition}\label{998io}
Let $L, R>0$, let $(X, x, \upsilon) \in M(n, K)$ and let $f \in C^{\infty}(B_R(x))$ with
\[\int_{B_R(x)}\left(|f|^2+|\Delta f|^2\right)d\upsilon \le L.\]
Then for every $r<R$ we have 
\[\int_{B_r(x)}\left( |\nabla f|^{2n/(n-1)}+|\nabla|\nabla f|^2|^{2n/(2n-1)}+|\mathrm{Hess}_f|^2\right)d\upsilon \le C(n, K, r, R, L).\]
Moreover we have the following:
\begin{enumerate}
\item If $n \ge 4$, then we have
\[\int_{B_r(x)}|f|^{2n/(n-3)}d\upsilon \le  C(n, K, r, R, L).\]
\item If $n=3$ and $5r<R$, then we have the following Trudinger inequality:
\begin{align}\label{byh}
&\int_{B_r(x)}\exp \left( C_1(K, R)\left(\frac{1}{\upsilon (B_{5r}(x))}\int_{B_{5r}(x)}|\nabla f|^3d\upsilon \right)^{-1/3}\left|f-\frac{1}{\upsilon (B_r(x))}\int_{B_r(x)}fd\upsilon \right| \right)d\upsilon \nonumber \\
&\le C_2(K, R)\upsilon (B_r(x)).
\end{align}
\item If $n=2$, then
for any $z \in B_r(x)$, $t>0$ with $B_{5t}(z) \subset B_r(x)$, and $\alpha, \beta \in B_t(z)$ we have 
\[|f(\alpha)-f(\beta)|\le C(K, R)t^{1/2}d_X(\alpha, \beta)^{1/2}\left(\frac{1}{\upsilon (B_{5t}(z))}\int_{B_{5t}(z)}|\nabla f|^3d\upsilon \right)^{1/3}.\] 
\end{enumerate}
\end{proposition}
\begin{proof}
We give a proof in the case that $n \ge 4$ only because the proof of the other case is similar by using \cite[Theorem $5.1$]{HK1} with the Bishop-Gromov inequality. 

By an argument similar to the proof of \cite[Claim $4.24$]{holp}, applying Theorem \ref{sobo} for $(q, p)=(n/(n-1), 1)$ to $|\nabla f|^2$ yields
\[\int_{B_r(x)}\left(|\nabla f|^{2n/(n-1)}+|\mathrm{Hess}_f|^2\right)d\upsilon \le C(n, K, r, R, L).\]
Thus applying Theorem \ref{sobo} for $(q, p)=(2n/(n-3), 2n/(n-1))$ to $f$ yields
\[\int_{B_r(x)}|f|^{2n/(n-3)}d\upsilon \le C(n, K, r, R, L).\]
On the other hand the H$\ddot{\text{o}}$lder inequality gives
\begin{align*}
&\int_{B_r(x)}|\nabla |\nabla f|^2|^{2n/(2n-1)}d\upsilon\\
&\le \int_{B_r(x)}|\mathrm{Hess}_f|^{2n/(2n-1)}|\nabla f|^{2n/(2n-1)}d\upsilon \\
&\left(\int_{B_r(x)}|\mathrm{Hess}_f|^2d\upsilon \right)^{n/(2n-1)}\left(\int_{B_r(x)}|\nabla f|^{2n/(n-1)}d\upsilon \right)^{(n-1)/(2n-1)}\\
& \le C(n, K, r, R, L).
\end{align*}
\end{proof}
\begin{remark}\label{trudingerinequality}
Under the same notation as in Proposition \ref{998io} we assume $n=3$.
By \cite[Theorem $6.33$]{ch-co} there exists $\phi \in C^{\infty}(X)$ such that $\mathrm{supp}\,\phi \subset B_R(x)$, that $\phi|_{B_r(x)}\equiv 1$, that $0 \le \phi \le 1$ and that $|\nabla \phi|+|\Delta \phi|\le C(K, r, R)$. 
By applying $(2)$ of Proposition \ref{998io} to $\phi f \in C^{\infty}(X)$ it is not difficult to check that for every $1 \le p <\infty$ we have
$||f||_{L^p(B_r(x))}\le C(K, r, R, L, p)$.
\end{remark}
We now prove a more general result than Theorem \ref{221} which means that the weakly second-order differential structure on a Ricci limit space can be defined \textit{intrinsically}:
\begin{theorem}\label{aarr}
Let $R>0$, let $(X, \upsilon) \in \overline{M(n, K, d)}$ with $\mathrm{diam}\,X>0$, let $x \in X$ and let $f \in \mathcal{D}^2(\Delta^{\upsilon}, B_R(x))$.
Then for every $r<R$, the following hold:
\begin{enumerate}
\item $f \in H^{1, 2n/(n-1)}(B_r(x))$. Moreover we have the following:
\begin{enumerate}
\item If $n \ge 4$, then $f \in L^{2n/(n-3)}(B_r(x))$.
\item If $n=3$ and $5r<R$, then (\ref{byh}) holds.
\item If $n=2$, then $f$ is $1/2$-H$\ddot{\text{o}}$lder continuous on $B_r(z)$.
\end{enumerate}
\item $f \in \Gamma_2(B_R(x))$.
\item $|\nabla f|^2 \in H^{1, 2n/(2n-1)}(B_r(x))$. 
\item $\mathrm{Hess}_f^{g_X} \in L^2(T^0_2B_r(x))$.
\item For every $g \in \mathcal{D}^2(\Delta^{\upsilon}, B_R(x))$ we have $\nabla _{\nabla g}^{g_X}\nabla f \in L^{2n/(2n-1)}(T^0_2B_r(x))$.
\end{enumerate} 
As a corollary of them if a rectifiable system $\mathcal{A}=\{(C_i, \phi_i)\}_{i}$ of $(X, \upsilon)$ satisfies that;
\begin{itemize}
\item for any $i, j$, there exist $z \in X$, $s>0$ and $\hat{\phi}_{i, j} \in \mathcal{D}^2(\Delta^{\upsilon}, B_s(z))$ such that $C_i \subset B_s(z)$ and $\hat{\phi}_{i, j}|_{C_i}\equiv \phi_{i, j}$, where $\phi_i:=(\phi_{i, 1}, \ldots, \phi_{i, k})$ and $k=\mathrm{dim}\,X$,
\end{itemize}
then $\mathcal{A}$ is a weakly second-order differential system of $(X, \upsilon)$.
In particular we have Theorem \ref{221}.
\end{theorem}
\begin{proof}
We give a proof in the case that $n \ge 4$ only.

Let $r<s< R$, let $\{(X_i, \upsilon_i)\}_i$ be a sequence of $(X_i, \upsilon_i) \in M(n, K, d)$ with $(X_i, \upsilon_i) \stackrel{GH}{\to} (X, \upsilon)$, and let $\{x_i\}_i$ be a sequence of $x_i \in X_i$ with $x_i \stackrel{GH}{\to} x$.

Since $C^{\infty}(B_s(x_i)) \cap L^2(B_s(x_i))$ is dense in $L^2(B_s(x_i))$, by Theorem \ref{app2} and the elliptic regularity theorem, without loss of generality we can assume that there exists a sequence $\{f_i\}_i$ of $f_i \in C^{\infty}(B_s(x_i))$ such that $f_i, \nabla f_i, \Delta f_i$ $L^2$-converge strongly to $f, \nabla f, \Delta^{\upsilon} f$ on $B_s(x)$, respectively.
Thus (the proof of) \cite[Theorem $1.3$]{holp} and Proposition \ref{998io} yield $(1), (2), (3)$ and $(4)$.
Similarly by
\begin{align}\label{uuhu}
|\nabla^{g_X}_{\nabla g}\nabla f|=|\mathrm{Hess}^{g_X}_f(\nabla g, \cdot)|\le |\mathrm{Hess}^{g_X}_f||\nabla g|,
\end{align}
we have $(5)$.
The final statement follows directly from $(2)$ and an argument similar to the proof of \cite[Proposition $3.25$]{ho0}.
\end{proof}
\begin{remark}\label{aaaaaaa}
Let $f, g$ be as in Theorem \ref{aarr}.
Then by (\ref{uuhu})
if $\nabla g \in L^{\infty}(T^0_1B_r(x))$, then $\nabla _{\nabla g}^{g_X}\nabla f \in L^{2}(T^0_1B_r(x))$. 
\end{remark}
Similarly by applying \cite[Theorem $1.3$]{holp} with Theorem \ref{tt} and Remark \ref{trudingerinequality}, we have the following:
\begin{theorem}\label{hessc}
Let $R>0$, let $(X_i, \upsilon_i) \stackrel{GH}{\to} (X_{\infty}, \upsilon_{\infty})$ in $\overline{M(n, K, d)}$ with $\mathrm{diam}\,X_{\infty}>0$, let $\{x_i\}_{i \le \infty}$ be a convergent sequence of $x_i \in X_i$, let $\{f_{i, j}\}_{i < \infty, j \in \{1, 2\}}$ be sequences of $f_{i, j} \in \mathcal{D}^2(\Delta^{\upsilon_i}, B_R(x_i))$ with
\[\sup_{i<\infty, j \in \{1, 2\}}\int_{B_R(x_i)}\left(|f_{i, j}|^2 +|\Delta^{\upsilon_i}f_{i, j}|^2\right)d\upsilon_i<\infty,\]
and let $f_{\infty, j}$ be the $L^2$-weak limit on $B_R(x_{\infty})$ of $\{f_{i, j}\}_i$.
Then for any $r<R$ and $j \in \{1, 2\}$, we have $f_{\infty, j} \in \mathcal{D}^2(\Delta^{\upsilon_{\infty}}, B_r(x_{\infty}))$.
Moreover the following hold:
\begin{enumerate}
\item We have the following:
\begin{enumerate}
\item If $n \ge 4$, then $f_{i, j}$ $L^{2n/(n-3)}$-converges weakly to $f_{\infty, j}$ on $B_r(x_{\infty})$.
Moreover $f_{i, j}$ $L^{p}$-converges strongly to $f_{\infty, j}$ on $B_r(x_{\infty})$ for every $1<p<2n/(n-3)$.
\item If $n=3$, then $f_{i, j}$ $L^p$-converges strongly to $f_{\infty, j}$ on $B_r(x_{\infty})$ for every $1<p<\infty$.
\item If $n=2$, then $f_{i, j}$ converges uniformly to $f_{\infty, j}$ on $B_r(x_{\infty})$.
\end{enumerate}
\item $\nabla f_{i, j}$ $L^{2n/(n-1)}$-converges weakly to $\nabla f_{\infty, j}$ on $B_r(x_{\infty})$.
\item $\nabla f_{i, j}$ $L^{p}$-converges strongly to $\nabla f_{\infty, j}$ on $B_r(x_{\infty})$ for every $1<p<2n/(n-1)$.
\item $\nabla |\nabla f_{i, j}|^2$ $L^{2n/(2n-1)}$-converges weakly to $\nabla |\nabla f_{\infty, j}|^2$ on $B_r(x_{\infty})$.
\item $\mathrm{Hess}_{f_{i, j}}^{g_{X_i}}$ $L^2$-converges weakly to $\mathrm{Hess}_{f_{\infty, j}}^{g_{X_{\infty}}}$ on $B_r(x_{\infty})$.
\item $\nabla_{\nabla f_{i, 1}}^{g_{X_i}}\nabla f_{i, 2}$ $L^{2n/(2n-1)}$-converges weakly to $\nabla_{\nabla f_{\infty, 1}}^{g_{X_{\infty}}}\nabla f_{\infty, 2}$ on $B_r(x_{\infty})$.
\end{enumerate}
\end{theorem}
We are now in a position to prove Theorem \ref{33766}.

\textit{Proof of Theorem \ref{33766}.}

We first prove $(1)$.

Let $x \in U$ and let $r>0$ with $B_{r}(x) \subset U$.
Then Theorem \ref{aarr} gives that $f|_{B_r(x)} \in \Gamma_2(B_r(x))$.
Since $r$ is arbitrary, we have $f \in \Gamma_2(U)$.
Moreover if $\mathrm{dim}\,X=n$, then \cite[Theorem $1.5$]{holp} and the proof of Theorem \ref{aarr} imply (\ref{34221}).
Thus we have $(1)$.

Next we prove $(2)$.

By Remark \ref{appremark} we will check that 
\begin{align}\label{wedd}
&2\int_Xg_0\mathrm{Hess}^{g_X}_f(\nabla g_1, \nabla g_2)d\upsilon \nonumber \\
&=\int_X\left(-\langle \nabla f, \nabla g_1\rangle \mathrm{div}^{\upsilon}(g_0\nabla g_2)-\langle \nabla f, \nabla g_2\rangle \mathrm{div}^{\upsilon}(g_0\nabla g_1)-g_0\left\langle \nabla f, \nabla \langle \nabla g_1, \nabla g_2\rangle \right\rangle\right)d\upsilon
\end{align}
for any $g_1, g_2 \in \mathrm{Test}F(X)$ and $g_0 \in \mathrm{LIP}(X)$.

We first prove the following.
\begin{claim}\label{first}
(\ref{wedd}) holds if $\Delta^{\upsilon}f, \Delta^{\upsilon}g_i \in \mathrm{LIP}(X)$ for every $i \in \{1, 2\}$.
\end{claim}
The proof is as follows.
By Theorem \ref{hessc} and an argument similar to the proof of Theorem \ref{app6} there exist a sequence $\{(X_i, \upsilon_i)\}_i$ of $(X_i, \upsilon_i) \in M(n, K, d)$ and sequences $\{g_{i, j}, f_i\}_{i < \infty, j \in \{0, 1, 2\}}$ of $g_{i, j}, f_i \in C^{\infty}(X_i)$ such that the following hold:
\begin{itemize}
\item $(X_i, \upsilon_i) \stackrel{GH}{\to} (X, \upsilon)$.
\item $\sup_{i<\infty, j \in \{1, 2\}}\left(\mathbf{Lip}g_{i, j}+\mathbf{Lip}\Delta g_{i, j} + \mathbf{Lip}f_i + \mathbf{Lip}\Delta f_i+\mathbf{Lip}g_{i, 0}\right)<\infty$.
\item $g_{i, j}, \nabla g_{i, j}, f_i, \nabla f_i$ $L^2$-converge strongly to $g_{j}, \nabla g_{j}, f, \nabla f$ on $X$, respectively for every $j \in \{0, 1, 2\}$.
\item $\nabla \langle \nabla g_{i, 1}, \nabla g_{i, 2}\rangle$ $L^{2n/(2n-1)}$-converges weakly to $\nabla \langle \nabla g_{1}, \nabla g_{2}\rangle$ on $X$.
\item $\mathrm{Hess}^{g_{X_i}}_{f_i}$ $L^2$-converges weakly to $\mathrm{Hess}_{f}^{g_{X}}$ on $X$.
\end{itemize}
Then since
\begin{align*}
&2\int_{X_i}g_{i, 0}\mathrm{Hess}^{g_{X_i}}_{f_i}(\nabla g_{i, 1}, \nabla g_{i, 2})d\upsilon_i \\
&=\int_{X_i}\left(-\langle \nabla f_i, \nabla g_{i, 1}\rangle \mathrm{div}^{\upsilon_i}(g_{i, 0}\nabla g_{i, 2})-\langle \nabla f_i, \nabla g_{i, 2}\rangle \mathrm{div}^{\upsilon_i}(g_{i, 0}\nabla g_{i, 1})-g_{i, 0}\left\langle \nabla f_i, \nabla \langle \nabla g_{i, 1}, \nabla g_{i, 2}\rangle \right\rangle \right)d\upsilon_i
\end{align*}
for every $i<\infty$ by letting $i \to \infty$ we have Claim \ref{first}.
\begin{claim}\label{thrr}
(\ref{wedd}) holds if $\Delta^{\upsilon}f \in \mathrm{LIP}(X)$.
\end{claim}
The proof is as follows.
For every $j \in \{1, 2\}$, let $\{G_{i, j}\}_{i<\infty}$ be a sequence of $G_{i, j} \in \mathrm{LIP}(X)$ with $G_{i, j} \to \Delta^{\upsilon}g_j$ in $H^{1, 2}(X)$ and 
\[\int_XG_{i, j}d\upsilon=0.\]
Put $g_{i, j}:=(\Delta^{\upsilon})^{-1}G_{i, j}$.
Theorem \ref{pois} yields that $g_{i, j} \in \mathrm{Test}F(X)$ and that $g_{i, j}, \nabla g_{i, j}$ $L^2$-converge strongly to $g_j, \nabla g_j$ on $X$, respectively.
By Claim \ref{first}, since
\begin{align*}
&2\int_{X}g_{0}\mathrm{Hess}^{g_{X}}_{f}(\nabla g_{i_1, 1}, \nabla g_{i_2, 2})d\upsilon \\
&=\int_{X}\left(-\langle \nabla f, \nabla g_{i_1, 1}\rangle \mathrm{div}^{\upsilon}(g_{0}\nabla g_{i_2, 2})-\langle \nabla f, \nabla g_{i_2, 2}\rangle \mathrm{div}^{\upsilon}(g_{0}\nabla g_{i_1, 1})-g_{0}\left\langle \nabla f, \nabla \langle \nabla g_{i_1, 1}, \nabla g_{i_2, 2}\rangle \right\rangle \right)d\upsilon
\end{align*}
for any $i_1, i_2 \in \mathbf{N}$, by letting $i_1 \to \infty$ and $i_2 \to \infty$, we have Claim \ref{thrr}.

We turn to finishing the proof of $(2)$.
For every $f \in \mathcal{D}^2(\Delta^{\upsilon}, X)$ there exists a sequence $\{F_i\}_i$ in $\mathrm{LIP}(X)$ such that $F_i \to \Delta^{\upsilon}f$ in $L^2(X)$ and that
\[\int_XF_id\upsilon=0\]
for every $i$.

Let $f_i:= (\Delta^{\upsilon})^{-1}F_i$. 
Note that Theorem \ref{hessc} yields that $\mathrm{Hess}_{f_i}^{g_{X}} \to \mathrm{Hess}_{f}^{g_X}$ in $L^2(T^0_2X)$ and that $\nabla f_i \to \nabla f$ in $L^2(TX)$.

Claim \ref{thrr} gives that (\ref{wedd}) holds for $f=f_i$ and every $i<\infty$.
Since $g_j \in \mathrm{LIP}(X)$ and $\nabla \langle \nabla g_1, \nabla g_2 \rangle \in L^2(TX)$ by letting $i \to \infty$ in (\ref{wedd}) for $f=f_i$ we have $(2)$.

Finally we prove $(3)$.

Since Gigli proved in \cite[Proposition $3.3.18$]{gigli} that $H^{2, 2}(X)$ is the closure of $\mathcal{D}^2(\Delta^{\upsilon}, X)$ in $W^{2, 2}(X)$, it suffices to check that $\mathcal{D}^2(\Delta^{\upsilon}, X)$ is closed in $W^{2, 2}(X)$. 

Let $\{f_i\}_{i<\infty}$ be a Cauchy sequence in $\mathcal{D}^2(\Delta^{\upsilon}, X)$ with respect to the $W^{2, 2}$-norm.
By $(1)$ and $(2)$ we have $\Delta^{\upsilon}f_i=\Delta^{g_{X}}f_i$.
In particular we have $\sup_i||\Delta^{\upsilon}f_i||_{L^2}<\infty$.
Thus Theorem \ref{hessc} gives that there exists $f_{\infty} \in \mathcal{D}^2(\Delta^{\upsilon}, X)$ such that
$f_i \to f_{\infty}$ in $H^{1, 2}(X)$ and that $\mathrm{Hess}_{f_i}^{g_X}$ $L^2$-converges weakly to $\mathrm{Hess}_{f_{\infty}}^{g_{X}}$ on $X$.
Since $\{\mathrm{Hess}_{f_{i}}^{g_{X}}\}_{i<\infty}$ is a Cauchy sequence in $L^2(T^2_0X)$ we see that  
$\mathrm{Hess}_{f_i}^{g_X}$ $L^2$-converges strongly to $\mathrm{Hess}_{f_{\infty}}^{g_{X}}$ on $X$.
This completes the proof. $\,\,\,\,\,\,\,\,\,\,\,\,\,\,\,\,\,\,\Box$
\begin{remark}
Note that we can also give an alternative proof of $(2)$ of Theorem \ref{33766} by using an approximation given in Proposition \ref{approx}.
See for instance Theorems \ref{techni2} and \ref{198183}. 
\end{remark}
We end this section by giving the following three remarks.
\begin{remark}
By (the proof of) Theorem \ref{33766} and the Bochner formula, we can easily check that $\mathcal{D}^2(\Delta^{\upsilon}, X)$ is a Hilbert space equipped with the norm
\[||f||_{\mathcal{D}^2}:=\left( ||f||_{L^2}^2+||\delta^{\upsilon} f||_{L^2}^2\right)^{1/2}\]
and that 
\begin{align}\label{2wkk}
||f||_{W^{2, 2}_C}\le C(n, K, d)||f||_{\mathcal{D}^2}
\end{align}
for every $f \in \mathcal{D}^2(\Delta^{\upsilon}, X)$. 
The inequality (\ref{2wkk}) was already proven in \cite{gigli} on a $RCD$-space.
This argument gives an altenative proof on our setting via the smooth approximation.
\end{remark}
\begin{remark}\label{dsds}
By Theorem \ref{33766} we have 
\[[\nabla f_{\infty}, \nabla g_{\infty}]^{\upsilon_{\infty}}=\nabla^{\upsilon_{\infty}}_{\nabla f_{\infty}}\nabla g_{\infty}- \nabla^{\upsilon_{\infty}}_{\nabla g_{\infty}}\nabla f_{\infty} =\nabla^{g_{X_{\infty}}}_{\nabla f_{\infty}}\nabla g_{\infty}- \nabla^{g_{X_{\infty}}}_{\nabla g_{\infty}}\nabla f_{\infty} =[\nabla f_{\infty}, \nabla g_{\infty}]\]
for any $f_{\infty}, g_{\infty} \in \mathrm{Test}F(X_{\infty})$.
In particular we have 
\begin{align*}
[\alpha_{\infty}, \beta_{\infty}]^{\upsilon_{\infty}}=[\alpha_{\infty}, \beta_{\infty}]
\end{align*}
for any $\alpha_{\infty}, \beta_{\infty} \in \mathrm{Text}T^1_0X_{\infty}$.
\end{remark}
\begin{remark}
We give a typical example of non-$C^2$-functions satisfying (\ref{34221}).

Define $f \in \Gamma_2((-1, 1))$ by
\[f(t):= \begin{cases}
-t^2/2 & (t \ge 0)\\
0 & otherwise.
\end{cases}
\]
Then for every $g \in \mathrm{LIP}_c((-1, 1))$,
integration by parts gives
\begin{align*}
\int_{-1}^1\langle \nabla f, \nabla g \rangle dt=\int_{0}^1-t\frac{dg}{dt}dt=\int_0^1gdt=\int_{-1}^11_{[0,1]}gdt.
\end{align*}
This implies $f \in \mathcal{D}^2(\Delta^{H^1}, (-1, 1))$ and 
\[\Delta^{H^1}f=1_{[0, 1]}=-\mathrm{tr}(\mathrm{Hess}_f^{g_{\mathbf{R}}})=\Delta^{g_{\mathbf{R}}}f.\]
\end{remark}
\section{Schr$\ddot{\text{o}}$dinger operators and generalized Yamabe constants}
This section is devoted to the proofs of the results stated in subsections $1.2$ and $1.3$.
\subsection{Schr$\ddot{\text{o}}$dinger operators} 
In order to prove Theorem \ref{schro} we first discuss the discreteness of the spectrum of a Schr$\ddot{\text{o}}$dinger operator.
The following is standard, however we give a proof for convenience.
\begin{proposition}\label{discrete}
Let $q>2$, let $(X, \upsilon) \in \overline{M(n, K, d)}$ with $\mathrm{diam}\,X>0$ and let $g \in L^q(X)$.
Assume that $(X, \upsilon)$ satisfies the $(2q/(q-2), 2)$-Sobolev inequality on $X$ for some $(A, B)$.
Then the spectrum of $\Delta^{\upsilon}+g$ is discrete and it is bounded below.
\end{proposition}
\begin{proof}
Let $\mathcal{L}$ be the bilinear form on $H^{1, 2}(X)$ defined by
\[\mathcal{L}(u, v):=\int_X\left(\langle \nabla u, \nabla v \rangle +guv\right)d\upsilon.\]
By the Sobolev inequality it is not difficult to check that for every $\epsilon>0$ there exists $C>0$ such that 
\[\int_{X}|g|u^2d\upsilon \le \epsilon \int_X|\nabla u|^2d\upsilon+C\int_{X}|u|^2d\upsilon\]
for every $u \in H^{1, 2}(X)$ (c.f. \cite[Remark $4.2$]{acm2}).
Thus we have
\begin{align}\label{gggn}
\mathcal{L}(u, u) &=\int_X\left(|\nabla u|^2+gu^2\right)d\upsilon \nonumber \\
& \ge (1-\epsilon)\int_X|\nabla u|^2d\upsilon-C\int_X|u|^2d\upsilon
\end{align}
for every $u \in H^{1, 2}(X)$.
Then by an argument similar to that in subsection $8.2$ of \cite{gt} with (\ref{gggn}), we have the assertion.
\end{proof}
Next we discuss the upper semicontinuity of eigenvalues of Schr$\ddot{\text{o}}$dinger operators.
\begin{theorem}\label{hbv}
Let $\{q_i\}_{i \le \infty}$ be a convergent sequence in $(2, \infty)$, let $(X_i, \upsilon_i) \stackrel{GH}{\to} (X_{\infty}, \upsilon_{\infty})$ in $\overline{M(n, K, d)}$ with $\mathrm{diam}\,X_{\infty}>0$, and let $\{g_i\}_{i \le \infty}$ be an $\{L^{q_i/2}\}_i$-weak convergent sequence on $X_{\infty}$ of $g_i \in L^{q_i/2}(X_i)$.
Assume that there exist $A, B>0$ such that for every $i \le \infty$, $(X_i, \upsilon_i)$ satisfies the $(2q_i/(q_i-2), 2)$-Sobolev inequality on $X_i$ for $(A, B)$. 
Then for every $k \ge 0$ we have
\[\limsup_{i \to \infty}\lambda^{g_i}_k(X_i)\le \lambda_k^{g_{\infty}}(X_{\infty}).\]
\end{theorem}
\begin{proof}
By min-max principle we have
\[\lambda_k^{g_i}(X_i)=\inf_{E_k}\left(\sup_{u \in E_k \setminus \{0\}}R^{g_i}(u)\right),\]
where $E_k$ runs over all $(k+1)$-dimensional subspaces of $H^{1, 2}(X_i)$ and 
\[R^{g_i}(u):=\frac{\int_{X_i}\left(|du|^2+g_i|u|^2\right)d\upsilon_i}{\int_{X_i}|u|^2d\upsilon_i}.\]
Let $\epsilon>0$ and let $E_{\infty, k}$ be a $(k+1)$-dimensional space of $H^{1, 2}(X_{\infty})$ with
\[\lambda_k^{g_{\infty}}(X_{\infty})=\sup_{u \in E_{\infty, k}}R^{g_{\infty}}(u) \pm \epsilon.\]
Fix an $L^2$-orthogonal basis $f_{\infty, 1}, \ldots, f_{\infty, k+1}$ of $E_{\infty, k}$.
By \cite[Theorem $4.2$]{holip} without loss of generality we can assume that there exist sequences $\{f_{i, l}\}_{i<\infty, l \le k+1}$ of $f_{i, l} \in \mathrm{LIP}(X_i)$ such that $f_{i, l}, \nabla f_{i, l}$ $L^2$-converge strongly to $f_{\infty, l}, \nabla f_{\infty, l}$ on $X_{\infty}$, respectively.
Thus Theorem \ref{sobo3} gives that $f_{i, l}$ $\{L^{2q_i/(q_i-2)}\}_i$-converges strongly to $f_{\infty, l}$ on $X_{\infty}$.
In particular since $f_{i, l}f_{i, m}$ $\{L^{q_i/(q_i-2)}\}_i$-converges strongly to $f_{\infty, l}f_{\infty, m}$ on $X_{\infty}$ we have
\begin{align}\label{fbv}
\lim_{i \to \infty}\int_{X_i}g_if_{i, l}f_{i, m}d\upsilon_i=\int_{X_{\infty}}g_{\infty}f_{\infty, l}f_{\infty, m}d\upsilon_{\infty}.
\end{align} 
Let $E_{i, k}:= \mathrm{span}\{f_{i, l}\}_l$.
Note that $\mathrm{dim}\,E_{i, k}=k+1$ for every sufficiently large $i$.
By (\ref{fbv}) since it is easy to check that
\begin{align*}
\lim_{i \to \infty}\sup_{u \in E_{i, k} \setminus \{0\}}R^{g_i}(u)=\sup_{u \in E_{\infty, k} \setminus \{0\}}R^{g_{\infty}}(u),
\end{align*}
we have 
\[\limsup_{i \to \infty}\lambda_k^{g_i}(X_i)\le \lambda_k^{g_{\infty}}(X_{\infty})+\epsilon.\]
Thus by letting $\epsilon \to 0$ we have the assertion.
\end{proof}
We are now in a position to prove Theorem \ref{schro}:

\textit{Proof of Theorem \ref{schro}.}

Since $(1)$ follows directly from Theorem \ref{stasob}, we will check ($2$) by induction for $k$, i.e., it suffices to check that the following $(\star k)$ holds for every $k \ge 0$.
\begin{enumerate}
\item[$(\star k)$] For every $0 \le l \le k$, we see that
\[\lim_{i \to \infty}\lambda_{l}^{g_i}(X_i)=\lambda_{l}^{g_{\infty}}(X_{\infty})\]
and that for every sequence $\{f_{i, l}\}_{i<\infty}$ of $\lambda_l^{g_i}(X_i)$-eigenfunctions $f_{i, l} \in H^{1, 2}(X_i)$ with $||f_{i, l}||_{L^2}=1$ there exist a subsequence $\{i(j)\}_j$ and a $\lambda_{l}^{g_{\infty}}(X_{\infty})$-eigenfunction $f_{\infty, l} \in H^{1, 2}(X_{\infty})$ such that $f_{i(j), l}$ $L^{2p/(p-2)}$-converges strongly to $f_{\infty, l}$ on $X_{\infty}$ and that $\nabla f_{i(j), l}$ $L^2$-converges strongly to $f_{\infty, l}$ on $X_{\infty}$. 
\end{enumerate}

Let $\{f_{i, 0}\}_{i<\infty}$ be a sequence of $\lambda_{0}^{g_i}(X_i)$-eigenfunctions $f_{i, 0}$ with $||f_{i, 0}||_{L^2}=1$.
Since 
\begin{align}\label{hghghg}
\int_{X_i}|\nabla f_{i, 0}|^2d\upsilon_i &=\lambda_0^{g_i}(X_i)-\int_{X_i}g_i|f_{i, 0}|^2d\upsilon_i \nonumber \\
&\le \lambda_0^{g_i}(X_i)+\left(\int_{X_i}|g_i|^{p/2}d\upsilon_i\right)^{2/p}\left(\int_{X_i}|f_{i, 0}|^{2p/(p-2)}d\upsilon_i\right)^{(p-2)/p} \nonumber \\
&\le \lambda_0^{g_i}(X_i)+\left(\int_{X_i}|g_i|^{p/2}d\upsilon_i\right)^{2/p}\left(A\int_{X_i}|\nabla f_{i, 0}|^2d\upsilon_i+B\right),
\end{align}
we have $\sup_i||\nabla f_{i, 0}||_{L^2}<\infty$.

Thus by Theorems \ref{srell} and \ref{tt}
without loss of generality we can assume that there exists  $f_{\infty, 0} \in H^{1, 2}(X_{\infty})$ such that $f_{i, 0}$ $L^r$-converges strongly to $f_{\infty, 0}$ on $X_{\infty}$ for every $r<2p/(p-2)$ and that $\nabla f_{i, 0}$ $L^2$-converges weakly to $\nabla f_{\infty, 0}$ on $X_{\infty}$. 
In particular we have
\[\lim_{i \to \infty}\int_{X_i}g_i|f_{i, 0}|^2d\upsilon_i=\int_{X_{\infty}}g_{\infty}|f_{\infty, 0}|^2d\upsilon_{\infty}.\]
Therefore we have
\begin{align*}
\liminf_{i \to \infty}\lambda_0^{g_i}(X_i)&=\liminf_{i \to \infty}\int_{X_i}\left(|\nabla f_{i, 0}|^2 +g_i|f_{i, 0}|^2\right)d\upsilon_i\nonumber \\
&\ge \int_{X_{\infty}}\left(|\nabla f_{\infty, 0}|^2 +g_{\infty}|f_{\infty, 0}|^2\right)d\upsilon_{\infty} \ge \lambda_{0}^{g_{\infty}}(X_{\infty}).
\end{align*}
Thus by Theorem \ref{hbv} we have
\[\lim_{i \to \infty}\lambda_0^{g_i}(X_i)=\lambda_0^{g_{\infty}}(X_{\infty}).\]

On the other hand for every $h_{\infty} \in \mathrm{LIP}(X_{\infty})$ by \cite[Theorem $4.2$]{holip} 
 without loss of generality we can assume that there exists a sequence $\{h_i\}_{i}$ of $h_i \in \mathrm{LIP}(X_i)$
with $\sup_i\mathbf{Lip}h_i<\infty$ such that $h_i, \nabla h_i$ $L^s$-converge strongly to $h_{\infty}, \nabla h_{\infty}$ on $X_{\infty}$ for every $1<s<\infty$, respectively.
Since
\[\int_{X_i}\left(\langle \nabla f_{i, 0}, \nabla h_i \rangle +g_if_{i, 0}h_i\right)d\upsilon_i=\lambda_0^{g_i}(X_i)\int_{X_i}f_{i, 0}h_id\upsilon_i\]
for every $i<\infty$, by letting $i \to \infty$, we see that $f_{\infty, 0}$ is a $\lambda_0^{g_{\infty}}(X_{\infty})$-eigenfunction of $\Delta^{\upsilon_{\infty}}+g_{\infty}$.
Then
\begin{align*}
\lim_{i \to \infty}\int_{X_i}|\nabla f_{i, 0}|^2d\upsilon_i&=\lim_{i \to \infty}\left(\lambda_0^{g_i}(X_i)-\int_{X_i}g_i|f_{i, 0}|^2d\upsilon_i\right)\\
&=\lambda_0^{g_{\infty}}(X_{\infty})-\int_{X_{\infty}}g_{\infty}|f_{\infty, 0}|^2d\upsilon_{\infty} =\int_{X_{\infty}}|\nabla f_{\infty, 0}|^2d\upsilon_{\infty}.
\end{align*}
Thus $\nabla f_{i, 0}$ $L^2$-converges strongly to $\nabla f_{\infty, 0}$ on $X_{\infty}$.
Therefore Theorem \ref{sobo3} yields that $f_{i, 0}$ $L^{2p/(p-2)}$-converges strongly to $f_{\infty, 0}$ on $X_{\infty}$.
Thus we see that $(\star 0)$ holds.

Next we assume $(\star k)$ holds for some $k \ge 0$.

Let $\{f_{i, l}\}_{i<\infty, l \le k+1}$ be sequences of $\lambda_l^{g_i}(X_i)$-eigenfunctions $f_{i, l}$ with 
\[\int_{X_i}f_{i, l}f_{i, m}d\upsilon_i=\delta_{lm}\]
for any $l, m \le k+1$.
By assumption without loss of generality we can assume that for every $l \le k$ there exists a $\lambda_l^{g_{\infty}}(X_{\infty})$-eigenfunction $f_{\infty, l} \in H^{1, 2}(X_{\infty})$ such that $f_{i, l}$ $L^{2p/(p-2)}$-converges strongly to $f_{\infty, l}$ on $X_{\infty}$ and that $\nabla f_{i, l}$ $L^2$-converges strongly to $\nabla f_{\infty, l}$ on $X_{\infty}$.

On the other hand by an argument similar to that of (\ref{hghghg}) without loss of generality we can assume that there exists $f_{\infty, k+1} \in H^{1, 2}(X_{\infty})$ such that $f_{i, k+1}$ $L^r$-converges strongly to $f_{\infty, k+1}$ on $X_{\infty}$ for every $r<2p/(p-2)$ and that $\nabla f_{i, k+1}$ $L^2$-converges weakly to $\nabla f_{\infty, k+1}$ on $X_{\infty}$. 
In particular we have
\begin{align}\label{76yu}
\lim_{i \to \infty}\int_{X_i}g_if_{i, l}f_{i, m}d\upsilon_i=\int_{X_{\infty}}g_{\infty}f_{\infty, l}f_{\infty, m}d\upsilon_{\infty}
\end{align}
for any $l, m \le k+1$.

Let $E_{i, k+1}:=\mathrm{span}\{f_{i, l}\}_{l \le k+1}$. 
Note that $\mathrm{dim}\,E_{i, k+1}=k+2$ for every sufficiently large $i \le \infty$.
Since
\[\lim_{i \to \infty}\int_{X_i}\langle \nabla f_{i, l}, \nabla f_{i, m}\rangle d\upsilon_i =\int_{X_{\infty}}\langle \nabla f_{\infty, l}, \nabla f_{\infty, m}\rangle d\upsilon_{\infty}\]
if $(l, m) \neq (k+1, k+1)$,
and
\[\liminf_{i \to \infty}\int_{X_i}|\nabla f_{i, k+1}|^2d\upsilon_i \ge \int_{X_{\infty}}|\nabla f_{\infty, k+1}|^2d\upsilon_{\infty},\]
it is easy to check that
\[\liminf_{i \to \infty}\lambda_{k+1}^{g_i}(X_i)=\liminf_{i \to \infty}\max_{u \in E_{i, k+1}}R^{g_i}(u)\ge \max_{u \in E_{\infty, k+1}}R^{g_{\infty}}(u)\ge \lambda_{k+1}^{g_{\infty}}(X_{\infty}).\]
Thus by Theorem \ref{hbv} we have
\[\lim_{i \to \infty}\lambda_{k+1}^{g_i}(X_i)=\lambda_{k+1}^{g_{\infty}}(X_{\infty}).\]
By an argument similar to that in the case of $(\star 0)$ we see that $f_{\infty, k+1}$ is a $\lambda_{k+1}^{g_{\infty}}(X_{\infty})$-eigenfunction on $X_{\infty}$, that $f_{i, k+1}$ $L^{2p/(p-2)}$-converges strongly to $f_{\infty, k+1}$ on $X_{\infty}$ and that $\nabla f_{i, k+1}$ $L^2$-converges strongly to $\nabla f_{\infty, k+1}$ on $X_{\infty}$.
Thus we see that $(\star k+1)$ holds.
This completes the proof.
$\,\,\,\,\,\,\,\,\,\,\,\,\,\,\,\,\,\Box$

\textit{Proof of Corollary \ref{horr}.}

It follows directly from Theorems \ref{schro}, \ref{sobo}, \ref{tt} and Remark \ref{akl}. $\,\,\,\,\,\,\,\,\,\,\,\,\,\,\,\,\Box$
\begin{corollary}\label{mnj}
Let $n \ge 3$, let $0<d_1\le d_2 <\infty$, let $L>0$, let $q >n/2$ and let $M$ be an $n$-dimensional compact Riemannian manifold with $d_1 \le \mathrm{diam}\,M\le d_2$ and
$\mathrm{Ric}_M \ge K(n-1)$.
Then for any $l$ and $g \in L^q(M)$ with $||g||_{L^q}\le L$, we have
\[|\lambda_l^g(M)| \le C(d_1, d_2, n, K, L, l, q).\]
\end{corollary}
\begin{proof}
The proof is done by a contradiction.
Assume that the assertion is false.
Then by Gromov's compactness theorem there exist $q>n/2$, $l \ge 0$, a sequence $\{(X_i, \upsilon_i)\}_{i <\infty}$ of $(X_i, \upsilon_i) \in M(n, K, d_2)$ with $d_1\le \mathrm{diam}\,X_i\le d_2$, the Gromov-Hausdorff limit $(X_{\infty}, \upsilon_{\infty}) \in \overline{M(n, K, d_2)}$ of them, and an $L^q$-weak convergent sequence  $\{g_i\}_{i \le \infty}$ on $X_{\infty}$ of $g_i \in L^q(X_i)$ such that
\[|\lambda_l^{g_i}(X_i)| \to \infty\]
as $i \to \infty$.
Since $0<\mathrm{diam}\,X_{\infty}<\infty$, this contradicts Corollary \ref{horr}.
\end{proof}
\begin{remark}\label{lap}
By an argument similar to the proof of Theorem \ref{schro} (or \cite[Theorem $1.6$]{holp}), we can prove continuities of \textit{$(q, p)$-Poincar\'e constants}:
\[\lambda_{q, p}(X):=\inf_f \frac{\left(\int_{X}|\nabla f|^pd\upsilon\right)^{1/p}}{\left(\int_X|f-\int_Xf|^qd\upsilon \right)^{1/q}}\]
and
\[\hat{\lambda}_{q, p}(X):=\inf_f \frac{\left(\int_{X}|\nabla f|^pd\upsilon\right)^{1/p}}{\left(\inf_{c \in \mathbf{R}}\int_X|f-c|^qd\upsilon \right)^{1/q}}\]
with respect to the Gromov-Hausdorff topology on $\overline{M(n, K, d)}$ and variables $p \in (1, n)$, $q \in [1, pn/(n-p))$, where $f$ run over all nonconstant $f \in \mathrm{LIP}(X)$.
Note that Theorem \ref{yamayama} is related to an extremal case that $p=2$ and $q=2n/(n-2)$.
See also \cite{cheegerconstant} for a related work.
\end{remark}
\subsection{Generalized Yamabe constants}
We first discuss the upper semicontinuity of generalized Yamabe constants:
\begin{theorem}\label{up}
Let $\{p_i\}_{i \le \infty}$ be a convergent sequence in $(2, \infty)$, let $2<p<\infty$, let $(X_i, \upsilon_i) \stackrel{GH}{\to} (X_{\infty}, \upsilon_{\infty})$ in $\overline{M(n, K, d)}$ with $\mathrm{diam}\,X_{\infty}>0$, and let $\{g_i\}_{i \le \infty}$ be an $L^{p/2}$-weak convergent sequence on $X_{\infty}$ of $g_i \in L^{p/2}(X_i)$. 
Then we have
\[\limsup_{i \to \infty}Y^{g_i}_{p_i}(X_i)\le Y_{p_{\infty}}^{g_{\infty}}(X_{\infty}).\]
\end{theorem}
\begin{proof}
Let $f_{\infty} \in \mathrm{LIP}(X_{\infty})$ with $||f_{\infty}||_{L^{2p_{\infty}/(p_{\infty}-2)}}=1$. 
By \cite[Theorem $4. 2$]{holip} without loss of generality we can assume that there exists a sequence $\{f_i\}_i$ of $f_i \in \mathrm{LIP}(X_i)$ such that $\sup_i\mathbf{Lip}f_i<\infty$ and that $f_i, \nabla f_i$ $\{L^{q_i}\}_i$-converge strongly to $f_{\infty}, \nabla f_{\infty}$ on $X_{\infty}$, respectively for every convergent sequence  $\{q_i\}_{i \le \infty}$ in $(1, \infty)$.

Then we have
\begin{align}\label{kkj}
Y^{g_i}_{p_i}(X_i)&\le \left(\int_{X_i}|f_i|^{2p_i/(p_i-2)}d\upsilon_i\right)^{-(p_i-2)/p_i}\left(\int_{X_i}\left(|\nabla f_i|^2+g_{i}|f_i|^2\right)d\upsilon_i\right).
\end{align}
Thus by letting $i \to \infty$ we have 
\[\limsup_{i \to \infty}Y^{g_i}_{p_i}(X_i)\le \int_{X_{\infty}}\left(|\nabla f_{\infty}|^2+g_{\infty}|f_{\infty}|^2\right)d\upsilon_{\infty}.\]
Since $f_{\infty}$ is arbitrary this completes the proof.
\end{proof}
We are now in a position to give the main result in this subsection.
Note that Theorem \ref{yamayama} is a corollary of Theorems \ref{sobo}, \ref{stasob}, \ref{stapo} and the following.
See also \cite{bl, leep}.
\begin{theorem}\label{contya}
Let $\{p_i\}_{i \le \infty}$ be a convergent sequence in $(2, \infty)$ and let $(X_i, \upsilon_i) \stackrel{GH}{\to} (X_{\infty}, \upsilon_{\infty})$ in $\overline{M(n, K, d)}$ with $\mathrm{diam}\,X_{\infty}>0$.
Assume that the following two conditions hold:
\begin{enumerate}
\item There exists $\tau >0$ such that for every $i \le \infty$, $(X_i, \upsilon_i)$ satisfies the $(2p_i/(p_i-2), 2)$-Poincar\'e inequality on $X_i$ for $\tau$.
\item There exist $A, B>0$ such that for every $i \le \infty$, $(X_i, \upsilon_i)$ satisfies the $(2p_i/(p_i-2), 2)$-Sobolev inequality on $X_i$ for $(A, B)$.
\end{enumerate}
Then for any $q>p_{\infty}/2$ and $L^q$-weak convergent sequence $\{g_i\}_{i \le \infty}$ on $X_{\infty}$ of $g_i \in L^q(X_i)$ with
\begin{align}\label{aubuni}
\limsup_{i \to \infty}Y^{g_i}_{p_i}(X_i)<A^{-1},
\end{align}
we have 
\begin{align}\label{4rfv}
\lim_{i \to \infty}Y^{g_i}_{p_i}(X_i)=Y_{p_{\infty}}^{g_{\infty}}(X_{\infty}).
\end{align}
In particular
\[Y^{g_{\infty}}_{p_{\infty}}(X_{\infty})<A^{-1}.\]
Moreover if $u_{\infty}$ is the $L^{2}$-weak limit on $X_{\infty}$ of a sequence $\{u_i\}_{i<\infty}$ of minimizers $u_i \in H^{1, 2}(X_i)$ of $Y^{g_i}_{p_i}(X_i)$ with $||u_i||_{L^{2p_i/(p_i-2)}}=1$, then we see that $u_i$ $\{L^{2p_i/(p_i-2)}\}_i$-converges strongly to $u_{\infty}$ on $X_{\infty}$, that $\nabla u_i$ $L^2$-converges strongly to $\nabla u_{\infty}$ on $X_{\infty}$ and that $u_{\infty}$ is also a minimizer of $Y^{g_{\infty}}_{p_{\infty}}(X_{\infty})$. 
\end{theorem}
\begin{proof}
By (\ref{ybound}), without loss of generality we can assume that $\{Y^{g_i}_{p_i}(X_i)\}_{i<\infty}$ is a convergent sequence in $\mathbf{R}$.
Theorem \ref{acmex} gives that for every sufficiently large $i<\infty$ there exists a minimizer $u_i \in H^{1, 2}(X_i)$ of $Y^{g_i}_{p_i}(X_i)$ with $||u_i||_{L^{2p_i/(p_i-2)}}=1$.

Then since 
\begin{align*}
\int_{X_i}|\nabla u_i|^{2}d\upsilon_i&=Y^{g_i}_{p_i}(X_i)-\int_{X_i}g_i|u_i|^2d\upsilon_i \\
&\le Y^{g_i}_{p_i}(X_i) + \left(\int_{X_i}|g_i|^{p_i/2}d\upsilon_i\right)^{2/p_i}\left(\int_{X_i}|u_i|^{2p_i/(p_i-2)}d\upsilon_i\right)^{(p_i-2)/p_i}\\
&\le Y^{g_i}_{p_i}(X_i) + \left(\int_{X_i}|g_i|^{q}d\upsilon_i\right)^{1/q}\left(A \int_{X_i}|\nabla u_i|^2d\upsilon_i + B\right)
\end{align*}
for every sufficiently large $i<\infty$, we have $\sup_{i<\infty}||\nabla u_i||_{L^2}<\infty$.
Thus by Theorems  \ref{srell} and \ref{tt} without loss of generality we can assume that there exists $u_{\infty} \in H^{1, 2}(X_{\infty})$ such that $u_i$ $\{L^{2p_i/(p_i-2)}\}_i$-converges weakly to $u_{\infty}$ on $X_{\infty}$, that $u_i$ $L^{r}$-converges strongly to $u_{\infty}$ on $X_{\infty}$ for every $r<2p_{\infty}/(p_{\infty}-2)$ and that $\nabla u_i$ $L^2$-converges weakly to $\nabla u_{\infty}$ on $X_{\infty}$.
In particular we have $||u_{\infty}||_{L^{2p_{\infty}/(p_{\infty}-2)}}\le \liminf_{i \to \infty}||u_i||_{L^{2p_i/(p_i-2)}}=1$ and
\[\lim_{i \to \infty}\int_{X_i}g_i|u_i|^2d\upsilon_i=\int_{X_{\infty}}g_{\infty}|u_{\infty}|^2d\upsilon_{\infty}.\]
By \cite[Theorem $4.2$]{holip} without loss of generality we can assume that there exist sequences $\{u_{i, j}\}_{i \le \infty, j<\infty}$ of  $u_{i, j} \in \mathrm{LIP}(X_i)$ such that $\sup_{i}\mathbf{Lip}u_{i, j}<\infty$ for every $j$, that $u_{i, j}, \nabla u_{i, j}$ $\{L^{q_i}\}_i$-converge strongly to $u_{\infty, j}, \nabla u_{\infty, j}$ on $X_{\infty}$ for every convergent sequence $\{q_i\}_{i \le \infty}$ in $(1, \infty)$, respectively and that $u_{\infty, j} \to u_{\infty}$ in $H^{1, 2}(X_{\infty})$.
Note that by Theorem \ref{sobo3} we see that $u_{\infty, j} \to u_{\infty}$ in $L^{2p_{\infty}/(p_{\infty}-2)}(X_{\infty})$. 

Then for every $j$ we have
\begin{align*}
\lim_{i \to \infty}Y^{g_i}_{p_i}(X_i)&=\lim_{i \to \infty}\int_{X_i}\left(|\nabla u_i|^2 +g_i|u_i|^2\right)d\upsilon_i \\
&=\lim_{i \to \infty}\int_{X_i}\left(|\nabla u_{i, j}|^2 +2 \langle \nabla (u_{i}-u_{i, j}), \nabla u_{i, j}\rangle +|\nabla (u_i-u_{i, j})|^2  +g_i|u_i|^2\right)d\upsilon_i \\
&\ge \int_{X_{\infty}}\left(|\nabla u_{\infty, j}|^2 + g_{\infty}|u_{\infty}|^2\right)d\upsilon_{\infty} + 2\int_{X_{\infty}}\langle \nabla (u_{\infty}-u_{\infty, j}), \nabla u_{\infty, j}\rangle d\upsilon_{\infty}\\
&\,\,\,+\limsup_{i \to \infty} \int_{X_i}|\nabla (u_{i, j}-u_i)|^2d\upsilon_i. \\
\end{align*}
Thus by letting $j \to \infty$, Theorem \ref{up} yields
\begin{align}\label{aa}
&\lim_{i \to \infty}Y^{g_i}_{p_i}(X_i)\nonumber \\
&\ge Y^{g_{\infty}}_{p_{\infty}}(X_{\infty})\left(\int_{X_{\infty}}|u_{\infty}|^{2p_{\infty}/(p_{\infty}-2)}d\upsilon_{\infty}\right)^{(p_{\infty}-2)/p_{\infty}}+ \limsup_{j \to \infty}\left(\limsup_{i \to \infty} \int_{X_i}|\nabla (u_{i, j}-u_i)|^2d\upsilon_i\right) \nonumber \\
&\ge \lim_{i \to \infty}Y^{g_i}_{p_i}(X_i)\left(\int_{X_{\infty}}|u_{\infty}|^{2p_{\infty}/(p_{\infty}-2)}d\upsilon_{\infty}\right)^{(p_{\infty}-2)/p_{\infty}}+ \limsup_{j \to \infty}\left(\limsup_{i \to \infty} \int_{X_i}|\nabla (u_{i, j}-u_i)|^2d\upsilon_i\right),
\end{align}
i.e., 
\begin{align}\label{star2}
\lim_{i \to \infty}Y^{g_i}_{p_i}(X_i)\left( 1-\left(\int_{X_{\infty}}|u_{\infty}|^{2p_{\infty}/(p_{\infty}-2)}d\upsilon_{\infty}\right)^{(p_{\infty}-2)/p_{\infty}}\right) \ge \limsup_{j \to \infty}\left(\limsup_{i \to \infty} \int_{X_i}|\nabla (u_{i, j}-u_i)|^2d\upsilon_i\right).
\end{align}
Thus if $\lim_{i \to \infty}Y^{g_i}_{p_i}(X_i)\le 0$, then $\nabla u_i$ $L^2$-converges strongly to $\nabla u_{\infty}$ on $X_{\infty}$.
Therefore Theorem \ref{sobo3} yields that $u_{i}$ $\{L^{2p_i/(p_i-2)}\}_i$-converges strongly to $u_{\infty}$ on $X_{\infty}$.
In particular, we have $||u_{\infty}||_{L^{2p_{\infty}/(p_{\infty}-2)}}=1$.
Thus (\ref{aa}) gives $\lim_{i \to \infty}Y^{g_i}_{p_i}(X_i)=Y^{g_{\infty}}_{p_{\infty}}(X_{\infty})$.

Next we assume $\lim_{i \to \infty}Y^{g_i}_{p_i}(X_i)>0$. 
Since
\[(a+b)^p \le (1+\epsilon)^{p-1}a^p+(1+1/\epsilon)^{p-1}b^p\]
for any $a \ge 0$, $b \ge 0$, $p \ge 1$ and $\epsilon >0$, we have 
\begin{align*}
|u_i|^{2p_i/(p_i-2)}&\le (|u_{i, j}-u_i|+|u_{i, j}|)^{2p_i/(p_i-2)}\\
&\le (1+\epsilon)^{(p_i+2)/(p_i-2)}|u_{i, j}-u_i|^{2p_i/(p_i-2)}+(1+1/\epsilon)^{(p_i+2)/(p_i-2)}|u_{i, j}|^{2p_i/(p_i-2)}
\end{align*}
for every $\epsilon >0$.
Thus we have
\begin{align}\label{mm}
|u_i|^{2p_i/(p_i-2)}-(1+\epsilon)^{(p_i+2)/(p_i-2)}|u_{i, j}-u_i|^{2p_i/(p_i-2)} \le (1+1/\epsilon)^{(p_i+2)/(p_i-2)}|u_{i, j}|^{2p_i/(p_i-2)}.
\end{align}
Note that by Proposition \ref{li} and Theorem \ref{nn} we see that the left hand side of (\ref{mm}) $w$-converges to 
\[|u_{\infty}|^{2p_{\infty}/(p_{\infty}-2)}-(1+\epsilon)^{(p_{\infty}+2)/(p_{\infty}-2)}|u_{\infty, j}-u_{\infty}|^{2p_{\infty}/(p_{\infty}-2)}\] on $X_{\infty}$ as $i \to \infty$ for every $j$.
Since the right hand side of (\ref{mm}) $w$-converges to $(1+1/\epsilon)^{(p_{\infty}+2)/(p_{\infty}-2)}|u_{\infty, j}|^{2p_{\infty}/(p_{\infty}-2)}$ on $X_{\infty}$ and
\begin{align*}
&\lim_{i \to \infty}\int_{X_i}\left((1+1/\epsilon )^{(p_i+2)/(p_i-2)}|u_{i, j}|^{2p_i/(p_i-2)}\right)d\upsilon_i\\
&=\int_{X_{\infty}}\left((1+1/\epsilon)^{(p_{\infty}+2)/(p_{\infty}-2)}|u_{\infty, j}|^{2p_{\infty}/(p_{\infty}-2)}\right)d\upsilon_{\infty},
\end{align*}
Corollary \ref{fatou2} yields
\begin{align*}
&\limsup_{i \to \infty}\left(\int_{X_i}\left(|u_i|^{2p_i/(p_i-2)}-(1+\epsilon)^{(p_i+2)/(p_i-2)}|u_{i, j}-u_i|^{2p_i/(p_i-2)}\right)d\upsilon_i\right)\\
&\le \int_{X_{\infty}}\left(|u_{\infty}|^{2p_{\infty}/(p_{\infty}-2)}-(1+\epsilon)^{(p_{\infty}+2)/(p_{\infty}-2)}|u_{\infty, j}-u_{\infty}|^{2p_{\infty}/(p_{\infty}-2)}\right)d\upsilon_{\infty},
\end{align*}
i.e.,
\begin{align*}
&1-(1+\epsilon)^{(p_{\infty}+2)/(p_{\infty}-2)}\liminf_{i \to \infty}\int_{X_i}|u_{i, j}-u_i|^{2p_i/(p_i-2)}d\upsilon_i \\
&\le \int_{X_{\infty}}|u_{\infty}|^{2p_{\infty}/(p_{\infty}-2)}d\upsilon_{\infty}-(1+\epsilon)^{(p_{\infty}+2)/(p_{\infty}-2)}\int_{X_{\infty}}|u_{\infty, j}-u_{\infty}|^{2p_{\infty}/(p_{\infty}-2)}d\upsilon_{\infty}.
\end{align*}
By letting $j \to \infty$ and $\epsilon \to 0$ we have
\[1-\int_{X_{\infty}}|u_{\infty}|^{2p_{\infty}/(p_{\infty}-2)}d\upsilon_{\infty} \le \liminf_{j \to \infty}\left(\liminf_{i \to \infty}\int_{X_i}|u_{i, j}-u_i|^{2p_i/(p_i-2)}d\upsilon_i\right).\]
Thus since $a^p-b^p\le (a-b)^p$ for any $a \ge b \ge 0$ and $0<p\le 1$ we have
\begin{align*}
&1- \left(\int_{X_{\infty}}|u_{\infty}|^{2p_{\infty}/(p_{\infty}-2)}d\upsilon_{\infty}\right)^{(p_{\infty}-2)/p_{\infty}} \\
&\le \left(1-\int_{X_{\infty}}|u_{\infty}|^{2p_{\infty}/(p_{\infty}-2)}d\upsilon_{\infty}\right)^{(p_{\infty}-2)/p_{\infty}} \\
&\le \left(\liminf_{j \to \infty}\left(\liminf_{i \to \infty}\int_{X_i}|u_{i, j}-u_i|^{2p_i/(p_i-2)}d\upsilon_i\right)\right)^{(p_{\infty}-2)/p_{\infty}}\\
&= \liminf_{j \to \infty}\left(\liminf_{i \to \infty}\left(\int_{X_i}|u_{i, j}-u_i|^{2p_i/(p_i-2)}d\upsilon_i\right)^{(p_{i}-2)/p_{i}}\right).
\end{align*}
Therefore by (\ref{star2})
\begin{align}\label{poi}
&\left(\lim_{i \to \infty}Y^{g_i}_{p_i}(X_i)\right)\liminf_{j \to \infty}\left(\liminf_{i \to \infty}\left(\int_{X_i}|u_{i, j}-u_i|^{2p_i/(p_i-2)}d\upsilon_i\right)^{(p_{i}-2)/p_{i}}\right) \nonumber \\
&\ge \limsup_{j \to \infty}\left(\limsup_{i \to \infty} \int_{X_i}|\nabla (u_{i, j}-u_i)|^2d\upsilon_i\right).
\end{align}
On the other hand by assumption we have
\[\left(\int_{X_i}|u_{i,j}-u_i|^{2p_i/(p_i-2)}d\upsilon_i\right)^{(p_i-2)/p_i}\le A \int_{X_i}|\nabla (u_{i, j}-u_i)|^{2}d\upsilon_i+B\int_{X_i}|u_{i, j}-u_i|^2d\upsilon_i.\]
In particular
\begin{align}\label{lkk}
\limsup_{j \to \infty}\left(\limsup_{i \to \infty}\left(\int_{X_i}|u_{i, j}-u_i|^{2p_{i}/(p_{i}-2)}d\upsilon_i\right)^{(p_{i}-2)/p_{i}}\right) \le A \limsup_{j \to \infty}\left(\limsup_{i \to \infty} \int_{X_i}|\nabla (u_{i, j}-u_i)|^2d\upsilon_i\right).
\end{align}
Thus (\ref{poi}) and (\ref{lkk}) yield
\[A\left(\lim_{i \to \infty}Y^{g_i}_{p_i}(X_i)\right)\limsup_{j \to \infty}\left(\limsup_{i \to \infty} \int_{X_i}|\nabla (u_{i, j}-u_i)|^2d\upsilon_i \right)\ge \limsup_{j \to \infty}\left(\limsup_{i \to \infty} \int_{X_i}|\nabla (u_{i, j}-u_i)|^2d\upsilon_i \right).\]
Since $A\lim_{i \to \infty}Y^{g_i}_{p_i}(X_i)<1$, we have
\[\limsup_{j \to \infty}\left(\limsup_{i \to \infty} \int_{X_i}|\nabla (u_{i, j}-u_i)|^2d\upsilon_i\right)=0,\]
i.e., $\nabla u_i$ $L^2$-converges strongly to $\nabla u_{\infty}$ on $X_{\infty}$.
Therefore by an argument similar to that in the case that $\lim_{i \to \infty}Y^{g_i}_{p_i}(X_i)\le 0$ we have $\lim_{i \to \infty}Y^{g_i}_{p_i}(X_i)=Y^{g_{\infty}}_{p_{\infty}}(X_{\infty})$.

The final statement on the behavior of minimizers follows directly from the argument above.
\end{proof}
\begin{remark}
In Theorem \ref{contya}, if we do not consider the extremal case, then we do not need the assumption (\ref{aubuni}) in order to get (\ref{4rfv}).

In fact, by an argument similar to the proof of Theorem \ref{schro} we can prove the following:
\begin{itemize}
\item Let $\{p_i\}_{i \le \infty}$ be a convergent sequence in $(2, \infty]$ and let $(X_i, \upsilon_i) \stackrel{GH}{\to} (X_{\infty}, \upsilon_{\infty})$ in $\overline{M(n, K, d)}$ with $\mathrm{diam}\,X_{\infty}>0$. Assume that there exist $A, B>0$ such that for every $i<\infty$, $(X_i, \upsilon_i)$ satisfies the $(2p_i/(p_i-2), 2)$-Sobolev inequality on $X_i$ for $(A, B)$.
Then for any $q>p_{\infty}/2$, $L^q$-weak convergent sequence $\{g_i\}_{i \le \infty}$ on $X_{\infty}$ of $g_i \in L^q(X_i)$, and convergent sequence $\{r_i\}_{i \le \infty}$ in $(2, \infty)$ with $r_{\infty}>p_{\infty}$,  we have
\[\lim_{i \to \infty}Y^{g_i}_{r_i}(X_i)=Y^{g_{\infty}}_{r_{\infty}}(X_{\infty}).\]
\end{itemize} 
See Remark \ref{lap}.
\end{remark}
We are now in a position to prove Corollary \ref{yamabe}.

\textit{Proof of Corollary \ref{yamabe}}

It follows directly from Remark \ref{akl}, Theorems \ref{yamayama}, \ref{sobo} and \ref{stapo}. $\,\,\,\,\,\,\,\,\,\,\,\Box$
\begin{remark}
We give remarks on generalized Yamabe constants on a noncollapsed Ricci limit space.

Let $2<p<\infty$, let $\{(X_i, \upsilon_i)\}_i$ be a sequence of $(X_i, \upsilon_i) \in M(n, K, d)$, let $(X, \upsilon)$ be the noncollapsed Gromov-Hausdorff limit of them and let $g \in L^{p/2}(X)$.
For an open subset $U$ of $X$ we define \textit{the generalized $p$-Yamabe constant $Y^{g}_p(U)$ of $U$ associated with $g$} by
\[Y^g_p(U):=\inf_f\int_U\left(|\nabla f|^2+g|f|^2\right)d\upsilon,\]
where $f$ runs over all $f \in \mathrm{LIP}_c(U)$ with $||f||_{L^{2p/(p-2)}(U)}=1$. 

Assume that
\begin{itemize}
\item there exists $K_1>0$ such that $|\mathrm{Ric}_{X_i}|\le K_1$ for every $i$.
\end{itemize}
Then Cheeger-Colding showed in \cite{ch-co} that $\mathcal{R}$ is an open subset of $X$ and is a $C^{1, \alpha}$-Riemannian manifold for every $0<\alpha<1$.

Moreover Cheeger-Naber proved in \cite{chna} 
\[\mathrm{dim}_H(X \setminus \mathcal{R})\le n-4,\]
where $\mathrm{dim}_H$ is the Hausdorff dimension.
See also \cite{and, bkn, cct, tian}.
In particular we see that the Sobolev $2$-capacity of $X \setminus \mathcal{R}$ is zero (see \cite{ha} and \cite{km} for the definition of the Sobolev capacity).
Thus \cite[Theorem $4.8$]{shanm} yields that the canonical inclusion
\[\stackrel{\circ}{H^{1, 2}}(\mathcal{R}) \hookrightarrow H^{1, 2}(X)\]
is isomorphic, where $\stackrel{\circ}{H^{1, 2}}(\mathcal{R})$ is the closure  of $\mathrm{LIP}_c(\mathcal{R})$ in $H^{1, 2}(X)$.

In particular we have
\[Y^{g}_p(\mathcal{R})=Y^g_p(X)\]
if $p\ge n$.

Moreover we assume that
\begin{itemize}
\item $X_i$ is an Einstein manifold for every $i$.
\end{itemize}
Then Cheeger-Colding proved in \cite{ch-co1} that $\mathcal{R}$ is also an Einstein manifold.
Thus $(X, \upsilon)$ is a typical example of \textit{almost smooth metric measure spaces} in the sense of Akutagawa-Carron-Mazzeo introduced in \cite{acm}. 
\end{remark}
\section{Rellich type compactness for tensor fields}
In this section we establish a Rellich type compactness for tensor fields with respect to the Gromov-Hausdorff topology.
By using this we also discuss the compatibility between two Levi-Civita connections $\nabla^{g_X}, \nabla^{\upsilon}$ introduced in \cite{ho0, gigli}, respectively.
The main results of this section are Theorems \ref{rr}, \ref{mthm} and Corollary \ref{nnbbmm}.
\subsection{A closedness of $\nabla^{g_X}$}
We start this subsection
by giving the following lemma.
Note that by definition we have $H^{1, q}(B_R(x)) \subset H^{1, p}(B_R(x))$ for every $1<p \le q<\infty$
(see for instance \cite{ch1, ha, shanm2}).
\begin{lemma}\label{hkt1}
Let $1<p\le q<\infty$, let $R>0$, let $(X, x, \upsilon) \in \overline{M(n, K)}$ with $\mathrm{diam}\,X>0$, and let $f \in H^{1, p}(B_R(x))$.
If $\nabla f \in L^q(TB_R(x))$, then $f \in H^{1, q}(B_R(x))$.
\end{lemma}
\begin{proof}
By the $(1, p)$-Poincar\'e inequality and the H$\ddot{\text{o}}$lder inequality we have 
\begin{align*}
&\frac{1}{\upsilon (B_r(y))}\int_{B_r(y)}\left|f - \frac{1}{\upsilon (B_r(y))}\int_{B_r(y)}fd\upsilon \right| d\upsilon \\
&\le C(n, K, R, p)r\left(\frac{1}{\upsilon (B_r(y))}\int_{B_r(y)}|\nabla f|^pd\upsilon \right)^{1/p} \\
&\le C(n, K, R, p)r\left(\frac{1}{\upsilon (B_r(y))}\int_{B_r(y)}|\nabla f|^qd\upsilon \right)^{1/q}
\end{align*}
for any $y \in B_R(x)$ and $r>0$ with $B_r(y) \subset B_R(x)$.
Thus by \cite[Theorem $1.1$]{hkt} we have $f \in H^{1, q}(B_R(x))$. 
\end{proof}
\begin{definition}\label{mmjj}
Let $1<p\le \infty$, let $r, s \in \mathbf{Z}_{\ge 0}$, let $R>0$, let $(X, \upsilon) \in \overline{M(n, K, d)}$ with $\mathrm{diam}\,X>0$, and let $x \in X$.
We denote by $W_{2p}(T^r_sB_R(x))$ the set of $T \in \Gamma_1(T^r_sB_R(x))$ satisfying the following three conditions:
\begin{enumerate}
\item $T \in L^{2p}(T^r_sB_R(x))$.
\item $\nabla^{g_{X}} T \in L^2(T^r_{s+1}B_R(x))$.
\item For any $y \in B_R(x)$, $\hat{r}>0$ with $\overline{B}_{\hat{r}}(y) \subset B_R(x)$, harmonic map $h=(h_i)_{1\le i \le r+s}: B_{\hat{r}}(y) \to \mathbf{R}^{r+s}$, and $t<\hat{r}$ there exists $1<q<\infty$ such that $\langle T, \nabla ^r_sh\rangle \in H^{1, q}(B_t(y))$. 
\end{enumerate}
\end{definition}
\begin{remark}\label{aarrff}
If a Borel tensor field $T \in \Gamma_0(T^r_sB_R(x))$ satisfies $(3)$ of Definition \ref{mmjj}, then $T \in \Gamma_1(T^r_sB_R(x))$.
This follows directly from Theorems \ref{fundpr}, \ref{aarr} and an argument similar to the proof of \cite[Proposition $3.25$]{ho0} with Theorem \ref{2n}.
Note that if $(X, \upsilon) \in M(n, K, d)$ and $T \in C^{\infty}(T^r_sB_R(x))$, then $T$ satisfies $(3)$. 
\end{remark}
\begin{proposition}\label{linea}
$W_{2p}(T^r_sB_R(x))$ is a linear space.
\end{proposition}
\begin{proof}
It suffices to check that for every $T \in W_{2p}(T^r_sB_R(x))$ we can choose $q$ as in $(3)$ of Definition \ref{mmjj} by $2p/(p+1)$.

Let $h, y$ and $t$ be as in Definition \ref{mmjj}.
Note that  Corollary \ref{apphar} yields the Lipschitz continuity of $h$ on $B_t(y)$.

The H$\ddot{\text{o}}$lder inequality, (\ref{6554}) and Theorem \ref{aarr} yield
\begin{align*}
&||\nabla\left\langle T,  \nabla^r_sh\right\rangle||_{L^{2p/(p+1)}(B_t(y))}\\
&\le ||\nabla^{g_{X}} T||_{L^{2p/(p+1)}(B_t(y))}||\nabla^r_sh||_{L^{\infty}(B_t(y))}+|||T|\left|\nabla^{g_{X}} \nabla^r_sh\right| ||_{L^{2p/(p+1)}(B_t(y))}\\
&\le ||\nabla^{g_{X}} T||_{L^{2p/(p+1)}(B_t(y))}||\nabla^r_sh||_{L^{\infty}(B_t(y))}+|||T|\left|\nabla^{g_{X}} \nabla^r_sh\right| ||_{L^{2p/(p+1)}(B_t(y))}\\
&\le ||\nabla^{g_{X}} T||_{L^{2p/(p+1)}(B_t(y))}||\nabla^r_sh||_{L^{\infty}(B_t(y))}+||T||_{L^{2p}(B_t(y))}||\nabla^{g_{X}} \nabla^r_sh||_{L^2(B_t(y))}\\
&<\infty.
\end{align*}
Thus Lemma \ref{hkt1} yields the assertion.
\end{proof}
We define a norm $||\cdot||_{W_{2p}}$ on $W_{2p}(T^r_sB_R(x))$ by
\[||T||_{W_{2p}}:=||T||_{L^{2p}}+||\nabla^{g_X}T||_{L^2}.\]

Let us consider the following setting throughout this subsection.
\begin{itemize}
\item Let $R>0$, let $1<p \le \infty$ and let $r, s \in \mathbf{Z}_{\ge 0}$.
\item Let $(X_i, \upsilon_i) \stackrel{GH}{\to} (X_{\infty}, \upsilon_{\infty})$ in $\overline{\mathcal{M}(n, K, d)}$ with $\mathrm{diam}\,X_{\infty}>0$.
\item Let $\{x_i\}_{i \le \infty}$ be a convergent sequence of $x_i \in X_i$.
\item Let $\{T_i\}_{i < \infty}$ be a sequence of $T_i \in W_{2p}(T^r_sB_R(x_i))$ with $\sup_i||T_i||_{W_{2p}}<\infty$.
\item Let $T_{\infty}$ be the $L^{2p}$-weak limit  on $B_R(x_{\infty})$ of $\{T_i\}_i$.
\end{itemize}
\begin{proposition}\label{l}
Let $\{y_i\}_{i \le \infty}$ be a convergent sequence of $y_{i} \in B_R(x_{i})$, let $\hat{r}>0$ with $\overline{B}_{\hat{r}}(y_{\infty})\subset B_R(x_{\infty})$, and let $\{h_i\}_{i \le \infty}$ be a uniform convergent sequence on $B_{\hat{r}}(y_{\infty})$ of harmonic maps $h_{i}=(h_{i, j})_{1 \le j \le r+s}: B_{\hat{r}}(y_{i}) \to \mathbf{R}^{r+s}$.
Then for every $t<\hat{r}$ we see that $\left\langle T_{\infty},  \nabla^r_sh_{\infty} \right\rangle \in H^{1, 2p/(p+1)}(B_t(y_{\infty}))$, 
that $\left\langle T_i,  \nabla^r_sh_i \right\rangle$ $L^{2p}$-converges strongly to $\left\langle T_{\infty},   \nabla^r_sh_{\infty}\right\rangle$ on $B_t(y_{\infty})$
that
$\left\langle T_i,  \nabla^r_sh_i \right\rangle$ $L^{\alpha}$-converges strongly to $\left\langle T_{\infty},   \nabla^r_sh_{\infty}\right\rangle$ on $B_t(y_{\infty})$ for every $\alpha<2p$ and that $\nabla\left\langle T_i,  \nabla^r_sh_{i}\right\rangle$ $L^{2p/(p+1)}$-converges weakly to $\nabla\left\langle T_{\infty},  \nabla^r_sh _{\infty}\right\rangle$ on $B_t(y_{\infty})$.
\end{proposition}
\begin{proof}
Fix $t<\hat{r}$.
By the proof of Proposition \ref{linea} we have $\left\langle T_{i},  \nabla^r_sh_{i} \right\rangle \in H^{1, 2p/(p+1)}(B_t(y_{i}))$ for every $i<\infty$ and 
\[\sup_{i<\infty}||\left\langle T_{i},  \nabla^r_sh_{i} \right\rangle||_{H^{1, 2p/(p+1)}(B_t(y_i))}<\infty.\]
Since $\nabla^r_sh_i$ $L^2$-converges strongly to $\nabla^r_sh_{\infty}$ on $B_t(y_{\infty})$ with $\sup_i||\nabla^r_sh_i||_{L^{\infty}(B_t(y_i))}<\infty$, we see that $\left\langle T_{i},  \nabla^r_sh_{i} \right\rangle$ $L^{2p}$-converges weakly to $\left\langle T_{\infty},  \nabla^r_sh_{\infty} \right\rangle$ on $B_t(y_{\infty})$.
Thus the assertion follows from Theorem \ref{tt}.
\end{proof}
\begin{corollary}\label{diff}
We have $T_{\infty} \in \Gamma_1(T^r_sB_R(x_{\infty}))$.
\end{corollary}
\begin{proof}
It follows directly from Corollary \ref{apphar}, Remark \ref{aarrff} and Proposition \ref{l}.
\end{proof}
We are now in a position to give a closedness of $\nabla^{g_X}$ with respect to the Gromov-Hausdorff topology:
\begin{theorem}\label{uuy}
If $T_i$ $L^2$-converges strongly to $T_{\infty}$ on $B_R(x_{\infty})$, then we see that $T_{\infty} \in W_{2p}(T^r_sB_R(x_{\infty}))$ and that $\nabla^{g_{X_i}}T_i$ $L^2$-converges weakly to $\nabla^{g_{X_{\infty}}}T_{\infty}$ on $B_R(x_{\infty})$.
\end{theorem}
\begin{proof}
By Corollary \ref{associ} and \cite[Corollary $3.53$]{holp} it suffice to check the following:
\begin{claim}\label{difflp}
Let $\{y_i\}_{i \le \infty}, \{h_i\}_{i \le \infty}$ and $\hat{r}$ be as in Proposition \ref{l}, and let $\{f_i\}_{i \le \infty}$ be a uniform convergent sequence on $B_{\hat{r}}(y_{\infty})$ of harmonic functions $f_i$ on $B_{\hat{r}}(y_i)$.
Then we have 
\[\lim_{i \to \infty}\int_{B_t(y_i)}\left\langle \nabla_{\nabla f_{i}}^{g_{X_i}}T_i, \nabla^r_sh_{i}\right\rangle d\upsilon_i=\int_{B_t(y_{\infty})}\left\langle \nabla_{\nabla f_{\infty}}^{g_{X_{\infty}}}T_{\infty}, \nabla^r_sh_{\infty}\right\rangle d\upsilon_{\infty}\]
for every $t<\hat{r}$.
\end{claim}
The proof is as follow.
Note
\[\left\langle \nabla_{\nabla f_{i}}^{g_{X_i}}T_i, \nabla^r_sh_{i}\right\rangle=\left\langle \nabla f_{i}, \nabla \left\langle T_i, \nabla^r_sh_{i} \right\rangle \right\rangle-\left\langle T_i, \nabla_{\nabla f_{i}}^{g_{X_i}}\nabla^r_sh_{i}\right\rangle.\]
Corollary \ref{apphar}, Theorem \ref{hessc} and Proposition \ref{l} yield
\[\lim_{i \to \infty}\int_{B_t(y_i)}\left\langle \nabla f_{i}, \nabla \left\langle T_i, \nabla^r_sh_{i} \right\rangle \right\rangle d\upsilon_i=\int_{B_t(y_{\infty})}\left\langle \nabla f_{\infty}, \nabla \left\langle T_{\infty}, \nabla^r_sh_{\infty} \right\rangle \right\rangle d\upsilon_{\infty}.\]

On the other hand 
Theorem \ref{hessc} with Remark \ref{aaaaaaa} yields that
$\nabla_{\nabla f_{i}}^{g_{X_i}}\nabla^r_sh_{i}$ $L^2$-converges weakly to $\nabla_{\nabla f_{\infty}}^{g_{X_{\infty}}}\nabla^r_sh_{\infty}$ on $B_t(y_{\infty})$.
Thus we have
\[\lim_{i \to \infty}\int_{B_t(y_i)}\left\langle T_i, \nabla_{\nabla f_{i}}^{g_{X_i}}\nabla^r_sh_{i}\right\rangle d\upsilon_i=\int_{B_t(y_{\infty})}\left\langle T_{\infty}, \nabla_{\nabla f_{\infty}}^{g_{X_{\infty}}}\nabla^r_sh_{\infty}\right\rangle d\upsilon_{\infty}.\]
This completes the proof.
\end{proof}
Next we give a sufficient condition for the $L^2$-strong convergence of $\{T_i\}_i$.
\begin{theorem}\label{pplki}
If $(X_{\infty}, \upsilon_{\infty})$ is the noncollapsed limit of $\{(X_i, \upsilon_i)\}_i$, 
then we see that $T_i$ $L^2$-converges strongly to $T_{\infty}$ on $B_R(x_{\infty})$.
\end{theorem}
\begin{proof}
Let $\mathcal{A}:=\{(C_{\infty, j}, \phi_{\infty, j})\}_{j \in \mathbf{N}}$ be a harmonic rectifiable system of $(X_{\infty}, \upsilon_{\infty})$,
 let $w_{\infty} \in \bigcup_jC_{\infty, j}$ and let $\epsilon>0$. 
Without loss of generality we can assume that $\mathrm{Leb}\,C_{\infty, j}=C_{\infty, j}$ (see for instance \cite[Lemma $3.5$]{holip}).
Then there exist $j$, $r_0 >0$ and a $C(n, K)$-Lipschitz harmonic map $\hat{\phi}_{\infty, j}: B_{r_0}(w_{\infty}) \to \mathbf{R}^k$ such that $w_{\infty} \in C_{\infty, j}$, that $\hat{\phi}_{\infty, j}|_{C_{\infty, j} \cap B_{r_0}(w_{\infty})}=\phi_{\infty, j}$, that 
$\max \{\mathbf{Lip}\phi_{\infty, j}, \mathbf{Lip}(\phi_{\infty, j})^{-1}\} \le 1 +\epsilon$ and that
\[\frac{\upsilon_{\infty}(B_t(w_{\infty}) \cap C_{\infty, j})}{\upsilon_{\infty}(B_t(w_{\infty}))}\ge 1-\epsilon\]
for every $t<r_0$.
Then for any $t<r_0$ and $l, m$ we have
\begin{align*}
&\frac{1}{\upsilon_{\infty}(B_t(w_{\infty}))}\int_{B_t(w_{\infty})}\left|\langle \nabla \hat{\phi}_{\infty, j, l}, \nabla \hat{\phi}_{\infty, j, m}\rangle - \delta_{lm}\right|d\upsilon_{\infty} \\
&\le \frac{1}{\upsilon_{\infty}(B_t(w_{\infty}))}\int_{B_t(w_{\infty}) \cap C_{\infty, j}}\left|\langle \nabla \phi_{\infty, j, l}, \nabla \phi_{\infty, j, m}\rangle - \delta_{lm}\right|d\upsilon_{\infty} + C(n, K)\epsilon \\
&\le C(n, K)\epsilon,
\end{align*}
where $\hat{\phi}_{\infty, j}:=(\hat{\phi}_{\infty, j, 1}, \ldots, \hat{\phi}_{\infty, j, k})$ and $k=\mathrm{dim}\,X_{\infty}$.

Let $\{w_i\}_{i\le \infty}$ be a convergent sequence of $w_i \in X_i$ and fix $t>0$ with $t<r_0/2$. 
By Corollary \ref{apphar} without loss of generality we can assume that there exists a uniform convergent sequence $\{\hat{\phi}_{i, j}\}_{i\le \infty}$ on $B_{2r_0/3}(w_{\infty})$ of $C(n, K, r_0)$-Lipschitz harmonic maps $\hat{\phi}_{i, j}=(\hat{\phi}_{i, j, 1}, \ldots, \hat{\phi}_{i, j, k}): B_{2r_0/3}(w_i) \to \mathbf{R}^k$.

Then we have
\[\frac{1}{\upsilon_i(B_t(w_{i}))}\int_{B_t(w_{i})}\left|\langle \nabla \hat{\phi}_{i, j, l}, \nabla \hat{\phi}_{i, j, m}\rangle - \delta_{lm}\right|d\upsilon_i<C(n, K)\epsilon\]
for every sufficiently large $i$.
Let $A_i:=\bigcap_{l, m}\{|\langle \nabla \hat{\phi}_{i, j, l}, \nabla \hat{\phi}_{i, j, m}\rangle - \delta_{lm}|\le \epsilon^{1/2}\} \cap B_t(w_i)$.
Then we have
\begin{align*}
\frac{\upsilon_i(B_t(w_i)\setminus A_i)}{\upsilon_i(B_t(w_i))} &\le \frac{\epsilon^{-1/2}}{\upsilon_i(B_t(w_i))}\sum_{l, m}\int_{A_{i, l, m}}\left|\langle \nabla \hat{\phi}_{i, j, l}, \nabla \hat{\phi}_{i, j, m}\rangle - \delta_{lm}\right|d\upsilon_i \\
&\le C(n, K)\epsilon^{1/2},
\end{align*}
where $A_{i, l, m}:=\{|\langle \nabla \hat{\phi}_{i, j, l}, \nabla \hat{\phi}_{i, j, m}\rangle - \delta_{lm}|\le \epsilon^{1/2}\} \cap B_t(w_i)$.
It is easy to check the following (see also \cite[Proposition $2.1$]{holp}):
\begin{claim}\label{99889988}
Let $\epsilon>0$, let $V$ be a $k$-dimensional Hilbert space with the inner product $\langle \cdot, \cdot \rangle$, and let $\{v_i\}_{1\le i \le k}$ be an $\epsilon$-orthogonal bases of $V$ which means that 
\[\langle v_i, v_j \rangle =\delta_{ij}\pm \epsilon\]
holds for any $i, j$. 
Then for every $v \in V$ we have
\[|v|^2=\left(1\pm \Psi(\epsilon; k)\right)\sum_{i=1}^k\langle v, v_i\rangle^2.\]
\end{claim}
Let $\Lambda$ be the set of maps from $\{1, \ldots, r+s\}$ to $\{1, \ldots, k\}$ and let $\hat{\phi}_{i, j}^{\sigma}$ be the harmonic map from  $B_t(w_i)$ to $\mathbf{R}^{r+s}$ defined by $\hat{\phi}_{i, j}^{\sigma}:=(\hat{\phi}_{i, j, \sigma(1)}, \ldots, \hat{\phi}_{i, j, \sigma(r+s)})$ for every $\sigma \in \Lambda$.
Note that  $|\nabla^r_s\hat{\phi}_{i, j}^{\sigma}| \le C(n, K, r_0, r, s)$ and that
 $\{\nabla^r_s\hat{\phi}_{i, j}^{\sigma}(x)\}_{\sigma \in \Lambda}$ is a $\Psi(\epsilon; n, r, s)$-orthogonal bases of $(T^r_s)_xX_i$ for any $i$ and $x \in A_i$.
Then
\begin{align}\label{weddee}
&\frac{1}{\upsilon_i(B_t(w_i))}\int_{B_t(w_i)}|T_i|^2d\upsilon_i \nonumber \\
&\le \frac{1}{\upsilon_i(B_t(w_i))}\int_{A_i}|T_i|^2d\upsilon_i + \frac{1}{\upsilon_i(B_t(w_i))}\int_{B_t(w_i)\setminus A_i}|T_i|^2d\upsilon_i \nonumber \\
&\le \frac{1}{\upsilon_i(B_t(w_i))}\int_{A_i}|T_i|^2d\upsilon_i + \left(\frac{\upsilon_i(B_t(w_i)\setminus A_i)}{\upsilon_i(B_t(w_i))}\right)^{(p-1)/p}\left(\int_{B_t(w_i)}|T_i|^{2p}d\upsilon_i\right)^{1/p} \nonumber \\
&\le \frac{1}{\upsilon_i(B_t(w_i))}\int_{A_i}|T_i|^2d\upsilon_i + \Psi(\epsilon; n, K, L, p)
\end{align}
for every sufficiently large $i$.
Thus Proposition \ref{l} and Claim \ref{99889988} give
\begin{align}\label{vvfg}
&\frac{1}{\upsilon_i(B_t(w_i))}\int_{B_t(w_i)}|T_i|^2d\upsilon_i \nonumber \\
&\le \frac{1}{\upsilon_i(B_t(w_i))}\int_{A_i}|T_i|^2d\upsilon_i + \Psi(\epsilon;n, K, L, p)\nonumber \\
&\le \frac{1 + \Psi(\epsilon;n, r, s)}{\upsilon_i(B_r(w_i))}\int_{A_i}\sum_{\sigma \in \Lambda}\left\langle T_i, \nabla^r_s\hat{\phi}_{i, j}^{\sigma} \right\rangle^2 d\upsilon_i + \Psi(\epsilon; n, K, L, p)\nonumber \\
&\le \frac{1 + \Psi(\epsilon;n,  r, s)}{\upsilon_{i}(B_t(w_{i}))}\int_{B_t(w_i)}\sum_{\sigma \in \Lambda}\left\langle T_i, \nabla^r_s\hat{\phi}_{i, j} ^{\sigma}\right\rangle^2 d\upsilon_{i} + \Psi(\epsilon; n, K, L, p)\nonumber \\
&\le \frac{1 + \Psi(\epsilon;n,  r, s)}{\upsilon_{\infty}(B_t(w_{\infty}))}\int_{B_t(w_{\infty})}\sum_{\sigma \in \Lambda}\left\langle T_{\infty}, \nabla^r_s\hat{\phi}_{\infty, j}^{\sigma} \right\rangle^2 d\upsilon_{\infty} + \Psi(\epsilon; n, K, L, p).
\end{align}
On the other hand by an argument similar to (\ref{weddee}) we have
\begin{align*}
&\frac{1}{\upsilon_{\infty}(B_t(w_{\infty}))}\int_{B_t(w_{\infty})}\sum_{\sigma \in \Lambda}\left\langle T_{\infty}, \nabla^r_s\hat{\phi}_{\infty, j}^{\sigma} \right\rangle^2 d\upsilon_{\infty} \\
&\le \frac{1}{\upsilon_{\infty}(B_t(w_{\infty}))}\int_{A_{\infty}}\sum_{\sigma \in \Lambda}\left\langle T_{\infty}, \nabla^r_s\hat{\phi}_{\infty, j}^{\sigma} \right\rangle^2 d\upsilon_{\infty} + \Psi(\epsilon;n, K, L, p, r_0, r, s).
\end{align*}
Thus we have
\begin{align}\label{uulloi}
&\frac{1}{\upsilon_{\infty}(B_t(w_{\infty}))}\int_{B_t(w_{\infty})}\sum_{\sigma \in \Lambda}\left\langle T_{\infty}, \nabla^r_s\hat{\phi}_{\infty, j}^{\sigma} \right\rangle^2 d\upsilon_{\infty} \nonumber \\
&\le\frac{1 + \Psi(\epsilon;n, r, s)}{\upsilon_{\infty}(B_t(w_{\infty}))}\int_{A_{\infty}}|T_{\infty}|^2 d\upsilon_{\infty} + \Psi(\epsilon; n, K, L, p, r_0, r, s)\nonumber \\
&\le \frac{1 + \Psi(\epsilon;n,  r, s)}{\upsilon_{\infty}(B_t(w_{\infty}))}\int_{B_t(w_{\infty})}|T_{\infty}|^2 d\upsilon_{\infty} + \Psi(\epsilon; n, K, L, p, r_0, r, s)
\end{align}
for every sufficiently large $i$.
Therefore (\ref{vvfg}) and (\ref{uulloi}) yield that $\{|T_i|^2\}_i$ $L^1_{\mathrm{loc}}$-weakly upper semicontinuous at $w_{\infty}$.
In particular by Proposition \ref{upperpro} we have
\[\limsup_{i \to \infty}\int_{B_R(x_i)}|T_i|^2d\upsilon_i\le \int_{B_R(x_{\infty})}|T_{\infty}|^2d\upsilon_{\infty}.\]
This completes the proof.
\end{proof}
Theorems \ref{uuy}, \ref{pplki} and the $L^p$-weak compactness give the following Rellich type compactness:
\begin{theorem}\label{rr}
Let $R>0$, let $r, s \in \mathbf{Z}_{\ge 0}$, let $1<p \le \infty$, let $(X_i, \upsilon_i) \stackrel{GH}{\to} (X_{\infty}, \upsilon_{\infty})$ in $\overline{M(n, K, d)}$ with $\mathrm{dim}\,X_i=\mathrm{dim}\,X_{\infty} \ge 1$ for every $i$, let $\{x_i\}_{i \le \infty}$ be a convergent sequence of $x_i \in X_i$ and let 
$\{T_i\}_{i < \infty}$ be a sequence of $T_i \in W_{2p}(T^r_sB_R(x_i))$ with $\sup_i||T_i||_{W_{2p}}<\infty$.
Then there exist $T_{\infty} \in W_{2p}(T^r_sB_R(x_{\infty}))$ and a subsequence $\{i(j)\}_j$ such that $T_{i(j)}$ $L^2$-converges strongly to $T_{\infty}$ on $B_R(x_{\infty})$ and that $\nabla^{g_{X_{i(j)}}} T_{i(j)}$ $L^{2}$-converges weakly to $\nabla^{g_{X_{\infty}}} T_{\infty}$ on $B_R(x_{\infty})$.
\end{theorem}
\begin{corollary}
$W_{2p}(T^r_sB_R(x))$ is a Banach space for any $R>0$, $p \in (1, \infty]$, $r, s \in \mathbf{Z}_{\ge 0}$, $(X, \upsilon) \in \overline{M(n, K, d)}$ with $\mathrm{diam}\,X>0$, and $x \in X$.
\end{corollary}
\begin{proof}
It suffices to check the completeness.

Let $\{T_i\}_{i<\infty}$ be a Cauchy sequence in $W_{2p}(T^r_sB_R(x))$.
Since $L^{2p}(T^r_sB_R(x))$ and $L^2(T^r_{s+1}B_R(x))$ are complete there exist the $L^{2p}$-strong limit $T_{\infty} \in L^{2p}(T^r_sB_R(x))$ of $\{T_i\}_i$ and the $L^2$-strong limit $S_{\infty} \in L^2(T^r_{s+1}B_R(x))$ of $\{\nabla^{g_{X_i}} T_i\}_i$.   
On the other hand, applying Theorem \ref{rr} to the case that $(X_i, \upsilon_i) \equiv (X, \upsilon)$ yields that $T_{\infty} \in W_{2p}(T^r_sB_R(x))$ and $S_{\infty}=\nabla^{g_{X_{\infty}}} T_{\infty}$.
This completes the proof.
\end{proof}
\subsection{A closedness of Gigli's Levi-Civita connection $\nabla^{\upsilon}$}
The main purpose of this subsection is to introduce a generalization of the closedness of Gigli's Levi-Civita connection \cite[Theorem $3.4.2$]{gigli}  to the Gromov-Hausdorff setting which plays a key role in the next section.
It is the following:
\begin{theorem}\label{mnj}
Let $(X_i, \upsilon_i) \stackrel{GH}{\to} (X_{\infty}, \upsilon_{\infty})$ in $\overline{\mathcal{M}(n, K, d)}$ with $\mathrm{diam}\,X_{\infty}>0$, 
let $r, s \in \mathbf{Z}_{\ge 0}$,
let $\{T_i\}_{i<\infty}$ be a sequence of $T_i \in W^{1, 2}_C(T^r_sX_i)$ with $\sup_i||T_i||_{W^{1, 2}_C}<\infty$, and let 
$T_{\infty}$ be the $L^2$-strong limit on $X_{\infty}$ of them.
Then we see that $T_{\infty} \in W^{1, 2}_C(T^r_sX_{\infty})$ and that 
$\nabla^{\upsilon_i}T_i$ $L^2$-converges weakly to $\nabla^{\upsilon_{\infty}}T_{\infty}$ on $X_{\infty}$.
\end{theorem}
The proof will be given in subsection $7.2$.
By Theorem \ref{mnj} we have the following:
\begin{theorem}\label{mthm}
Let $R>0$, let $r, s \in \mathbf{Z}_{\ge 0}$, let $\{(X_i, \upsilon_i)\}_{i<\infty}$ be a sequence in $M(n, K, d)$, let $(X_{\infty}, \upsilon_{\infty})$ be the Gromov-Hausdorff limit of them with $\mathrm{diam}\,X_{\infty}>0$, let $\{x_i\}_{i \le \infty}$ be a convergent sequence of $x_i \in X_i$, 
let $\{T_i\}_{i < \infty}$ be a sequence of $T_i \in C^{\infty}(T^r_sB_R(x_i))$ with 
\[\sup_{i < \infty}\left( ||T_i||_{L^2} + ||\nabla T_i||_{L^2}\right)<\infty,\]
and let $T_{\infty}$ be the $L^2$-strong limit on $B_R(x_{\infty})$ of them.
Then we have the following.
\begin{enumerate}
\item $T_{\infty} \in W_{2n/(n-1)}(T^r_sB_R(x_{\infty}))$.
\item $|T_{\infty}|^2 \in H^{1, 2n/(2n-1)}(B_R(x_{\infty}))$.
\item $T_i$ $L^{2n/(n-1)}$-converges weakly to $T_{\infty}$ on $B_R(x_{\infty})$.
\item $T_{i}$ $L^{r}$-converges strongly to $T_{\infty}$ on $B_R(x_{\infty})$ for every $1<r<2n/(n-1)$.
\item $\nabla |T_i|^2$ $L^{2n/(2n-1)}$-converges weakly to $\nabla |T_{\infty}|^2$ on $B_R(x_{\infty})$.
\item $\nabla T_i$ $L^2$-converges weakly to $\nabla^{g_{X_{\infty}}} T_{\infty}$ on $B_R(x_{\infty})$.
\item If $R>d$, i.e., $X_i=B_R(x_i)$ for every $i \le \infty$, then we have $T_{\infty} \in W^{1, 2}_C(T^r_sX_{\infty})$ and $\nabla^{g_{X_{\infty}}}T_{\infty}=\nabla^{\upsilon_{\infty}}T_{\infty}$.
\end{enumerate}
\end{theorem}
\begin{proof}
Since $|\nabla |T_i|^2|\le 2|\nabla T_i||T_i|\le |\nabla T_i|^2 + |T_i|^2$, by Theorems \ref{sobo}, \ref{tt} and \ref{uuy} we see  $(1)$, $(3)$,  $(6)$ and that $|T_i|^2$ $L^r$-converges strongly to $|T_{\infty}|^2$ on $B_R(x_{\infty})$ for every $1<r<n/(n-1)$.
In particular we have $(4)$.

On the other hand for every $i<\infty$, Young's inequality yields
\begin{align*}
|\nabla |T_i|^2|^{2n/(2n-1)}&\le 2^{2n/(2n-1)}|\nabla T_i|^{2n/(2n-1)}|T_i|^{2n/(n-1)}\nonumber \\
&\le \frac{2^{2n/(2n-1)}}{2n-1}\left( n|\nabla T_i|^2+(n-1)|T_i|^{2n/(n-1)}\right).
\end{align*}
Thus Theorem \ref{srell} yields $(2)$ and $(5)$.

Finally Theorem \ref{mnj} yields $(7)$.
This completes the proof.
\end{proof}
Note that Theorem \ref{3ew3} is a direct consequence of the following:
\begin{corollary}\label{nnbbmm}
Let $R>0$, let $r, s \in \mathbf{Z}_{\ge 0}$, let $\{(X_i, \upsilon_i)\}_{i<\infty}$ be a sequence in $M(n, K, d)$, let $(X_{\infty}, \upsilon_{\infty})$ be the noncollapsed Gromov-Hausdorff limit of them, let $\{x_i\}_{i \le \infty}$ be a convergent sequence of $x_i \in X_i$ and
let $\{T_i\}_{i < \infty}$ be a sequence of $T_i \in C^{\infty}(T^r_sB_R(x_i))$ with 
\[\sup_{i < \infty}\left( ||T_i||_{L^2} + ||\nabla T_i||_{L^2}\right)<\infty.\]
Then there exist $T_{\infty} \in W_{2n/(n-1)}(T^r_sB_R(x_{\infty}))$ and a subsequence $\{i(j)\}_j$ such that the following hold.
\begin{enumerate}
\item $T_{i(j)}$ $L^{2n/(n-1)}$-converges weakly to $T_{\infty}$ on $B_R(x_{\infty})$.
\item $T_{i(j)}$ $L^r$-converges strongly to $T_{\infty}$ on $B_R(x_{\infty})$ for every $1<r<2n/(n-1)$.
\item $|T_{\infty}|^2 \in H^{1, 2n/(2n-1)}(B_R(x_{\infty}))$.
\item $\nabla |T_{i(j)}|^2$ $L^{2n/(2n-1)}$-converges weakly to $\nabla |T_{\infty}|^2$ on $B_R(x_{\infty})$.
\item $\nabla T_{i(j)}$ $L^2$-converges weakly to $\nabla^{g_{X_{\infty}}} T_{\infty}$ on $B_R(x_{\infty})$. 
\item If $R>d$,  then we have $T_{\infty} \in W^{1, 2}_C(T^r_sX_{\infty})$ and $\nabla^{g_{X_{\infty}}}T_{\infty}=\nabla^{\upsilon_{\infty}}T_{\infty}$.
\end{enumerate}
\end{corollary}
\begin{proof}
It follows directly from Theorems \ref{pplki} and \ref{mthm}.
\end{proof}
\begin{remark}\label{honsda}
Theorem \ref{mthm} yields that the following question is natural:
\begin{itemize}
\item Let $\{(X_i, \upsilon_i)\}_{i<\infty}$ be a sequence in $M(n, K, d)$ and let $(X_{\infty}, \upsilon_{\infty})$ be the Gromov-Hausdorff limit of them with $\mathrm{diam}\,X_{\infty}>0$.
For given $T_{\infty} \in L^2(T^r_sX_{\infty})$ when does there exist a sequence $\{T_i\}_{i<\infty}$ of $T_i \in C^{\infty}(T^r_sX_i)$ with $\sup_{i<\infty} ||\nabla T_i||_{L^2}<\infty$ such that $T_i$ $L^2$-converges strongly to $T_{\infty}$ on $X_{\infty}$?
\end{itemize}
We will give an answer to this question in subsection $7.2$.
See Remark \ref{ggtttt}, Corollaries \ref{66tytt} and \ref{99i99i}.
\end{remark}
Note that in general we can \textit{not} prove Corollary \ref{nnbbmm} without the noncollapsed assumption.
We give such two examples.
\begin{remark}\label{riemlp}
Let $(X_i, \upsilon_i) \stackrel{GH}{\to} (X_{\infty}, \upsilon_{\infty})$ in $\overline{M(n, K, d)}$ with $\mathrm{diam}\,X_{\infty}>0$,  and let $\{x_i\}_{i \le \infty}$ be a convergent sequence of $x_i \in X_i$.
Then we proved in \cite[Theorem $1.2$]{holp} that $g_{X_i}$ $L^p$-converges weakly to $g_{X_{\infty}}$ on $B_R(x_{\infty})$ for any $1<p<\infty$ and $R>0$.
Moreover it is proved that the following three conditions are equivalent:
\begin{enumerate}
\item $g_{X_i}$ $L^p$-converges strongly to $g_{X_{\infty}}$ on $B_R(x_{\infty})$ for any $1<p<\infty$ and $R>0$.
\item $g_{X_i}$ $L^p$-converges strongly to $g_{X_{\infty}}$ on $B_R(x_{\infty})$ for some $1<p<\infty$ and $R>0$.
\item $(X_{\infty}, \upsilon_{\infty})$ is the noncollapsed Gromov-Hausdorff limit of $\{(X_i, \upsilon_i)\}_i$.
\end{enumerate}
In particular if $(X_{\infty}, \upsilon_{\infty})$ is the collapsed limit of $\{(X_i, \upsilon_i)\}_i$, then although $\nabla^{g_{X_i}}g_{X_i} \equiv 0$, $g_{X_i}$ does not $L^2$-converge strongly to $g_{X_{\infty}}$ on $B_R(x_{\infty})$ for any $1<p<\infty$ and $R>0$.
This example shows that the noncollapsed assumption in Corollary \ref{nnbbmm} is essential.
\end{remark}
\begin{remark}\label{harharhar}
Let $X_{\epsilon}:=\mathbf{S}^1(1) \times \mathbf{S}^1(\epsilon)$, let $\omega_{\epsilon}$ be a harmonic $1$-form on $\mathbf{S}^1(\epsilon)$ with $|\omega_{\epsilon}|\equiv 1$, and let $\eta_{\epsilon}:=\pi_2^*\omega_{\epsilon}$, where $\mathbf{S}^1(\epsilon):=\{x \in \mathbf{R}^2; |x|=\epsilon\}$ and $\pi_2$ is the projection from $X_{\epsilon}$ to $\mathbf{S}^1(\epsilon)$.
Note that $\nabla \eta_{\epsilon} \equiv 0$ and that $(X_{\epsilon}, H^2/(4\pi^2 \epsilon)) \stackrel{GH}{\to} (\mathbf{S}^1(1), H^1/(2\pi))$ as $\epsilon \to 0$.
Then it is easy to check that $\eta_{\epsilon}$ $L^2$-converges weakly to $0$ on $\mathbf{S}^1(1)$.
However, since $|\eta_{\epsilon}|\equiv 1$, $\eta_{\epsilon}$ does \textit{not} $L^2$-converge strongly to $0$ on $\mathbf{S}^1(1)$.
\end{remark}
We end this subsection by giving the following application of Theorems \ref{mnj} and \ref{mthm}.
It means that \textit{the noncollapsed Gromov-Hausdorff limit of a sequence of compact K$\ddot{\text{a}}$hler manifolds is also K$\ddot{\text{a}}$hler in some weak sense}:
\begin{theorem}\label{uuhhff}
Let $\{(X_i, \upsilon_i)\}_{i<\infty}$ be a sequence in $M(n, K, d)$, let $(X_{\infty}, \upsilon_{\infty})$ be the Gromov-Hausdorff limit of them with $\mathrm{diam}\,X_{\infty}>0$, let $\{J_i\}_{i<\infty}$ be a sequence of $J_i \in C^{\infty}(T^1_1X_i)$ and let $J_{\infty}$ is the $L^2$-weak limit on $X_{\infty}$ of  them.
Assume that $(X_i, J_i, g_{X_i})$ is K$\ddot{\text{a}}$hler for every $i<\infty$, i.e., the following three conditions hold.
\begin{itemize}
\item $\nabla J_i \equiv 0$.
\item $g_{X_i}(J_i(u), J_i(v))=g_{X_i}(u, v)$ for any $x_i \in X_i$ and $u, v \in T_{x_i}X_i$, where we used the canonical identification:
\[T_{x_i}X_i \otimes T_{x_i}^*X_i \simeq \mathrm{Aut}(T_{x_i}X_i).\]
\item $J^2_i \equiv -id$.
\end{itemize}
Then we see that $J_{\infty}$ is the $L^2$-strong limit on $X_{\infty}$ of $\{J_i\}_i$ if and only if $(X_{\infty}, \upsilon_{\infty})$ is the noncollapsed Gromov-Hausdorff limit of $\{(X_i, \upsilon_i)\}_i$.
Moreover if $(X_{\infty}, \upsilon_{\infty})$ is the noncollapsed Gromov-Hausdorff limit of $\{(X_i, \upsilon_i)\}_i$, then
we have the following:
\begin{enumerate}
\item $J_{\infty} \in W^{1, 2}_C(T^1_1X_{\infty}) \cap W_{\infty}(T^1_1X_{\infty})$.
\item $\nabla^{g_{X_{\infty}}}J_{\infty}=\nabla^{\upsilon_{\infty}}J_{\infty}=0$.
\item For a.e. $x_{\infty} \in X_{\infty}$, $g_{X_{\infty}}(J_{\infty}(u), J_{\infty}(v))=g_{X_{\infty}}(u, v)$ for any $u, v \in T_{x_{\infty}}X_{\infty}$.
\item $J_{\infty}^2 =-id$ in $L^{\infty}(T^1_1X_{\infty})$.
\end{enumerate}
\end{theorem}
\begin{proof}
By Corollary \ref{nnbbmm}, if $(X_{\infty}, \upsilon_{\infty})$ is the noncollapsed Gromov-Hausdorff limit of $\{(X_i, \upsilon_i)\}_i$, then 
$J_{\infty}$ is the $L^2$-strong limit  on $X_{\infty}$ of $\{J_i\}_i$.

Assume that $J_{\infty}$ is the $L^2$-strong limit  on $X_{\infty}$ of $\{J_i\}_i$.
Note that $|J_i|^2\equiv \mathrm{dim}\,X_i$ for every $i<\infty$.
Thus Theorem \ref{mthm} and \cite[Proposition $3.45$]{holp} give $(1)$, $(2)$ and that $J_i, J_i^2$ $L^2$-converge strongly to $J_{\infty}, J_{\infty}^2$ on $X_{\infty}$, respectively. 
In particular we have $(4)$.

Next we prove $(3)$.

Let $\{x_{i, j}\}_{i \le \infty, j \in \{1, 2, 3\}}$ be sequences of $x_{i, j} \in X_{i}$ with $x_{i, j} \stackrel{GH}{\to} x_{\infty, j}$.
Since $J_{i}(\nabla r_{x_{i, j}})$ $L^2$-converges strongly to $J_{\infty}(\nabla r_{x_{\infty, j}})$ on $X_{\infty}$, we have
\begin{align*}
\int_{B_r(x_{\infty, 3})}\langle J_{\infty}(\nabla r_{x_{\infty, 1}}), J_{\infty}(\nabla r_{x_{\infty, 2}})\rangle d\upsilon_{\infty}&=\lim_{i \to \infty}\int_{B_r(x_{i, 3})}\langle J_{i}(\nabla r_{x_{i, 1}}), J_{i}(\nabla r_{x_{i, 2}})\rangle d\upsilon_{i}\\
&=\lim_{i \to \infty}\int_{B_r(x_{i, 3})}\langle \nabla r_{x_{i, 1}}, \nabla r_{x_{i, 2}}\rangle d\upsilon_{i}\\
&=\int_{B_r(x_{\infty, 3})}\langle \nabla r_{x_{\infty, 1}}, \nabla r_{x_{\infty, 2}}\rangle d\upsilon_{\infty}
\end{align*}
for every $r>0$.
Thus the Lebesgue differentiation theorem and \cite[Theorem $3,1$]{holip} yield $(3)$.

Finally $(3)$ gives $|J_{\infty}|^2=\mathrm{dim}\,X_{\infty}$.
Thus we have 
\[\lim_{i \to \infty}\mathrm{dim}\,X_i=\lim_{i \to \infty}\int_{X_i}|J_i|^2d\upsilon_i=\int_{X_{\infty}}|J_{\infty}|^2d\upsilon_{\infty}=\mathrm{dim}\,X_{\infty},\]
i.e., $(X_{\infty}, \upsilon_{\infty})$ is the noncollapsed Gromov-Hausdorff limit of $\{(X_i, \upsilon_i)\}_i$.
This completes the proof.
\end{proof}
\section{Differential forms}
In this section we discuss the behavior of differential forms with respect to the Gromov-Hausdorff topology.
\subsection{Bochner inequality}
The main purpose of this subsection is to establish a Bochner inequality for $1$-forms on a Ricci limit space.
For that we start this subsection by giving the following.
\begin{proposition}\label{bounds}
Let $L, R>0$, let $(X, x, \upsilon) \in M(n, K)$ and let $\omega \in C^{\infty}(T^*B_R(x))$ with
\[\frac{1}{\upsilon (B_R(x))}\int_{B_R(x)}\left(|\omega|^2+|d\omega|^2+|\delta \omega|^2\right)d\upsilon\le L.\]
Then for every $r<R$ we have 
\[\frac{1}{\upsilon (B_r(x))}\int_{B_r(x)}|\nabla \omega|^2d\upsilon \le C(n, K, r, R, L).\]
\end{proposition}
\begin{proof}
By \cite[Theorem $8.16$]{ch-co} there exists $\phi \in C^{\infty}(X)$ such that $0\le \phi \le 1$, that $\mathrm{supp} \,\phi \subset B_R(x)$, that $\phi \equiv 1$ on $B_r(x)$ and that $|\nabla \phi|+|\Delta \phi|\le C(n, K, r, R)$.
The Bochner formula yields 
\[-\frac{1}{2}\phi \Delta |\omega|^2\ge \phi|\nabla \omega|^2 - \phi\langle \Delta_{H, 1} \omega, \omega \rangle + K(n-1)\phi|\omega|^2.\]
Integrating this on $B_R(x)$ and Bishop-Gromov's volume comparison theorem give
\begin{align*}
C(n, r, R, K, L) &\ge -\frac{1}{2\upsilon (B_R(x))}\int_{B_R(x)}|\omega|^2\Delta \phi d\upsilon \\
&\ge \frac{1}{\upsilon (B_R(x))}\int_{B_R(x)}\left(\phi|\nabla \omega|^2 - \langle d\omega, d(\phi \omega)\rangle - \langle \delta \omega, \delta (\phi \omega)\rangle +K(n-1)\phi|\omega|^2\right)d\upsilon \\
&\ge \frac{C(n, K, r, R)}{\upsilon (B_r(x))}\int_{B_r(x)}|\nabla \omega|^2d\upsilon -C(n, r, R, K, L),
\end{align*}
where we used the inequalities 
\[|\langle d\omega,  d(\phi \omega) \rangle| \le |d\phi||\omega||d\omega|+|d\omega|^2 \le C(n, K, r, R)(|\omega|^2+|d\omega|^2)\]
and
\[|\langle \delta \omega, \delta(\phi \omega) \rangle| \le |d\phi ||\omega||\delta \omega|+|\delta \omega|^2 \le C(n, K, r, R)(|\omega|^2+|\delta \omega|^2).\]
This completes the proof.
\end{proof} 
\begin{theorem}\label{techni}
Let $R>0$, let $1<p<\infty$, let $(X_i, \upsilon_i) \stackrel{GH}{\to} (X_{\infty}, \upsilon_{\infty})$ in $\overline{M(n, K, d)}$ with $\mathrm{diam}\,X_{\infty}>0$, let $\{x_i\}_{i \le \infty}$ be a convergent sequence of $x_i \in X_i$ and let $\{\omega_i\}_{i \le \infty}$ be a sequence of $\omega_i \in \Gamma_1(\bigwedge^kT^*B_R(x_i))$.
Assume that $\nabla^{g_{X_i}} \omega_i$ $L^p$-converges weakly to $\nabla^{g_{X_{\infty}}} \omega_{\infty}$ on $B_R(x_{\infty})$.
Then we see that $d \omega_i$ $L^p$-converges weakly to $d\omega_{\infty}$ on $B_R(x_{\infty})$.
Moreover if $(X_{\infty}, \upsilon_{\infty})$ is the noncollapsed Gromov-Hausdorff limit of $\{(X_i, \upsilon_i)\}_i$ and $k=1$, then $\delta^{g_{X_i}}\omega_i$ $L^p$-converges weakly to $\delta^{g_{X_{\infty}}}\omega_{\infty}$ on $B_R(x_{\infty})$. 
\end{theorem}
\begin{proof}
Let $\{x_{i, j}\}_{i \le \infty, 0 \le j \le k}$ be sequences of $x_{i, j} \in X_i$ with $x_{i, j} \stackrel{GH}{\to} x_{\infty, j}$ as $i \to \infty$, and let $\{y_i\}_{i \le \infty}$ be a convergent sequence of $y_i \in B_R(x_i)$.
Then by (\ref{dna}) and the assumption we have
\begin{align*}
&\lim_{i \to \infty}\int_{B_r(y_{i})}d\omega_i(\nabla r_{x_{i, 0}}, \ldots, \nabla r_{x_{i, k}})d\upsilon_i \\
&=\lim_{i \to \infty}\left(\sum_{l=0}^{k}(-1)^l\int_{B_r(y_{i})}(\nabla^{g_{X_i}}_{\nabla r_{x_{i, l}}}\omega_i)(\nabla r_{x_{i, 0}}, \ldots, \nabla r_{x_{i, l-1}}, \nabla r_{x_{i, l+1}}, \ldots, \nabla r_{x_{i, k}})d\upsilon_{i}\right)\\
&=\sum_{l=0}^{k}(-1)^l\int_{B_r(y_{\infty})}(\nabla^{g_{X_{\infty}}}_{\nabla r_{x_{\infty, l}}}\omega_{\infty})(\nabla r_{x_{\infty, 0}}, \ldots, \nabla r_{x_{\infty, l-1}}, \nabla r_{x_{\infty, l+1}}, \ldots, \nabla r_{x_{\infty, k}})d\upsilon_{\infty}\\
&=\int_{B_r(y_{\infty})}d\omega_{\infty}(\nabla r_{x_{\infty, 0}}, \ldots, \nabla r_{x_{\infty, k}})d\upsilon_{\infty}
\end{align*}
for every $r>0$ with $B_r(y_{\infty}) \subset B_R(x_{\infty})$.
Thus we see that $d\omega_i$ $L^p$-converges weakly to $d\omega_{\infty}$ on $B_R(x_{\infty})$.
The last assertion is a direct consequence of \cite[Proposition $3.72$]{holp}.
\end{proof}
We are now in a position to introduce a Bochner inequality for $1$-forms on a Ricci limit space and a compatibility between codifferentials $\delta^{\upsilon}, \delta^{g_X}$ for $1$-forms:
\begin{theorem}\label{boch}
Let $\{(X_i, \upsilon_i)\}_{i<\infty}$ be a sequence in $M(n, K, d)$, let $(X_{\infty}, \upsilon_{\infty})$ be the Gromov-Hausdorff limit of them with $\mathrm{diam}\,X_{\infty}>0$, let $\{x_i\}_{i \le \infty}$ be a convergent sequence of $x_i \in X_i$, let $R>0$, and  
let $\{\omega_i\}_{i<\infty}$ be a sequence of $\omega_i \in C^{\infty}(T^*B_R(x_i))$ with 
\begin{align*}\sup_{i<\infty}\int_{B_R(x_i)}\left( |d\omega_i|^2+|\delta \omega_i|^2\right)d\upsilon_i<\infty,
\end{align*}
and let $\omega_{\infty}$ be the $L^2$-strong limit on $B_R(x_{\infty})$ of them.
Then the following hold:
\begin{enumerate}
\item We see that $\omega_{\infty} \in \mathcal{D}^2(\delta^{\upsilon_{\infty}}, B_R(x_{\infty}))$ and that $\delta^{\upsilon_{\infty}}\omega_{\infty}$ is the $L^2$-weak limit on $B_R(x_{\infty})$ of $\{\delta \omega_i\}_i$. 
Moreover if $\mathrm{dim}\,X_{\infty}=n$, then
$\delta^{g_{X_{\infty}}}\omega_{\infty}=\delta^{\upsilon_{\infty}} \omega_{\infty}$.
\item For every $r<R$, we see that $\omega_{\infty} \in W_{2n/(n-1)}(T^*B_r(x_{\infty}))$ and that $d\omega_{\infty}, \nabla^{g_{X_{\infty}}}\omega_{\infty}$ are the $L^2$-weak limits on $B_r(x_{\infty})$ of $\{d\omega_i\}_i, \{\nabla \omega_i\}_i$, respectively.
\item If $d\omega_{\infty}, \delta^{\upsilon_{\infty}}\omega_{\infty}$ are the $L^2$-strong limits on $B_r(x_{\infty})$ of $\{d\omega_i\}_i, \{\delta \omega_i\}_i$ for every $r<R$, respectively,
then we have the following;
\begin{align*}
-\frac{1}{2}\int_{B_R(x_{\infty})} \left\langle d\phi_{\infty}, d|\omega_{\infty}|^2 \right\rangle d\upsilon_{\infty}
&\ge \int_{B_R(x_{\infty})}\phi_{\infty}|\nabla^{g_{X_{\infty}}} \omega_{\infty}|^2d\upsilon_{\infty} \\
&-\int_{B_{R}(x_{\infty})}\left(\delta^{\upsilon_{\infty}} (\phi_{\infty} \omega_{\infty})\delta^{\upsilon_{\infty}} \omega_{\infty}+\langle d(\phi_{\infty}\omega_{\infty}), d\omega_{\infty} \rangle \right)d\upsilon_{\infty} \\
&+K(n-1)\int_{B_R(x_{\infty})}\phi_{\infty}|\omega_{\infty}|^2d\upsilon_{\infty}
\end{align*}
for every $\phi_{\infty} \in \mathrm{LIP}_c(B_R(x_{\infty}))$ with $\phi_{\infty} \ge 0$. 
\end{enumerate}
\end{theorem}
\begin{proof}
By Theorems \ref{mthm}, \ref{techni}, Proposition \ref{bounds} and \cite[Theorem $4.1$]{holp}, we have $(1)$ and $(2)$.

Assume that $d\omega_{\infty}, \delta^{\upsilon_{\infty}}\omega_{\infty}$ are $L^2$-strong limits on $B_r(x_{\infty})$ of $\{d\omega_i\}_i, \{\delta \omega_i\}_i$ for every $r<R$, respectively.
Let $\phi_{\infty} \in \mathrm{LIP}_c(B_R(x_{\infty}))$ with $\phi_{\infty}\ge 0$, and let $r>0$ with $\mathrm{supp}\,\phi_{\infty} \subset B_r(x_{\infty})$.

By an argument similar to the proof of \cite[Theorem $1.4$]{holp} without loss of generality we can assume that there exists a sequence $\{\phi_i\}_{i \le \infty}$ of $\phi_i \in \mathrm{LIP}_c(B_r(x_i))$ such that $\phi_i \ge 0$, that $\sup_i\mathbf{Lip}\phi_i<\infty$, that $\mathrm{supp}\,\phi_i \subset B_r(x_i)$ and that $\phi_i, d\phi_i$ $L^p$-converge strongly to $\phi_{\infty}, d\phi_{\infty}$ on $B_r(x_{\infty})$ for every $1<p<\infty$, respectively. 

Then the Bochner formula on $X_i$ yields 
\begin{align}\label{bobo}
-\frac{1}{2}\int_{B_r(x_{i})} \left\langle d\phi_i,  d|\omega_{i}|^2 \right\rangle d\upsilon_i&\ge \int_{B_r(x_{i})}\phi_{i}|\nabla \omega_{i}|^2d\upsilon_i \nonumber \\
&\,\,\,-\int_{B_{r}(x_{i})}\left(\delta (\phi_{i} \omega_{i})\delta \omega_{i} +\langle d(\phi_{i}\omega_i), d\omega_{i} \rangle \right)d\upsilon_i\nonumber \\
&\,\,\,+K(n-1)\int_{B_r(x_{i})}\phi_{i}|\omega_i|^2d\upsilon_i
\end{align}
for every $i<\infty$.
Theorem \ref{mthm} gives
\begin{align}\label{h1}
\lim_{i \to \infty}\int_{B_r(x_{i})} \left\langle d\phi_i,  d|\omega_{i}|^2 \right\rangle d\upsilon_i=\int_{B_r(x_{\infty})} \left\langle d\phi_{\infty},  d|\omega_{\infty}|^2 \right\rangle d\upsilon_{\infty}
\end{align}
and
\begin{align}\label{h2}
\lim_{i \to \infty}\int_{B_r(x_{i})}\phi_{i}|\omega_i|^2d\upsilon_i=\int_{B_r(x_{\infty})}\phi_{\infty}|\omega_{\infty}|^2d\upsilon_{\infty}.
\end{align}
Similarly by Theorem \ref{mthm}, since $(\phi_i)^{1/2}\nabla \omega_i$ $L^2$-converges weakly to $(\phi_{\infty})^{1/2}\nabla^{g_{X_{\infty}}} \omega_{\infty}$ on $B_r(x_{\infty})$ we have
\begin{align}\label{h99}
\liminf_{i \to \infty}\int_{B_r(x_i)}\phi_i|\nabla \omega_i|^2d\upsilon_i \ge \int_{B_r(x_{\infty})}\phi_{\infty}|\nabla^{g_{X_{\infty}}} \omega_{\infty}|^2d\upsilon_{\infty}.
\end{align}
On the other hand,
 we can easily see that $d(\phi_i \omega_i)$, $\delta (\phi_i \omega_i)$ $L^2$-converge strongly to  $d(\phi_{\infty} \omega_{\infty})$, $\delta^{\upsilon_{\infty}}(\phi_{\infty} \omega_{\infty})$ on $B_r(x_{\infty})$, respectively.
In particular we have
\begin{align}\label{h3}
&\lim_{i \to \infty}\int_{B_{r}(x_{i})}\left(\delta (\phi_{i} \omega_{i})\delta \omega_{i} +\langle d(\phi_{i}\omega_i), d\omega_{i} \rangle \right)d\upsilon_i \nonumber \\
&=\int_{B_{r}(x_{\infty})}\left(\delta^{\upsilon_{\infty}} (\phi_{\infty} \omega_{\infty})\delta^{\upsilon_{\infty}} \omega_{\infty} +\langle d(\phi_{\infty}\omega_{\infty}), d\omega_{\infty} \rangle \right)d\upsilon_{\infty}.
\end{align}
Thus by letting $i \to \infty$ in (\ref{bobo}) with (\ref{h1}), (\ref{h2}), (\ref{h99}) and (\ref{h3}), this completes the proof.
\end{proof}
\begin{corollary}\label{expb}
Let $R>0$, let $(X, \upsilon) \in \overline{M(n, K, d)}$ with $\mathrm{diam}\,X>0$, let $x \in X$ and let $f \in \mathcal{D}^2(\Delta^{\upsilon}, B_R(x))$.
Then we have the following:
\begin{align}\label{boch5}
-\frac{1}{2}\int_{B_R(x)}\langle d\phi, d|df|^2 \rangle d\upsilon &\ge \int_{B_R(x)}\phi|\mathrm{Hess}_{f}^{g_X}|^2d\upsilon + \int_{B_R(x)}\left(-\phi(\Delta^{\upsilon}f)^2+\Delta^{\upsilon}f\langle d\phi, df\rangle \right)d\upsilon \nonumber \\
&+K(n-1)\int_{B_R(x)}\phi|df|^2d\upsilon
\end{align}
for every $\phi \in \mathrm{LIP}_c(B_R(x))$ with $\phi \ge 0$ on $B_R(x)$.
\end{corollary}
\begin{proof}
It follows directly from Theorems \ref{app2}, \ref{hessc} and an argument similar to the proof of Theorem \ref{boch}.
\end{proof}
Note that we showed in \cite{holp} that (\ref{boch5}) holds for \textit{most} functions $f \in \mathcal{D}^2(\Delta^{\upsilon}, B_R(x))$. 
See \cite[Theorem $1.4$]{holp}.
\subsection{New test classes and Hodge Laplacian}
In this subsection we prove the remained results stated in subsection $1.4$.
Let $(X, \upsilon) \in \overline{M(n, K, d)}$ with $\mathrm{diam}\,X>0$.
We start this section by giving the following approximation.
Recall (\ref{newtest}) for the definition of $\widetilde{\mathrm{Test}}F(X)$.
\begin{proposition}\label{approx}
Let $f \in \mathrm{LIP}(X) \cap \mathcal{D}^2(\Delta^{\upsilon}, X)$.
Then there exists a sequence $\{f_i\}_{i}$ of $f_i \in \widetilde{\mathrm{Test}}F(X)$ such that $\sup_i\mathbf{Lip}f_i<\infty$, that $f_i \to f$  in $H^{1, 2}(X)$ and that $\Delta^{\upsilon}f_i \to \Delta^{\upsilon}f$ in $L^2(X)$.
Moreover if $f \in \mathrm{Test}F(X)$ then we can assume that $\Delta^{\upsilon}f_i \to \Delta^{\upsilon}f$ in $H^{1, 2}(X)$.
\end{proposition}
\begin{proof}
Let us consider the following approximation:
\[h_{\delta}(\widetilde{h}_{\epsilon}(f))\]
(recall (\ref{mollified}) for the definition of a mollified heat flow $\widetilde{h}_s$).
Then by \cite{ags0, gigli} it is easy to check the following:
\begin{itemize}
\item $h_{\delta}(\widetilde{h}_{\epsilon}(f)) \in \widetilde{\mathrm{Test}}F(X)$ for any $\epsilon, \delta >0$.
\item $\sup_{\delta<1, \epsilon<1}\mathbf{Lip}h_{\delta}(\widetilde{h}_{\epsilon}(f))<\infty$. 
\item $h_{\delta}(\widetilde{h}_{\epsilon}(f)) \to \widetilde{h}_{\epsilon}(f)$ in $H^{1, 2}(X)$ as $\delta \to 0$ for every $\epsilon >0$.
\item $\widetilde{h}_{\epsilon}(f) \to f$ in $H^{1, 2}(X)$ as $\epsilon \to 0$.
\item $\Delta^{\upsilon}(h_{\delta}(\widetilde{h}_{\epsilon}(f))) = h_{\delta}(\Delta^{\upsilon}(\widetilde{h}_{\epsilon}(f))) \to \Delta^{\upsilon}\widetilde{h}_{\epsilon}(f)$ in $L^{2}(X)$ as $\delta \to 0$ for every $\epsilon >0$.
\item $\Delta^{\upsilon}(\widetilde{h}_{\epsilon}f) = \widetilde{h}_{\epsilon}(\Delta^{\upsilon}f) \to \Delta^{\upsilon}f$ in $L^2(X)$ as $\epsilon \to 0$.
\end{itemize}
Moreover if $f \in \mathrm{Test}F(X)$, then we see that $\Delta^{\upsilon}(h_{\delta}(\widetilde{h}_{\epsilon}(f))) = h_{\delta}(\widetilde{h}_{\epsilon}(\Delta^{\upsilon}(f))) \to \Delta^{\upsilon}\widetilde{h}_{\epsilon}(f)$ in $H^{1, 2}(X)$ as $\delta \to 0$ for every $\epsilon >0$ and that $\widetilde{h}_{\epsilon}(\Delta^{\upsilon}f) \to \Delta^{\upsilon}f$ in $H^{1, 2}(X)$ as $\epsilon \to 0$.
This completes the proof.
\end{proof}
We now define new `test classes', $\widetilde{\mathrm{Test}}$, as follows:
\begin{itemize}
\item Let
\[\widetilde{\mathrm{Test}}T^r_sX:=\left\{ \sum_{k=1}^N h_0^k \nabla^r_s h^k ; N \in \mathbf{N}, h^k=(h^k_1, \ldots, h^k_{r+s}), h^k_i \in \widetilde{\mathrm{Test}}F(X)\right\} \subset \mathrm{Test}T^r_sX.\]
\item Let 
\[\widetilde{\mathrm{TestForm}}_k(X):=\left\{\sum_{i=1}^Nf_{0, i}df_{1, i} \wedge \cdots \wedge df_{k, i}; N \in \mathbf{N}, f_{j, i} \in \widetilde{\mathrm{Test}}F(X)\right\} \subset \mathrm{TestForm}_k(X).\]
\end{itemize}
\begin{remark}\label{ggtttt}
We give two remarks on $W^{1, 2}_H$-approximations.
\begin{itemize}
\item In Theorem \ref{app6}, by the proof, we can choose the $W^{1, 2}_H$-approximation $\{\omega_{i(j)}\}_j$ of $\omega$ by $\omega_{i(j)} \in \widetilde{\mathrm{TestForm}}_1(X_{i(j)})$.
\item By \cite[Proposition $3.5.12$]{gigli} and the proof of Theorem \ref{app6} with Theorem \ref{hessc}, we have the following weak type approximation: Let $k \ge 2$, let $(X_i, \upsilon_i) \stackrel{GH}{\to} (X_{\infty}, \upsilon_{\infty})$ in $\overline{M(n, K, d)}$ with $\mathrm{diam}\,X_{\infty}>0$, and let $\omega_{\infty} \in \widetilde{\mathrm{TestForm}}_k(X_{\infty})$.
Then there exist a subsequence $\{i(j)\}_j$ and a sequence $\{\omega_{i(j)}\}_j$ of $\omega_{i(j)} \in \widetilde{\mathrm{TestForm}}_k(X_{i(j)})$ with 
\[\sup_j\left(||\omega_{i(j)}||_{L^{\infty}}+||d^{\upsilon_{i(j)}}\omega_{i(j)}||_{L^{\infty}}\right)<\infty.\]
such that $\omega_{i(j)}, d^{\upsilon_{i(j)}}\omega_{i(j)}$ $L^2$-converge strongly to $\omega_{\infty}, d^{\upsilon_{\infty}}\omega_{\infty}$ on $X_{\infty}$, respectively and that $\delta^{\upsilon_{i(j)}}\omega_{i(j)}, \nabla \omega_{i(j)}$ $L^2$-converge weakly to $\delta^{\upsilon_{\infty}}\omega_{\infty}, \nabla \omega_{\infty}$ on $X_{\infty}$, respectively.
Moreover, if $(X_i, \upsilon_i) \in M(n, K, d)$ for every $i<\infty$, then we can choose $\omega_{i(j)}$ as smooth $k$-forms.
\item Similarly we have the following. 
Let $(X_i, \upsilon_i) \stackrel{GH}{\to} (X_{\infty}, \upsilon_{\infty})$ in $\overline{M(n, K, d)}$ with $\mathrm{diam}\,X_{\infty}>0$, and let $T \in \widetilde{\mathrm{Test}}T^r_sX_{\infty}$.
Then there exists a subsequence $\{i(j)\}_j$ and a sequence $\{T_{i(j)}\}_j$ of $T_{i(j)} \in \widetilde{\mathrm{Test}}T^r_sX_{i(j)}$ with 
\[\sup_j||T_{i(j)}||_{L^{\infty}}<\infty\]
such that 
$T_{i(j)}$ $L^2$-converges strongly to $T_{\infty}$ on $X_{\infty}$ and that $\nabla T_{i(j)}$ $L^2$-converges weakly to $\nabla T_{\infty}$ on $X_{\infty}$.
Moreover, if $(X_i, \upsilon_i) \in M(n, K, d)$ for every $i<\infty$, then we can choose $T_{i(j)}$ as smooth tensor fields.
This gives a positive answer to the question stated in Remark \ref{honsda}.
\end{itemize}
\end{remark}
The following is a key result to define new Sobolev spaces by using new test classes above (however we will refine them later in Theorem \ref{198183}). 
\begin{theorem}\label{density3}
We have the following:
\begin{enumerate}
\item $\widetilde{\mathrm{Test}}F(X)$ is dense in $H^{1, 2}(X)$. In particular it is also dense in $L^2(X)$.
\item $\widetilde{\mathrm{Test}}T^r_sX$ is dense in $L^2(T^r_sX)$.
\item $\widetilde{\mathrm{TestForm}}_k(X)$ is dense in $H^{1, 2}_d(\bigwedge^kT^*X)$. In particular it is also dense in $L^2(\bigwedge^kT^*X)$.
\end{enumerate}
\end{theorem}
\begin{proof}
We first prove $(1)$.
By Proposition \ref{approx} we see that $\widetilde{\mathrm{Test}}F(X)$ is dense in $\mathrm{Test}F(X)$ with respect to the $H^{1, 2}$-norm.
Since $\mathrm{Test}F(X)$ is dense in $H^{1, 2}(X)$ (see Remark \ref{appremark}), we have $(1)$.

Next we give a proof of $(2)$ in the case that $r=1$ and $s=0$ only for simplicity because the proof in the other case is similar.
Since $\mathrm{Test}TX$ is dense in $L^2(TX)$, it suffices to check that $\widetilde{\mathrm{Test}}TX$ is dense in $\mathrm{Test}TX$ in the sense of $L^2$.

Let $V =\sum_{i=1}^Nf_{1, \infty}^i\nabla f_{2, \infty}^i \in \mathrm{Test}TX$, where $f_{j, \infty}^i \in \mathrm{Test}TX$.
By Proposition \ref{approx}, for any $i, j$, there exists a sequence $\{f_{j, k}^i\}_{ k<\infty}$ of $f_{j, k}^i \in \widetilde{\mathrm{Test}}F(X)$ such that $\sup_{k}\mathbf{Lip}f_{j, k}^i<\infty$ and that $f_{j, k}^i \to f_{j, \infty}^i$ in $H^{1, 2}(X)$ as $k \to \infty$.
Let $V_k:=\sum_{i=1}^Nf_{1, k}^i\nabla f_{2, k}^i$ in $\widetilde{\mathrm{Test}}TX$ for every $k < \infty$.
Then since $V_k \to V$ in $L^2(TX)$ as $k \to \infty$, we have $(2)$.

Finally we prove $(3)$.
Let 
\[\omega:=\sum_{i=1}^Nf_{0, i, \infty}df_{1, i, \infty} \wedge \cdots \wedge df_{k, i, \infty} \in \mathrm{TestForm}_kX,\]
where $f_{j, i} \in \mathrm{Test}F(X)$.
By Proposition \ref{approx}, for any $i, j$ there exists a sequence $\{f_{j, i, l}\}_l$ of $f_{j, i, l} \in \widetilde{\mathrm{Test}}F(X)$ such that $\sup_l\mathbf{Lip}f_{j, i, l}<\infty$ and that $f_{j, i, l} \to f_{j, i}$ in $H^{1, 2}(X)$ as $l \to \infty$.
For every $l < \infty$, let 
\[\omega_{l}:=\sum_{i=1}^Nf_{0, i, l}df_{1, i, l} \wedge \cdots \wedge df_{k, i, l}.\]
By \cite[Theorem $3.5.2$]{gigli} we have 
\[d^{\upsilon}\omega_{l}=\sum_{i=1}^Ndf_{0, i, l}\wedge df_{1, i, l} \wedge \cdots \wedge df_{k, i, l}.\]
Thus we have $\omega_{l} \to \omega$ in $W^{1, 2}_d(\bigwedge^kT^*X)$ as $l \to \infty$.
This completes the proof.
\end{proof}
We now define new Sobolev spaces in the same manner of \cite{gigli} by using these new test classes, $\widetilde{\mathrm{Test}}$ instead of $\mathrm{Test}$ as follows:
\begin{itemize}
\item Let $\widetilde{W}^{1, 2}_C(T^r_sX)$ be the set of $T \in L^2(T^r_sX)$ satisfying that there exists $S \in L^2(T^r_{s+1}X)$ such that (\ref{appp}) holds for any $g_i \in \widetilde{\mathrm{Test}}F(X)$. By Theorem \ref{density3}, since $S$ is unique if it exists, we denote it by $\widetilde{\nabla}^{\upsilon}T$. It is easy to check that $\widetilde{W}^{1, 2}_C(T^r_sX)$ is a Hilbert space equipped with the norm
\[||T||_{\widetilde{W}^{1, 2}_C}:=\left(||T||_{L^2}^2+||\widetilde{\nabla}^{\upsilon}T||_{L^2}^2\right)^{1/2},\]
that $W^{1, 2}_C(T^r_sX)$ is a closed subspace of $\widetilde{W}^{1, 2}_C(T^r_sX)$, and that $\nabla^{\upsilon}T=\widetilde{\nabla}^{\upsilon}T$ for every $T \in W^{1, 2}_C(T^r_sX)$.
Thus for convenience we use the same notation: $\nabla^{\upsilon}T:=\widetilde{\nabla}^{\upsilon}T$ for every $T \in \widetilde{W}^{1, 2}_C(T^r_sX)$ (note that we will also use similar notation below).
Let $\widetilde{H}^{1, 2}_C(T^r_sX)$ be the closure of $\widetilde{\mathrm{Test}}T^r_sX$ in $\widetilde{W}^{1, 2}_C(T^r_sX)$.
\item Let $\widetilde{W}^{1, 2}_d(\bigwedge^kT^*X)$ be the set of $\omega \in L^2(\bigwedge^kT^*X)$ satisfying that there exists $\eta \in L^2(\bigwedge^{k+1}T^*X)$ such that (\ref{rrer}) holds for any $V_i \in \widetilde{\mathrm{Test}}T^1_0X$.
By Theorem \ref{density3}, since $\eta$ is unique if it exists, we denote it by $d^{\upsilon}\omega$.
By an argument similar to the proof of \cite[Theorem $3.5.2$]{gigli} we see that $\widetilde{W}^{1, 2}_d(\bigwedge^kT^*X)$ is a Hilbert space equipped with the norm
\[||\omega||_{\widetilde{W}^{1, 2}_d}:=\left(||\omega||_{L^2}^2+||d^{\upsilon}\omega||_{L^2}^2\right)^{1/2},\]
that $W^{1, 2}_d(\bigwedge^kT^*X)$ is a closed subspace of $\widetilde{W}^{1, 2}_d(\bigwedge^kT^*X)$
and that $\{(\omega, d^{\upsilon}\omega); \omega \in \widetilde{W}^{1, 2}_d(\bigwedge^kT^*X)\}$ is a closed subset of $L^2(\bigwedge^kT^*X) \times L^2(\bigwedge^{k+1}T^*X)$.
Let $\widetilde{H}^{1, 2}_d(\bigwedge^kT^*X)$ be the closure of $\widetilde{\mathrm{TestForm}}_k(X)$ in $\widetilde{W}^{1, 2}_d(\bigwedge^kT^*X)$, and let $\widetilde{W}_C^{1, 2}(\bigwedge^kT^*X):=\widetilde{W}_C^{1, 2}(T^0_kX) \cap L^2(\bigwedge^kT^*X)$.
\item Let $\widetilde{W}^{2, 2}(X)$ be the set of $f \in H^{1, 2}(X)$ satisfying that there exists $A \in L^2(T^0_2X)$ such that (\ref{78}) holds for any $g_i \in \widetilde{\mathrm{Test}}F(X)$. By Theorem \ref{density3} since $A$ is unique if it exists,
we also denote it by $\mathrm{Hess}_f^{\upsilon}$. It is easy to check that $\widetilde{W}^{2, 2}(X)$ is a Hilbert space equipped with the norm
\[||f||_{\widetilde{W}^{2, 2}}:=\left(||f||_{H^{1, 2}}^2+||\mathrm{Hess}^{\upsilon}_f||_{L^2}^2\right)^{1/2},\]
and that $W^{2, 2}(X)$ is a closed subspace of $\widetilde{W}^{2, 2}(X)$
Let $\widetilde{H}^{2, 2}(X)$ be the closure of $\widetilde{\mathrm{Test}}F(X)$ in $\widetilde{W}^{2, 2}(X)$.
\item Let $\widetilde{\mathcal{D}}^2(\delta^{\upsilon}_k, X)$ be the set of $\omega \in L^2(\bigwedge^kT^*X)$ satisfying that there exists $\eta \in L^2(\bigwedge^{k-1}T^*X)$ such that (\ref{codifferential}) holds for every $\alpha \in \widetilde{\mathrm{TestForm}}_{k-1}X$.
Since $\eta$ is unique if it exists, we denote it by $\delta^{\upsilon}_k\omega$ for short.
It is easy to check that $\widetilde{\mathcal{D}}^2(\delta^{\upsilon}_k, X)$ is a Hilbert space equipped with the norm
\[||\omega||_{\delta^{\upsilon}_k}:=\left(||\omega||_{L^2}^2+||\delta^{\upsilon}_k\omega||_{L^2}^2\right)^{1/2}\]
and that $\mathcal{D}^2(\delta_k^{\upsilon}, X)$ is a closed subset of $\widetilde{\mathcal{D}}^2(\delta^{\upsilon}_k, X)$. 
\item Let $\widetilde{W}^{1, 2}_H(\bigwedge^kT^*X):=\widetilde{W}^{1, 2}_d(\bigwedge^kT^*X) \cap \widetilde{\mathcal{D}}^2(\delta_k^{\upsilon}, X)$.
It is easy to check that $\widetilde{W}^{1, 2}_H(\bigwedge^kT^*X)$ is a Hilbert space equipped with the norm
\[||\omega||_{\widetilde{W}^{1, 2}_H}:=\left(||\omega||_{L^2}^2+||d^{\upsilon}\omega||_{L^2}^2 +||\delta^{\upsilon}\omega||_{L^2}^2\right)^{1/2},\]
and that $W^{1, 2}_H(\bigwedge^kT^*X)$ is a closed subspace of $\widetilde{W}^{1, 2}_H(\bigwedge^kT^*X)$.  
Let $\widetilde{H}^{1, 2}_H(\bigwedge^kT^*X)$ be the closure of $\widetilde{\mathrm{TestForm}}_k(X)$ in $\widetilde{W}^{1, 2}_H(\bigwedge^kT^*X)$.
\item Let $\widetilde{\mathcal{D}}^2(\Delta^{\upsilon}_{H, k}, X)$ be the set of $\omega \in \widetilde{W}^{1, 2}_H(\bigwedge^kT^*X)$ satisfying that 
there exists $\eta \in L^2(\bigwedge^kT^*X)$ such that (\ref{mkoiop}) holds for every $\alpha \in \widetilde{\mathrm{TestForm}}_k(X)$.
Since $\eta$ is unique if it exists we denote it by $\Delta^{\upsilon}_{H, k}\omega$. 
It is easy to check that $\widetilde{\mathcal{D}}^2(\Delta^{\upsilon}_{H, k}, X)$ is a Hilbert space equipped with the norm
\[||\omega||_{\mathcal{D}^2}:=\left(||\omega||_{\widetilde{W}^{1, 2}_H}^2+||\Delta_{H, k}^{\upsilon}\omega||_{L^2}^2\right)^{1/2}\]
and that $\mathcal{D}^2(\Delta^{\upsilon}_{H, k}, X)$ is a closed subset of $\widetilde{\mathcal{D}}^2(\Delta^{\upsilon}_{H, k}, X)$.
\end{itemize}
It is important that these `$\sim$-versions of Gigli's Sobolev spaces' have  \textit{closedness} with respect to the Gromov-Hausdorff topology:
\begin{theorem}\label{techni2}
Let $(X_i, \upsilon_i) \stackrel{GH}{\to} (X_{\infty}, \upsilon_{\infty})$ in $\overline{\mathcal{M}(n, K, d)}$ with $\mathrm{diam}\,X_{\infty}>0$, and 
let $k, r, s \in \mathbf{Z}_{\ge 0}$.
We have the following:
\begin{enumerate}
\item Let $\{\omega_i\}_{i < \infty}$ be a sequence of $\omega_i \in \widetilde{W}^{1, 2}_d(\bigwedge^kT^*X_i)$ with $\sup_{i<\infty}||\omega_i||_{\widetilde{W}^{1, 2}_d}<\infty$, 
and let $\omega_{\infty}$ be the $L^2$-strong limit on $X_{\infty}$ of them.
Then we see that $\omega_{\infty} \in \widetilde{W}^{1, 2}_d(\bigwedge^kT^*X_{\infty})$ and that $d^{\upsilon_i}\omega_i$ $L^2$-converges weakly to $d^{\upsilon_{\infty}}\omega_{\infty}$ on $X_{\infty}$.
\item Let $\{T_i\}_{i<\infty}$ be a sequence of $T_i \in \widetilde{W}^{1, 2}_C(T^r_sX_i)$ with $\sup_i||T_i||_{\widetilde{W}^{1, 2}_C}<\infty$, and let 
$T_{\infty}$ be the $L^2$-strong limit on $X_{\infty}$ of them.
Then we see that $T_{\infty} \in \widetilde{W}^{1, 2}_C(T^r_sX_{\infty})$ and that 
$\nabla^{\upsilon_i}T_i$ $L^2$-converges weakly to $\nabla^{\upsilon_{\infty}}T_{\infty}$ on $X_{\infty}$.
\end{enumerate}
\end{theorem}
\begin{proof}
We only prove $(1)$ because the proof of $(2)$ is similar. 

By an argument similar to the proof of Corollary \ref{contiad} without loss of generality we can assume that there exists the $L^2$-weak limit $\eta \in L^{2}(\bigwedge^{k+1}T^*X_{\infty})$ on $X_{\infty}$ of $\{d^{\upsilon_i}\omega_i\}_i$.

Let $\{\alpha_{\infty}^j\}_{0 \le j \le k} \subset \widetilde{\mathrm{Test}}T^1_0X_{\infty}$.
By Remark \ref{ggtttt} without loss of generality we can assume that there exist sequences $\{\alpha_i^j\}_{0 \le j \le k, i<\infty}$ of $\alpha_i^j \in \widetilde{\mathrm{Test}}T^1_0X_i$ with 
\[\sup_{j, i}||\alpha_i^j||_{L^{\infty}}<\infty\]
such that $\alpha_i^j$ $L^2$-converges strongly to $\alpha_{\infty}^j$ on $X_{\infty}$ for every $j$ and that 
$\nabla \alpha_i^j$ $L^2$-converges weakly to $\nabla \alpha_{\infty}^j$ on $X_{\infty}$ for every $j$.
Note that 
\begin{align}\label{bbv}
&\int_{X_i}d^{\upsilon_i}\omega_i(\alpha_i^0, \ldots, \alpha^k_i)d\upsilon_i \nonumber \\
&=\sum_l(-1)^{l+1}\int_{X_i}\omega_i(\alpha_i^0, \ldots, \alpha_{i}^{l-1}, \alpha_i^{l+1}, \ldots, \alpha_{i}^k)\mathrm{div}^{\upsilon_i}(\alpha_i^l)d\upsilon_i \nonumber \\
&+\sum_{l<m}(-1)^{l+m}\int_{X_i}\omega_i([\alpha_i^l, \alpha_i^m]^{\upsilon_i}, \alpha_i^0, \ldots, \alpha_i^{l-1}, \alpha_i^{l+1}, \ldots, \alpha_{i}^{m-1}, \alpha_{i}^{m+1}, \ldots, \alpha_i^k)d\upsilon_i
\end{align}
for every $i<\infty$. 
By Remarks \ref{aaaaaaa}, \ref{dsds} and Theorem \ref{hessc}, since $[\alpha_i^l, \alpha_i^m]^{\upsilon_i}$ $L^2$-converges weakly to $[\alpha_{\infty}^l, \alpha_{\infty}^m]^{\upsilon_{\infty}}$ on $X_{\infty}$, 
letting $i \to \infty$ in (\ref{bbv}) with \cite[Proposition $3.69$]{holp} yields $\eta=d^{\upsilon_{\infty}}\omega_{\infty}$.
This completes the proof.
\end{proof}
Let us consider the following natural question:

$\\ $
\textbf{Question $8$.} Do `$\sim$-versions' of Gigli's Sobolev spaces and the original one coincide?
$\\ $

The answer of this question is the following.
Note that Theorems \ref{aarrtt} and \ref{mnj} are direct consequences of Theorems \ref{techni2} and \ref{198183}.
\begin{theorem}\label{198183}
The question above has a positive answer, i.e. we have the following.
\begin{enumerate}
\item $\widetilde{\mathcal{D}}^2(\delta^{\upsilon}_k, X)=\mathcal{D}^2(\delta^{\upsilon}_k, X)$.
\item $\widetilde{H}^{1, 2}_d(\bigwedge^kT^*X)=H^{1, 2}_d(\bigwedge^kT^*X)$.
\item $\widetilde{H}^{1, 2}_H(\bigwedge^kT^*X)=H^{1, 2}_H(\bigwedge^kT^*X)$.
\item $\widetilde{\mathcal{D}}^2(\Delta^{\upsilon}_{H, k}, X)=\mathcal{D}^2(\Delta^{\upsilon}_{H, k}, X)$.
\item $\widetilde{H}^{2, 2}(X)=H^{2, 2}(X)$.
\item $\widetilde{W}^{2, 2}(X)=W^{2, 2}(X)$.
\item $\widetilde{W}^{1, 2}_d(\bigwedge^kT^*X)=W^{1, 2}_d(\bigwedge^kT^*X)$.
\item $\widetilde{W}^{1, 2}_C(T^r_sX)=W^{1, 2}_C(T^r_sX)$.
\item $\widetilde{H}^{1, 2}_C(T^r_sX)=H^{1, 2}_C(T^r_sX)$.
\item $\widetilde{W}^{1, 2}_H(\bigwedge^kT^*X)=W^{1, 2}_H(\bigwedge^kT^*X)$.
\end{enumerate}
\end{theorem}
\begin{proof}
We first prove $(1)$.

Let $\omega \in \widetilde{\mathcal{D}}^2(\delta^{\upsilon}_k, X)$ and let $\eta \in  \mathrm{TestForm}_{k-1}(X)$.
$(3)$ of Theorem \ref{density3} yields that there exists a sequence $\{\eta_i\}_{i<\infty}$ of $\eta_i \in \widetilde{\mathrm{TestForm}}_{k-1}(X)$ such that 
$\eta_i \to \eta$ in $W^{1, 2}_d(\bigwedge^{k-1}T^*X)$.
Since 
\[\int_X\langle \omega, d^{\upsilon}\eta_i \rangle d\upsilon=\int_X\langle \delta^{\upsilon}_k \omega, \eta_i \rangle d\upsilon\]
for every $i$, by letting $i \to \infty$ we have $\omega \in \mathcal{D}^2(\delta_k^{\upsilon}, X)$.
This completes the proof of $(1)$. 

$(2)$ is a direct consequence of $(3)$ of Theorem \ref{density3}. 

We now prove $(3)$.
Let 
\[\omega := \sum_{i=1}^Nf_{0, i}df_{1, i}\wedge \cdots \wedge df_{k, i} \in \mathrm{TestForm}_kX,\]
where $f_{j, i} \in \mathrm{Test}F(X)$.
By Proposition \ref{approx} for any $i, j$ there exists a sequence $\{f_{j, i, l}\}_l$ of $f_{j, i, l} \in \widetilde{\mathrm{Test}}F(X)$ such that $\sup_l\mathbf{Lip}f_{j, i, l}<\infty$, that $f_{j, i, l} \to f_{j, i}$ in $H^{1, 2}(X)$ and that $\Delta^{\upsilon}f_{j, i, l} \to \Delta^{\upsilon}f_{j, i}$ in $L^2(X)$ as $l \to \infty$ for any $i, j$.
In particular by Remark \ref{aaaaaaa} and Theorem \ref{hessc} we see that $[\nabla f_{j, i, l}, \nabla f_{\hat{j}, \hat{i}, l}]^{\upsilon} \to [\nabla f_{j, i}, \nabla f_{\hat{j}, \hat{i}}]^{\upsilon}$ in $L^2(TX)$ as $l \to \infty$ for any $i, j, \hat{i}, \hat{j}$.
Let 
\[\omega_l:=\sum_{i=1}^n f_{0, i, l}df_{1, i, l}\wedge \cdots \wedge df_{k, i, l} \in \widetilde{\mathrm{TestForm}}_k(X).\]
Then by \cite[Proposition $3.5.12$]{gigli} and \cite[Proposition $3.70$]{holp}
it is not difficult to check that $\omega_l, d^{\upsilon}\omega_l, \delta^{\upsilon}\omega_l$ $L^2$-converge strongly to $\omega, d^{\upsilon}\omega, \delta^{\upsilon}\omega$ on $X$, respectively.
This gives $(3)$.

$(4)$ is a direct consequence of $(3)$.

Next we prove $(5)$ (note that the following argument is essentially same to the proof of \cite[Proposition $3.3.18$]{gigli}).
Let $f \in \mathcal{D}^2(\Delta^{\upsilon}, X)$.
There exists a sequence $\{F_i\}_i$ of $F_i \in \mathrm{LIP}(X)$ such that $F_i \to \Delta^{\upsilon}f$ in $L^2(X)$ and that
\[\int_XF_id\upsilon=0.\]
Let $f_i:=(\Delta^{\upsilon})^{-1}F_i$. Theorem \ref{pois} with (\ref{lipreg}) gives that $f_i \in \widetilde{\mathrm{Test}}F(X)$ and that $f_i \to f$ in $H^{1, 2}(X)$.
Moreover Theorems \ref{33766} and \ref{hessc} yield that $\mathrm{Hess}_{f_i}^{\upsilon}$ $L^2$-converges weakly to $\mathrm{Hess}_f^{\upsilon}$ on $X$.
In particular $f_i$ converges weakly to $f$ in $W^{2, 2}(X)$.
Thus $f$ is in the closure of $\widetilde{\mathrm{Test}}F(X)$ with respect to the weak topology of $W^{2, 2}(X)$.
Since $\widetilde{\mathrm{Test}}F(X)$ is a linear subspace in $W^{2, 2}(X)$, we see that  $f$ is in the closure of $\widetilde{\mathrm{Test}}F(X)$ with respect to the strong topology of $W^{2, 2}(X)$ (c.f. Mazur's lemma), i.e.,
$f \in \widetilde{H}^{2, 2}(X)$. This completes the proof of $(5)$.  

We turn to the proof of $(6)$.
Let $f \in W^{2, 2}(X)$ and let $g_j \in \mathrm{Test}F(X)$, where $j \in \{0, 1, 2\}$.
By Proposition \ref{approx} for every $j$ there exists a sequence $\{g_{j, i}\}_i$ of $g_{j, i} \in \widetilde{\mathrm{Test}}F(X)$ such that 
$\sup_i\mathbf{Lip}g_{j, i}<\infty$, that $g_{j, i} \to g_j$ in $H^{1, 2}(X)$ and that $\Delta^{\upsilon}g_{j, i} \to \Delta^{\upsilon}g_j$ in $H^{1, 2}(X)$.
Note that by Theorem \ref{hessc} we see that $\langle \nabla g_{1, i}, \nabla g_{2, i}\rangle \to \langle \nabla g_{1}, \nabla g_2 \rangle$ in $H^{1, 2}(X)$.
Then since
\begin{align*}
&2\int_{X}g_0\left\langle \mathrm{Hess}_f^{\upsilon}, dg_{1, i} \otimes dg_{2, i} \right\rangle d\upsilon \nonumber \\
&=\int_X\left(-\langle \nabla f, \nabla g_{1, i} \rangle \mathrm{div}^{\upsilon}(g_{0, i}\nabla g_{2, i})-\langle \nabla f, \nabla g_{2, i}\rangle \mathrm{div}^{\upsilon}(g_{0, i}\nabla g_{1, i})- g_{0, i} \left\langle \nabla f, \nabla \left\langle \nabla g_{1, i}, \nabla g_{2, i} \right\rangle\right\rangle \right)d\upsilon,
\end{align*}
by letting $i \to \infty$, we have $f \in \widetilde{W}^{2, 2}(X)$.
This completes the proof of $(6)$.

Next we prove $(7)$.
We only give the proof in the case when $k=1$ only for simplicity because the proof in the case when $k \ge 2$ is similar.

Let $\omega \in \widetilde{W}^{1, 2}_d(\bigwedge^kT^*X)$ and let 
\[\alpha :=\sum_{l=1}^Nf_{1, \infty}^l\nabla f_{2, \infty}^l, \beta := \sum_{l=1}^Nf_{3, \infty}^l\nabla f_{4, \infty}^l \in \mathrm{Test}T^1_0(X),\]
where $f_{j, \infty}^l \in \mathrm{Test}F(X)$.
By Proposition \ref{approx}, for any $j, l$,
there exists a sequence $\{f_{j, i}^l\}_{i}$ 
of $f_{j, i}^l \in \widetilde{\mathrm{Test}}F(X)$ such that $\sup_{i}\mathbf{Lip}f_{j, i}^l<\infty$ and that $f_{j, i}^l, \Delta^{\upsilon}f_{j, i}^l \to f_{j, \infty}^l, \Delta^{\upsilon}f_{j, \infty}^l$ in $H^{1, 2}(X)$, respectively.

Let $\alpha_i:=\sum_{l=1}^Nf_{1, i}^l\nabla f_{2, i}^l, \beta_i := \sum_{l=1}^Nf_{3, i}^l\nabla f_{4, i}^l \in \widetilde{\mathrm{Test}}T^1_0(X)$.
Note that
\begin{align}\label{nnbbgghh}
[\alpha_i, \beta_i]^{\upsilon} =\sum_{l, m}\left( f_{3, i}^m \langle \nabla f_{1, i}^l, \nabla f_{4, i}^m \rangle \nabla f_{2, i}^l - f_{1, i}^l\langle \nabla f_{2, i}^l, \nabla f_{3, i}^m \rangle \nabla f_{4, i}^m + f_{1, i}^lf_{3, i}^m[\nabla f_{2, i}^l, \nabla f_{4, i}^m]^{\upsilon}\right).
\end{align}
By $(1)$, Theorem \ref{hessc} and Remark \ref{dsds} we see that $f_{1, i}^lf_{3, i}^m$ converges uniformly to $f_{1, \infty}^lf_{3, \infty}^m$ on $X$ and that $[\nabla f_{2, i}^l, \nabla f_{4, i}^m]^{\upsilon} \to [\nabla f_{2, \infty}^l, \nabla f_{4, \infty}^m]^{\upsilon}$ in $L^2(TX)$.
In particular we have
\begin{align}\label{v}
\lim_{i \to \infty}\int_Xf_{1, i}^lf_{3, i}^m\omega \left([\nabla f_{2, i}^l, \nabla f_{4, i}^m]^{\upsilon}\right)d\upsilon =\int_{X}f_{1, \infty}^lf_{3, \infty}^m \omega \left([\nabla f_{2, \infty}^l, \nabla f_{4, \infty}^m]^{\upsilon}\right)d\upsilon.
\end{align}
Since
\begin{align}\label{1232123}
\int_Xd^{\upsilon}\omega \left(\alpha_i, \beta_i\right)d\upsilon =\int_X\left( \omega(\alpha_i)\mathrm{div}^{\upsilon}(\beta_i)-\omega (\beta_i)\mathrm{div}^{\upsilon}(\alpha_i)-\omega \left([\alpha_i, \beta_i]^{\upsilon}\right)\right)d\upsilon,
\end{align}
for every $i<\infty$,
by (\ref{nnbbgghh}) and (\ref{v}), letting $i \to \infty$ in (\ref{1232123}) gives
\begin{align*}
\int_Xd^{\upsilon}\omega \left(\alpha, \beta\right)d\upsilon =\int_X\left( \omega(\alpha )\mathrm{div}^{\upsilon}(\beta )-\omega (\beta)\mathrm{div}^{\upsilon}(\alpha)-\omega \left([\alpha, \beta]^{\upsilon}\right)\right)d\upsilon.
\end{align*}
Thus we have $\omega \in W^{1, 2}_d(\bigwedge^kT^*X)$.
This completes the proof of $(7)$.

Similarly we have $(8)$.

Next we give a proof of $(9)$.
Let 
\[T:=\sum_{k=1}^Nh_0^k\nabla^r_s h^k \in \mathrm{Test}T^r_sX,\]
where $h^k=(h_1^k, \ldots, h_{r+s}^k)$ and $h^k_j \in \mathrm{Test}F(X)$.
By Proposition \ref{approx} for any $j, k$ there exists a sequence $\{h_{j, i}^k\}_i$ of $h_{j, i}^k \in \widetilde{\mathrm{Test}}F(X)$ such that $\sup_i\mathbf{Lip}h^k_{j, i}<\infty$ and that $h_{j, i}^k, \Delta^{\upsilon}h_{j, i}^k \to h_j^k, \Delta^{\upsilon}h_{j}^k$ in $H^{1, 2}(X)$ as $i \to \infty$, respectively.
Let $h^{k, i}:=(h_{1, i}^k, \ldots, h_{r+s, i}^k)$ and let
\[T_i:=\sum_{k=1}^Nh_{0, i}^k\nabla^r_s h^{k, i} \in \widetilde{\mathrm{Test}}T^r_sX.\]
By Remark \ref{aaaaaaa} and Theorem \ref{hessc}, it is easy to check that $T_i$ converges weakly to $T$ in $W^{1, 2}_C(T^r_sX)$.
This implies that $T$ is in the closure of $\widetilde{\mathrm{Test}}T^r_sX$ with respect to the weak topology of $W^{1, 2}_C(T^r_sX)$.
Since $\widetilde{\mathrm{Test}}T^r_sX$ is a linear subspace of $W^{1, 2}_C(T^r_sX)$, we have $T \in \widetilde{H}^{1, 2}_C(T^r_sX)$.
Thus we have $(9)$.

$(10)$ is a direct sequence of $(1)$ and $(7)$
\end{proof}
\begin{corollary}\label{66tytt}
Let $(X, \upsilon) \in \overline{M(n, K, d)}$ with $\mathrm{diam}\,X>0$, and let $\omega \in H^{1, 2}_C(\bigwedge^kT^*X)$. 
Then the same conclusion of Theorem \ref{aarrtt} holds.
\end{corollary}
\begin{proof}
By Remark \ref{ggtttt} and Theorem \ref{aarrtt} we see that this holds if $\omega \in \widetilde{\mathrm{TestForm}}_kX$.
Thus Theorem \ref{198183} yields the assertion.
\end{proof}
Similarly by Theorem \ref{mnj} and Remark \ref{ggtttt} we have the following.
\begin{corollary}\label{99i99i}
Let $(X, \upsilon) \in \overline{M(n, K, d)}$ with $\mathrm{diam}\,X>0$, and let $T \in H^{1, 2}_C(T^r_sX)$.
Then we see that $T$ is differentiable at a.e. $x \in X$, that $T \in W_{2n/(n-1)}(T^r_sX)$, that $|T|^2 \in H^{1, 2n/(2n-1)}(X)$ and that $\nabla^{g_X}T=\nabla^{\upsilon}T$.
\end{corollary}
We give another relationship between Sobolev spaces.
\begin{theorem}\label{jjhhgggg}
Let $(X, \upsilon) \in \overline{M(n, K, d)}$ with $\mathrm{dim}\,X=n$.
Then we have $H^{1, 2}_H(T^*X)=H^{1, 2}_C(T^*X)$ as sets.
Moreover the identity map 
\[\mathrm{id}: H^{1, 2}_H(T^*X) \to H^{1, 2}_C(T^*X)\]
gives a homeomorphism.
\end{theorem}
\begin{proof}
Let $\omega_i$ be a sequence in $\mathrm{TestForm}_1X$.
It suffices the check 
that the following are equivalent:
\begin{enumerate}
\item $\omega_i$ is a Cauchy sequence with respect to the norm $|| \cdot||_{H^{1, 2}_H}$.
\item  $\omega_i$ is a Cauchy sequence with respect to the norm $|| \cdot||_{H^{1, 2}_C}$.
\end{enumerate}

We first assume that (1) holds.
By Theorem \ref{boch},
since
\[\int_X|\nabla (\omega_i-\omega_j)|^2dH^n \le \int_X\left( |d(\omega_i-\omega_j)|^2+|\delta (\omega_i -\omega_j)|^2-K(n-1)|\omega_i-\omega_j|^2 \right)dH^n\]
for any $i, j$, we have (2).

Next we assume that (2) holds.
In order to prove that $\delta \omega_i$ is a Cauchy sequence in $L^2(X)$, we prepare the following:
\begin{claim}\label{deltaformula}
For every $\eta \in \mathrm{TestForm}_1X$, we have
\[\delta \omega= -\mathrm{tr}(\nabla \omega).\]
\end{claim}
The proof is as follows.
By Remark \ref{ggtttt} with Theorem \ref{198183}, there exist a sequence $(X_i, \upsilon_i) \in M(n, K, d)$ and a sequence $\omega_i \in C^{\infty}(T^*X_i)$ such that $(X_i, \upsilon_i) \stackrel{GH}{\to} (X, \upsilon)$, that $\omega_i, \delta \omega_i, \nabla \omega_i$ $L^2$-converge weakly to $\omega, \delta \omega, \nabla \omega$ on $X$, respectively.

Thus \cite[Proposition $3.72$]{holp} yields that $\mathrm{tr}(\nabla \omega_i)$ $L^2$-converges weakly to $\mathrm{tr}(\nabla \omega)$ on $X$.
Since $\delta \omega_i=-\mathrm{tr}(\nabla \omega_i)$, this completes the proof of Claim \ref{deltaformula}.

Claim \ref{deltaformula} yields that $\delta \omega_i$ is a Cauchy sequence in $L^2(X)$.

On the other hand by (\ref{dna}),  it is easy to check that 
 $d\omega_i$ $L^2$-converges strongly to $d\omega$ on $X$.
This completes the proof of Proposition \ref{jjhhgggg}. 
\end{proof}
\begin{remark}
By using Proposition \ref{approx}, for any $(X, \upsilon) \in \overline{M(n, K, d)}$ and $f \in \mathcal{D}^2(\Delta^{\upsilon}, X)$, it is easy to check that $df \in H^{1, 2}_H(T^*X)$.
\end{remark}
\begin{remark}
Let $(X, \upsilon) \in \overline{M(n, K, d)}$ with $\mathrm{dim}\,X=n$.
By Theorems \ref{33766}, \ref{aarrtt}, \ref{app6}, \ref{boch} and Corollary \ref{expb} we have the following two Bochner inequalities:
\begin{itemize}
\item For any $f \in \mathcal{D}^2(\Delta^{\upsilon}, X)$ and $\phi \in \mathrm{LIP}(X)$ with $\phi \ge 0$ we have
\begin{align*}
-\frac{1}{2}\int_{X}\langle d\phi, d|df|^2 \rangle d\upsilon &\ge \int_{X}\phi|\mathrm{Hess}_{f}^{\upsilon}|^2d\upsilon + \int_{X}\left(-\phi(\Delta^{\upsilon}f)^2+\Delta^{\upsilon}f\langle d\phi, df\rangle \right)d\upsilon \nonumber \\
&+K(n-1)\int_{X}\phi|df|^2d\upsilon.
\end{align*}
\item For any $\omega \in H_H^{1,2}(T^*X)$ and $\phi \in \mathrm{LIP}(X)$ with $\phi \ge 0$ we have
\begin{align*}
-\frac{1}{2}\int_{X} \left\langle d\phi, d|\omega|^2 \right\rangle d\upsilon 
&\ge \int_{X}\phi|\nabla^{\upsilon} \omega|^2d\upsilon -\int_{X}\left(\delta^{\upsilon} (\phi \omega)\delta^{\upsilon} \omega+\langle d^{\upsilon}(\phi \omega), d^{\upsilon}\omega \rangle \right)d\upsilon \\
&+K(n-1)\int_{X}\phi |\omega|^2d\upsilon.
\end{align*}
\end{itemize}
Note that these are already proved by Gigli via different approaches on $RCD(K(n-1), \infty)$-spaces.
See \cite[Lemma $3.6.2$]{gigli}.
\end{remark}
We give an application.
Recall that if $M$ is a closed nonnegatively Ricci curved Riemannian manifold and $\omega$ is a harmonic $1$-form on $M$, then $\omega$ is parallel, i.e., 
\[\nabla \omega \equiv 0.\]
The following is a generalization of this to the Gromov-Hausdorff setting.
\begin{theorem}\label{q223}
Let $\{\epsilon_i\}_{i<\infty}$ be a sequence of $\epsilon_i \ge 0$ with $\epsilon_i \to 0$,
let $\{(X_i, \upsilon_i)\}_{i<\infty}$ be a sequence of $(X_i, \upsilon_i) \in M(n, -\epsilon_i, d)$,
let $(X_{\infty}, \upsilon_{\infty})$ be the Gromov-Hausdorff limit of them with $\mathrm{diam}\,X_{\infty}>0$,
let $\{\omega_i\}_{i<\infty}$ be a sequence of $\omega_i \in C^{\infty}(T^*X_i)$ with 
\[\lim_{i \to \infty}\int_{X_i}\left(|d\omega_i|^2+|\delta \omega_i|^2\right)d\upsilon_i=0,\]
and let $\omega_{\infty}$ be the $L^2$-strong limit on $X_{\infty}$ of them. 
Then we see that $\omega_{\infty} \in \mathcal{D}^2(\Delta^{\upsilon_{\infty}}_{H, 1}, X_{\infty}) \cap W^{1, 2}_C(T^*X_{\infty})$, 
that $\Delta_{H, 1}^{\upsilon_{\infty}}\omega_{\infty}=0$ and that $d\omega_i, \delta \omega_i, \nabla \omega_i$ $L^2$-converge strongly to zeros on $X_{\infty}$, respectively.
In particular we have $d^{\upsilon_{\infty}}\omega_{\infty}=d\omega_{\infty}=0$, $\delta^{\upsilon_{\infty}}\omega_{\infty}=\delta^{g_{X_{\infty}}}\omega_{\infty}=0$ and $\nabla^{\upsilon_{\infty}}\omega_{\infty}=\nabla^{g_{X_{\infty}}} \omega_{\infty}=0$.
\end{theorem}
\begin{proof}
The Bochner formula gives 
\[-\frac{1}{2}\Delta |\omega_i|^2\ge |\nabla \omega_i|^2-\langle \Delta_{H, 1} \omega_i, \omega_i \rangle -\epsilon_i|\omega_i|^2.\]
Integrating this gives
\begin{align*}
\int_{X_i}|\nabla \omega_i|^2d\upsilon_i\le \int_{X_i}\left( |d\omega_i|^2 +|\delta \omega_i|^2\right)d\upsilon_i +\epsilon_i\int_{X_i}|\omega_i|^2d\upsilon_{i}.
\end{align*}
By letting $i \to \infty$ we have 
\begin{align}\label{uuuuu}
\lim_{i \to \infty}\int_{X_i}|\nabla \omega_i|^2d\upsilon_{i}=0.
\end{align}
Thus the assertion follows from Corollaries \ref{contiad}, Theorems \ref{mthm}, \ref{techni2}, \ref{198183}, (\ref{uuuuu}) and \cite[Proposition $3.74$]{holp}.
\end{proof}
We now give a closedness of the Hodge Laplacian with respect to the Gromov-Hausdorff topology:
\begin{theorem}\label{hodg}
Let $(X_i, \upsilon_i) \stackrel{GH}{\to} (X_{\infty}, \upsilon_{\infty})$ in $\overline{M(n, K, d)}$ with $\mathrm{diam}\,X_{\infty}>0$, and 
let $\{\omega_i\}_{i\le \infty}$ be an $L^2$-strong convergent sequence on $X_{\infty}$ of $\omega_i \in L^2(\bigwedge^kT^*X_i)$.
Assume that $\omega_i \in \mathcal{D}^2(\Delta^{\upsilon_i}_{H, k}, X_i)$ for every $i<\infty$, that
\begin{align}\label{munerin}
\sup_{i<\infty}\int_{X_i}\left(|d^{\upsilon_i}\omega_i|^2+|\delta^{\upsilon_i}\omega_i|^2+|\Delta_{H, k}^{\upsilon_i}\omega_i|^2\right)d\upsilon_i<\infty.
\end{align}
and that one of the following two conditions holds:
\begin{enumerate}
\item $k=1$.
\item $\delta^{\upsilon_i}\omega_i$ $L^2$-converges strongly to $\delta^{\upsilon_{\infty}}\omega_{\infty}$ on $X_{\infty}$.
\end{enumerate}
Then we see that $\omega_{\infty} \in \mathcal{D}^2(\Delta^{\upsilon_{\infty}}_{H, k}, X_{\infty})$ and that $\Delta_{H, k}^{\upsilon_i}\omega_i$ $L^2$-converges weakly to $\Delta^{\upsilon_{\infty}}_{H, k}\omega_{\infty}$ on $X_{\infty}$.
\end{theorem}
\begin{proof}
We give a proof in the case when $k=1$ only because  by Remark \ref{ggtttt}, the proof in the other case is similar.
By an argument similar to the proof of Corollary \ref{contiad} without loss of generality we can assume that
 there exists the $L^2$-weak limit $\eta_{\infty} \in L^{2}(T^*X_{\infty})$  on $X_{\infty}$ of $\{\Delta_{H, 1}^{\upsilon_{i}}\omega_{i}\}_i$.
 
Let $\alpha_{\infty} \in \widetilde{\mathrm{TestForm}}_1(X_{\infty})$.
By Theorem \ref{app6} and Remark \ref{ggtttt}, without loss of generality we can assume that there exists a sequence $\{\alpha_{i}\}_i$ of $\alpha_{i} \in \widetilde{\mathrm{TestForm}}_1(X_i)$ such that $\alpha_{i}, d^{\upsilon_{i}}\alpha_{i}, \delta^{\upsilon_{i}} \alpha_{i}$ $L^2$-converge strongly to $\alpha_{\infty}, d^{\upsilon_{\infty}}\alpha_{\infty}, \delta^{\upsilon_{\infty}}\alpha_{\infty}$ on $X_{\infty}$, respectively.

Then since 
\[\int_{X_{i}}\langle \alpha_{i}, \Delta_{H, 1}^{\upsilon_{i}}\omega_{i}\rangle d\upsilon_{i}=\int_{X_{i}}\left(\langle d^{\upsilon_{i}}\alpha_{i}, d^{\upsilon_{i}}\omega_{i}\rangle + (\delta^{\upsilon_{i}} \alpha_{i})(\delta^{\upsilon_{i}}\omega_{i})\right)d\upsilon_{i}\]
by letting $i \to \infty$, Corollary \ref{contiad}, Theorems \ref{techni2} and  \ref{198183} give that $\omega_{\infty} \in W^{1, 2}_H(T^*X)$ and that 
\[\int_{X_{\infty}}\langle \alpha_{\infty}, \eta_{\infty}\rangle d\upsilon_{\infty}=\int_{X_{\infty}}\left(\langle d^{\upsilon_{\infty}}\alpha_{\infty}, d^{\upsilon_{\infty}}\omega_{\infty}\rangle + (\delta^{\upsilon_{\infty}} \alpha_{\infty}) (\delta^{\upsilon_{\infty}}\omega_{\infty})\right)d\upsilon_{\infty}.\]
Thus we have $\omega_{\infty} \in \mathcal{D}^2(\Delta_{H, 1}^{\upsilon_{\infty}}, X_{\infty})$ and $\Delta_{H, 1}^{\upsilon_{\infty}}\omega_{\infty}=\eta_{\infty}$.
This completes the proof.
\end{proof}
We now give an $L^{\infty}$-estimate for an eigenform:
\begin{proposition}\label{aa33}
Let $0 \le \lambda \le L$, let $(X, \upsilon) \in M(n, K, d)$, and let $\omega \in C^{\infty}(T^*X)$ with $||\omega||_{L^2(X)}\le L$ and 
\[\Delta_{H, 1}\omega = \lambda \omega.\]
Then we have 
\[||\omega||_{L^{\infty}(X)}\le C(n, K, d, L).\]
\end{proposition}
\begin{proof}
The Bochner formula gives
\begin{align*}
-\frac{1}{2}\Delta |\omega|^2&\ge |\nabla \omega|^2 - \langle \Delta_{H, 1} \omega, \omega \rangle +K(n-1)|\omega|^2 \\
&\ge -\frac{1}{2}|\Delta_{H, 1}\omega|^2 -\frac{1}{2}|\omega|^2 +K(n-1)|\omega|^2\\
&\ge -\frac{L^2}{2}|\omega|^2 - \frac{1}{2}|\omega|^2- |K|(n-1)|\omega|^2 \\
&\ge -\left(\frac{L^2}{2}+\frac{1}{2}+|K|(n-1)\right)|\omega|^2.
\end{align*}
Thus Li-Tam's mean value inequality \cite[Corolalry $3.6$]{L1} yields
\[|\omega|\le C(n, K, d, L)\int_{X}|\omega|^2d\upsilon= C(n, K, d, L).\]
This completes the proof.
\end{proof}
\begin{remark}
By an argument similar to the proof of Proposition \ref{aa33} we have the following:
Let $L>0$, let $(X, \upsilon) \in M(n, K, d)$, and let $\omega \in C^{\infty}(T^*X)$ with
\[||\omega||_{L^{2}(X)}+||\Delta_{H, 1}\omega||_{L^{\infty}(X)}\le L.\]
Then we have
\[||\omega||_{L^{\infty}(X)}\le C(n, K, d, L).\]
\end{remark}
Theorem \ref{eigenfcont} is a direct consequence of the following:
\begin{theorem}\label{hodgelaplacian}
Let $\{\lambda_i\}_{i <\infty}$ be a bounded sequence in $\mathbf{R}$, let $\{(X_i, \upsilon_i)\}_{i<\infty}$ be a sequence in $M(n, K, d)$, let $(X_{\infty}, \upsilon_{\infty})$ be the Gromov-Hausdorff limit of them with $\mathrm{diam}\,X_{\infty}>0$, let $\{\omega_i\}_{i<\infty}$ be a sequence of $\lambda_i$-eigenforms $\omega_i \in C^{\infty}(\bigwedge^kT^*X_i)$ with $||\omega_i||_{L^2(X_i)}=1$, and let $\omega_{\infty}$ be the $L^2$-strong limit on $X_{\infty}$ of them.
Assume that one of the following two conditions holds:
\begin{enumerate}
\item $k=1$.
\item $\delta^{\upsilon_i}\omega_i$ $L^2$-converges strongly to $\delta^{\upsilon_{\infty}}\omega_{\infty}$ on $X_{\infty}$.
\end{enumerate}
Then we see that the limit
\[\lim_{i \to \infty}\lambda_i\]
exists and that $\omega_{\infty} \in \mathcal{D}^2(\Delta_{H, k}^{\upsilon_{\infty}}, X_{\infty})$ with
\[\Delta^{\upsilon_{\infty}}_{H, k}\omega_{\infty}=\left(\lim_{i \to \infty}\lambda_i\right)\omega_{\infty}.\]
\end{theorem}
\begin{proof}
Since
\begin{align*}
\int_{X_i}\left(|d\omega_i|^2+|\delta \omega_i|^2\right)d\upsilon_i&=\int_{X_i}\omega_i \Delta_{H, k}\omega_i d\upsilon_i \le \left( \int_{X_i}|\omega_i|^2d\upsilon_i\right)^{1/2}\left(\int_{X_i}|\Delta_{H, k} \omega_i|^2d\upsilon_i\right)^{1/2}
\end{align*}
for every $i<\infty$, we see that (\ref{munerin}) holds.

Let $\{i(j)\}_j$ be a subsequence of $\mathbf{N}$.
There exist a subsequence $\{j(l)\}_l$ of $\{i(j)\}_j$ and $\lambda \in [0, \infty)$ such that $\lambda_{j(l)} \to \lambda$.
Theorem \ref{hodg} yields that $\omega_{\infty} \in \mathcal{D}^2(\Delta^{\upsilon_{\infty}}_{H, k}, X_{\infty})$ and $\Delta_{H, k}^{\upsilon_{\infty}}\omega_{\infty}=\lambda \omega_{\infty}$.
Since $\lambda=||\Delta_{H, k}^{\upsilon_{\infty}}\omega_{\infty}||_{L^2}$ and $\{i(j)\}$ is arbitrary, this completes the proof.
\end{proof}
The following means that roughly speaking, in the Gromov-Hausdorff setting, the $L^2$-strong limit of a sequence of harmonic forms with uniform bounds on the $L^2$-energies is also harmonic: 
\begin{corollary}\label{harmonicha}
Let $\{(X_i, \upsilon_i)\}_{i<\infty}$ be a sequence in $M(n, K, d)$, let $(X_{\infty}, \upsilon_{\infty})$ be the Gromov-Hausdorff limit of them with $\mathrm{diam}\,X_{\infty}>0$, let $\{\omega_i\}_{i<\infty}$ be a sequence of harmonic forms $\omega_i \in C^{\infty}(\bigwedge^kT^*X_i)$ with (\ref{gyy}),
and let $\omega_{\infty}$ be the $L^2$-strong limit on $X_{\infty}$ of them.
Then we see that $\omega_{\infty} \in \mathcal{D}^2(\Delta_{H, k}^{\upsilon_{\infty}}, X_{\infty})$ with
\[\Delta^{\upsilon_{\infty}}_{H, k}\omega_{\infty}=0.\]
\end{corollary}
\begin{proof}
It is a direct consequence of Corollary \ref{contiad} and Theorems \ref{pplki}, \ref{techni} and \ref{hodgelaplacian}.
\end{proof}
\begin{remark}
Let $\{(X_i, \upsilon_i)\}_{i<\infty}$ be a sequence in $M(n, K, d)$, let $(X_{\infty}, \upsilon_{\infty})$ be the Gromov-Hausdorff limit of them. Define \textit{a generalized first betti number of $(X_{\infty}, \upsilon_{\infty})$} by
\[b_1(X_{\infty}):=\mathrm{dim}\,\{\omega \in \mathcal{D}^2(\Delta_{H, 1}^{\upsilon_{\infty}}, X_{\infty}); \Delta_{H, 1}^{\upsilon_{\infty}}\omega =0\}.\]
Then Corollary \ref{harmonicha} with Theorem \ref{3ew3} yields that if $\mathrm{dim}\,X_{\infty}=n$, then we see that the upper semicontinuity of the first betti numbers with respect to the Gromov-Hausdorff topology:
\begin{align}\label{aho}
\limsup_{i \to \infty}b_1(X_i)\le b_1(X_{\infty}).
\end{align}
Note that we do not know whether the equality of (\ref{aho}) holds and $b_1(X_{\infty})$ is equal to the first betti number in Gigli's sense, i.e.,
\[b_1(X_{\infty})=\mathrm{dim}\,\left(\{\omega \in \mathcal{D}^2(\Delta_{H, 1}^{\upsilon_{\infty}}, X_{\infty}); \Delta_{H, 1}^{\upsilon_{\infty}}\omega =0\} \cap H^{1, 2}_H(T^*X_{\infty})\right).\]
However it is worth pointing out that if $H^{1, 2}_H(T^*X_{\infty})=W^{1, 2}_H(T^*X_{\infty})$, then the equality above holds. 
\end{remark}
We now give a sufficient condition for the $L^2$-strong convergence of $\{d\omega_i\}_i$ and $\{\delta \omega_i\}_i$ which is a generalization of Theorem \ref{convlap} to the case of differential forms:
\begin{theorem}\label{444}
Let us consider the same assumption as in Theorem \ref{hodg}. 
Assume that $(X_i, \upsilon_i) \in M(n, K, d)$ for every $i<\infty$, that $\omega_i \in C^{\infty}(\bigwedge^kT^*X_i)$ for every $i<\infty$ and that $\omega_{\infty}$ has a smooth $W^{1, 2}_H$-approximation with respect to $\{(X_i, \upsilon_i)\}_i$.
Then $d\omega_i, \delta \omega_i$ $L^2$-converge strongly to $d^{\upsilon_{\infty}}\omega_{\infty}, \delta^{\upsilon_{\infty}}\omega_{\infty}$ on $X_{\infty}$, respectively.
\end{theorem}
\begin{proof}
Let $\{\eta_i\}_i$ be a smooth $W^{1, 2}_H$-approximation of $\omega_{\infty}$ with respect to $\{(X_i, \upsilon_i)\}_i$.
We first prove the following. Compare with Claim \ref{abc}.
\begin{claim}\label{83}
For every $i<\infty$ we have
\begin{align}\label{44}
\int_{X_i}\left( |d\eta_i|^2+|\delta \eta_i|^2\right)d\upsilon_i-2 \int_{X_i}\langle \Delta_{H, k}\omega_i, \eta_i\rangle d\upsilon _i\ge \int_{X_i}\left( |d\omega_i|^2+|\delta \omega_i|^2\right)d\upsilon_i - 2 \int_{X_i}\langle \Delta_{H, k}\omega_i, \omega_i\rangle d\upsilon_i.
\end{align}
\end{claim}
It follows from the identity
\begin{align*}
&\int_{X_i}\left( |d\eta_i|^2+|\delta \eta_i|^2\right)d\upsilon_i-2 \int_{X_i}\langle \Delta_{H, k}\omega_i, \eta_i\rangle d\upsilon_i \\
&=\int_{X_i}\left( |d\omega_i|^2+|\delta \omega_i|^2\right)d\upsilon_i - 2 \int_{X_i}\langle \Delta_{H, k}\omega_i, \omega_i\rangle d\upsilon_i + \int_{X_i}\left(|d(\omega_i -\eta_i)|^2+|\delta (\omega_i- \eta_i)|^2\right)d\upsilon_i.
\end{align*}

Letting $i \to \infty$ in (\ref{44}) with Theorem \ref{hodg} yields 
\begin{align*}
\limsup_{i \to \infty}\int_{X_i}\left(|d\omega_i|^2+|\delta \omega_i|^2\right)d\upsilon_i\le \int_{X_{\infty}}\left( |d^{\upsilon_{\infty}}\omega_{\infty}|^2+|\delta^{\upsilon_{\infty}}\omega_{\infty}|^2\right)d\upsilon_{\infty}.
\end{align*}
On the other hand, since Corollary \ref{contiad} and Theorem \ref{techni2} give 
\[\liminf_{i \to \infty}\int_{X_i}|d\omega_i|^2d\upsilon_i \ge \int_{X_{\infty}}|d^{\upsilon_{\infty}}\omega_{\infty}|^2d\upsilon_{\infty},\,\,\,\liminf_{i \to \infty}\int_{X_i}|\delta \omega_i|^2d\upsilon_i \ge \int_{X_{\infty}}|\delta^{\upsilon_{\infty}}\omega_{\infty}|^2d\upsilon_{\infty},\]
we have
\[\lim_{i \to \infty}\int_{X_i}|d\omega_i|^2d\upsilon_i = \int_{X_{\infty}}|d^{\upsilon_{\infty}}\omega_{\infty}|^2d\upsilon_{\infty},\,\,\,\lim_{i \to \infty}\int_{X_i}|\delta \omega_i|^2d\upsilon_i = \int_{X_{\infty}}|\delta^{\upsilon_{\infty}}\omega_{\infty}|^2d\upsilon_{\infty}.\]
This completes the proof.
\end{proof}
\begin{remark}
By Theorem \ref{app6} and the proof of Theorem \ref{444}, we have the following:
Let $(X_i, \upsilon_i) \stackrel{GH}{\to} (X_{\infty}, \upsilon_{\infty})$ in $\overline{M(n, K, d)}$, let $\{\omega_i\}_{i \le \infty}$ be an $L^2$-strong convergent sequence of $\omega_i \in \mathcal{D}^2(\Delta^{\upsilon_i}_{H, k}, X_i) \cap H^{1, 2}_H(\bigwedge^kT^*X_i)$ with 
\[\sup_{i}\int_{X_i} |\Delta^{\upsilon_i}_{H, k}\omega_i|^2 d\upsilon_i<\infty.\]
Assume that $k=1$, or that $\delta^{\upsilon_i}\omega_i$ $L^2$-converges strongly to $\delta^{\upsilon_{\infty}}\omega_{\infty}$ on $X_{\infty}$.
Then we see that $d^{\upsilon_i}\omega_i$, $\delta^{\upsilon_i}\omega_i$ $L^2$-converge strongly to $d^{\upsilon_{\infty}}\omega_{\infty}$, $\delta^{\upsilon_{\infty}}\omega_{\infty}$ on $X_{\infty}$, respectively and that $\Delta^{\upsilon_i}_{H, 1}\omega_i$ $L^2$-converges weakly to $\Delta^{\upsilon_{\infty}}_{H, 1}\omega_{\infty}$ on $X_{\infty}$.
\end{remark}
As a summary of this subsection we give the following two compactness for $1$-forms in noncollapsed setting:
\begin{theorem}\label{bbbbb}
Let $\{(X_i, \upsilon_i)\}_{i<\infty}$ be a sequence in $M(n, K, d)$, let $(X_{\infty}, \upsilon_{\infty})$ be the noncollapsed Gromov-Hausdorff limit of them, and
let $\{\omega_i\}_{i<\infty}$ be a sequence of $\omega_i \in C^{\infty}(T^*X_i)$ with (\ref{mmune}) and 
\[\sup_{i<\infty}\int_{X_i}|\omega_i|^2d\upsilon_i<\infty.\]
Then there exist $\omega_{\infty} \in W^{1, 2}_C(T^*X_{\infty}) \cap W^{1, 2}_H(T^*X_{\infty}) \cap W_{2n/(n-1)}(T^*X_{\infty})$ and a subsequence of $\{i(j)\}_j$ such that the following hold.
\begin{enumerate}
\item $\omega_{\infty}$ is the $L^r$-strong limit  on $X_{\infty}$ of $\{\omega_{i(j)}\}_j$ for every $1<r<2n/(n-1)$.
\item $\nabla^{g_{X_{\infty}}}\omega_{\infty} = \nabla^{\upsilon_{\infty}}\omega_{\infty}$ and it is the $L^2$-weak limit  on $X_{\infty}$ of $\{\nabla \omega_{i(j)}\}_j$.
\item $d\omega_{\infty}=d^{\upsilon_{\infty}}\omega_{\infty}$ and it is the $L^2$-weak limit  on $X_{\infty}$ of $\{d \omega_{i(j)}\}_j$.
\item $\delta^{g_{X_{\infty}}}\omega_{\infty}=\delta^{\upsilon_{\infty}}\omega_{\infty}$ and it is the $L^2$-weak limit  on $X_{\infty}$ of $\{\delta \omega_{i(j)}\}_j$.
\end{enumerate}
\end{theorem}
\begin{proof}
It follows directly from Corollary \ref{nnbbmm}, Proposition \ref{bounds}, Theorems \ref{boch}, \ref{techni2} and \ref{198183}. 
\end{proof}
\begin{theorem}\label{cccc}
Let $\{(X_i, \upsilon_i)\}_{i<\infty}$ be a sequence in $M(n, K, d)$, let $(X_{\infty}, \upsilon_{\infty})$ be the noncollapsed Gromov-Hausdorff limit of them, and
let $\{\omega_i\}_{i<\infty}$ be a sequence of $\omega_i \in C^{\infty}(T^*X_i)$ with 
\[\sup_{i<\infty}\int_{X_i}\left(|\omega_i|^2+|\Delta_{H, 1}\omega_i|^2\right)d\upsilon_i<\infty.\]
Then there exist $\omega_{\infty} \in W^{1, 2}_C(T^*X_{\infty}) \cap \mathcal{D}^2(\Delta_{H, 1}^{\upsilon_{\infty}}, X_{\infty}) \cap W_{2n/(n-1)}(T^*X_{\infty})$ and a subsequence $\{i(j)\}_j$ such that the same conclusions of Theorem \ref{bbbbb} hold and that
$\Delta_{H, 1}^{\upsilon_i}\omega_i$ $L^2$-converges weakly to $\Delta^{\upsilon_{\infty}}_{H, 1}\omega_{\infty}$ on $X_{\infty}$.
\end{theorem}
\begin{proof}
It is a direct consequence of Theorems \ref{hodg} and \ref{bbbbb}.
\end{proof}
\begin{remark}
Note that roughly speaking, most results for `smooth objects' shown in this paper are also hold for `objects having smooth approximations'.
For example, by Theorems \ref{app6}, \ref{mthm}, \ref{boch}, \ref{techni2}, Corollary \ref{contiad} and Proposition \ref{bounds}, for any $(X, \upsilon) \in \overline{M(n, K, d)}$ and $\omega \in H^{1, 2}(T^*X)$ with 
\[\int_X\left(|d^{\upsilon}\omega|^2+|\delta^{\upsilon}\omega|^2\right)d\upsilon \le L,\]
we see that $\omega \in W_{2n/(n-1)}(T^*X) \cap W^{1, 2}_C(T^*X)$, that $|\omega|^2 \in H^{1, 2n/(2n-1)}(X)$, that $d\omega =d^{\upsilon}\omega$, that $\delta^{g_X}\omega=\delta^{\upsilon}\omega$, that $\nabla^{g_X}\omega =\nabla^{\upsilon}\omega$ and that  
\[\int_X|\nabla^{\upsilon}\omega |^2d\upsilon \le C(n, K, d, L).\]
\end{remark}


\begin{thebibliography}{99}
\bibitem[ACM14]{acm}
\textsc{K. Akutagawa, G. Carron and R. Mazzeo},
The Yamabe problem on stratified spaces,  Geom. Funct. Anal. 24 (2014), 1039-1079.
\bibitem[ACM13]{acm2}
\textsc{K. Akutagawa, G. Carron and R. Mazzeo},
The Yamabe problem on Dirichlet spaces, preprint, arXiv:1306.4373.
\bibitem[AGS14a]{ags0}
\textsc{L. Ambrosio, N. Gigli, and G. Savar\'e},
Metric measure spaces with Riemannian Ricci curvature bounded from below, Duke Math. J. 163 (2014), 1405-1490.
\bibitem[AGS14b]{ags}
\textsc{L. Ambrosio, N. Gigli, and G. Savar\'e},
Calculus and heat flow in metric measure spaces and applications to spaces with Ricci bounds from below, Invent. Math. 195 (2014), 289-391.
\bibitem[An89]{and}
\textsc{M. T. Anderson},
Ricci curvature bounds and Einstein metrics on compact manifolds, J. Amer. Math. Soc, 2 (1989), 455-490.
\bibitem[AC95]{ac}
\textsc{C. Ann\'e and B. Colbois},
Spectre du Laplacian agissant sur les $p$-formes diff\'erentielles et \'ecrasement d'anses, Math. Ann. 303 (1995), 545-573.
\bibitem[Aub76a]{au0}
\textsc{T. Aubin},
Probl\`emes isop\'erim\'etriques et espaces de Sobolev, J. Differential Geom. 11 (1976), 533-598.
\bibitem[Aub76b]{au}
\textsc{T. Aubin},
Equation diff\'erentielles non lin\'eaires et Probl\`emes de Yamabe concernant la courbure scalaire, J. Math. Pures Appl. 55 (1976), 269-296.
\bibitem[BKN89]{bkn}
\textsc{S. Bando, A. Kasue and H. Nakajima},
On a construction of coordinates at infinity on manifolds with fast curvature decay and maximal volume growth, Invent. Math. 97 (1989), 313-349.
\bibitem[BL83]{bl}
\textsc{H. Brezis and E. Lieb},
A relation between pointwise convergence of functions and convergence of functionals, Proc. Amer. Math. Soc. 88 (1983), 486-490.
\bibitem[Ch99]{ch1}
\textsc{J. Cheeger}, 
Differentiability of Lipschitz functions on metric measure spaces, Geom. Funct. Anal. 9 (1999), 428-517.
\bibitem[Ch01]{ch}
\textsc{J. Cheeger}, 
Degeneration of Riemannian metrics under Ricci curvature bounds, Lezioni Fermiane, Scuola Normale Superiore, Pisa, 2001.
\bibitem[ChC96]{ch-co}
\textsc{J. Cheeger and T. H. Colding}, 
Lower bounds on Ricci curvature and the almost rigidity of warped products, Ann. of Math. 144 (1996), 189-237.
\bibitem[ChC97]{ch-co1}
\textsc{J. Cheeger and T. H. Colding}, 
On the structure of spaces with Ricci curvature bounded below, I, J. Differential Geom. 45 (1997), 406-480.
\bibitem[ChC00a]{ch-co2}
\textsc{J. Cheeger and T. H. Colding}, 
On the structure of spaces with Ricci curvature bounded below, II, J. Differential Geom. 54 (2000), 13-35.
\bibitem[ChC00b]{ch-co3}
\textsc{J. Cheeger and T. H. Colding}, 
On the structure of spaces with Ricci curvature bounded below, III, J. Differential Geom. 54 (2000), 37-74.
\bibitem[ChCT02]{cct}
\textsc{J. Cheeger, T. H. Colding and G. Tian},
On the singularity of spaces with bounded Ricci curvature, Geom. Funct. Anal. 12 (2002), 873-914.
\bibitem[ChN14]{chna}
\textsc{J. Cheeger and A. Naber},
Regularity of Einstein manifolds and the codimension $4$ conjecture, preprint, arXiv:1406.6534.
\bibitem[ChgY75]{ch-yau}
\textsc{S. Y. Cheng and S. T. Yau},
Differential equations on Riemannian manifolds and their geometric applications, Comm. Pure Appl. Math. 28 (1975), 333-354.
\bibitem[CC90]{cc90}
\textsc{B. Colbois and G. Courtois},
A note on the first nonzero eigenvalue of the Laplacian acting on $p$-forms, Manuscripta Math. 68 (1990), 143-160.
\bibitem[CC00]{cc00}
\textsc{B. Colbois and G. Courtois},
Petites valeurs propres des $p$-fomes diff\'erentielles et classe d'Euler des $\mathbf{S}^1$-fibr\'e, Ann. Scient. \'Ec. Norm. Sup. 33 (2000), 611-645.
\bibitem[C97]{co3}
\textsc{T. H. Colding}, 
Ricci curvature and volume convergence, Ann. of Math. 145 (1997), 477-501.
\bibitem[CN12]{co-na1}
\textsc{T. H. Colding and A. Naber},
Sharp H$\ddot{\text{o}}$lder continuity of tangent cones for spaces with a lower Ricci curvature bound and applications, Ann. of Math. 176 (2012), 1173-1229.
\bibitem[Dod82]{dod}
\textsc{J. Dodziuk},
Eigenvalues of the Laplacian on forms, Proc. Amer. Math. Soc. 85 (1982), 437-443.
\bibitem[Din02]{di2}
\textsc{Y. Ding},
Heat kernels and Green's functions on limit spaces, Commun. Anal. Geom. 10 (2002),  475-514.
\bibitem[DH02]{heab}
\textsc{O. Druet and E. Hebey},
The AB program in Geometric Analysis. Sharp Sobolev inequalities  
and related problems, Mem. Amer. Math. Soc. 160 (761) (2002). 
\bibitem[Fed69]{fed}
\textsc{H. Federer}, 
Geometric measure theory, Springer, Berlin-New York, 1969.
\bibitem[Fuky87]{fu}
\textsc{K. Fukaya},
Collapsing of Riemannian manifolds and eigenvalues of the laplace operator, Invent. Math. 87 (1987), 517-547.
\bibitem[Fuks80]{fuku}
\textsc{M. Fukushima},
Dirichlet forms and Markoff processes, North Holland (Amsterdam) 1980.
\bibitem[Ga83]{gallot}
\textsc{S. Gallot},
In\'egalit\'es isop\'erim\'etriques, courbure de Ricci et invariants g\'eom\'etriques. I,
C. R. Acad Sci. Paris 296 (1983), 333-336.
\bibitem[G14]{gigli}
\textsc{N. Gigli},
Nonsmooth differential geometry - An approach tailored for spaces with Ricci curvature bounded from below, preprint, arXiv:1407.0809. 
\bibitem[GT01]{gt}
\textsc{D. Gilbarg and N. Trudinger},
Elliptic partial differential equations of second order, Reprint of the 1998 edition, Classics in Mathematics, Springer-Verlag, Berlin (2001).
\bibitem[Grig09]{grigo}
\textsc{A. Grigor'yan},
Heat kernel and analysis on manifolds, AMS/IP Studies in Advanced Mathematics, 47, 2009.
\bibitem[Gr81a]{gromov}
\textsc{M. Gromov},
Curvature, diameter and Betti numbers, Comment. Math. Helv. 56 (1981), 179-195.
\bibitem[Gr81b]{gr}
\textsc{M. Gromov}, 
Metric Structures for Riemannian and Non-Riemannian Spaces, Birkhauser Boston Inc, Boston, MA, 1999, Based on the 1981 French original 
[MR 85e:53051], With appendices by M. Katz, P. Pansu, and S. Semmes, Translated from the French by Sean Michael Bates. 
\bibitem[H95]{ha}
\textsc{P. Haj\l asz},
Sobolev spaces on an arbitrary metric space, Potential Anal. 5 (1995), 1211-1215.
\bibitem[HK95]{HK}
\textsc{P. Haj\l asz and P. Koskela},
Sobolev meets Poincar\'e, C. R. Acad Sci. Paris 320 (1995), 1211-1215.
\bibitem[HK00]{HK1}
\textsc{P. Haj\l asz and P. Koskela},
Sobolev met Poincar\'e, Mem. Amer. Math. Soc. 145 (2000), no 688.
\bibitem[Heb91]{he2}
\textsc{E. Hebey},
Sobolev spaces on Riemannian manifolds, Lecture Notes in Mathematics, Springer-Verlag, Berlin, Germany, (1991).
\bibitem[Heb96]{he1}
\textsc{E. Hebey},
Optimal Sobolev inequalities on complete Riemannian manifolds with Ricci curvature bounded below and positive injectivity radius, Amer. J. Math., 118 (1996), 291-300.
\bibitem[HV95]{he-va}
\textsc{E. Hebey and M. Vaugon},
The best constant problems in the sobolev embedding theorem for complete Riemannian manifolds, Duke Math. J. 79 (1995), 235-279.
\bibitem[HKT07]{hkt}
\textsc{T. Heikkinen, P. Koskela and H. Tuominen},
Sobolev-type spaces from generalized Poincar\'e inequalities, Studia Math. 181 (2007), 1-16.
\bibitem[H07]{heinonen}
\textsc{J. Heinonen}
Nonsmooth calculus, Bull. Amer. Math. Soc. 44 (2007), 163-232.
\bibitem[Hon11]{holip}
\textsc{S. Honda},
Ricci curvature and convergence of Lipschitz functions, Commun. Anal. Geom. 19 (2011), 79-158.
\bibitem[Hon14]{ho0}
\textsc{S. Honda},
A weakly second-order differential structure on rectifiable metric measure spaces, Geom. Topol. 18 (2014), 633-668.
\bibitem[Hon13a]{holp}
\textsc{S. Honda},
Ricci curvature and $L^p$-convergence, J. reine angew. Math. 705 (2015), 85-154.
\bibitem[Hon13b]{cheegerconstant}
\textsc{S. Honda},
Cheeger constant, $p$-Laplacian, and Gromov-Hausdorff convergence, preprint, arXiv:1310.0304.
\bibitem[HKX13]{hkz}
\textsc{B. Hua, M. Kell and C. Xia},
Harmonic functions on metric measure spaces, preprint, arXiv:1308.3607.
\bibitem[J14]{jiang}
\textsc{R. Jiang},
Cheeger-harmonic functions in metric measure spaces revisited, J. Funct. Anal. 266 (2014), 1373-1394.
\bibitem[K03]{keith2}
\textsc{S. Keith},
Modulus and Poincar\'e inequality on metric measure spaces, Math. Z. 245 (2003), 255-292.
\bibitem[KM96]{km}
\textsc{J. Kinnunen and O. Martio},
The Sobolev capacity on metric spaces, Ann. Acad. Sci. Fenn. Math. 21, (1996), 367-382.
\bibitem[KS03a]{KS}
\textsc{K. Kuwae and T. Shioya},
Convergence of spectral structures: a functional analytic theory and its applications to spectral geometry, Commun. Anal. Geom. 11 (2003), 599-673.
\bibitem[KS03b]{ks2}
\textsc{K. Kuwae and T. Shioya},
Sobolev and Dirichlet spaces over maps between metric spaces. J. reine angew. Math. 555 (2003), 39-75.
\bibitem[KS08]{KS2}
\textsc{K. Kuwae and T. Shioya},
Variational convergence over metric spaces, Trans. Amer. Math. Soc. 360 (2008), 35-75 (electronic).
\bibitem[LP87]{leep}
\textsc{J. M. Lee and T. H. Parker},
The Yamabe problem, Bull. Amer. Math. Soc. 17 (1987), 37-91.
\bibitem[L08]{L1}
\textsc{P. Li},
Harmonic functions on complete Riemannian manifolds, Handbook of geometric analysis. No. 1, 195-227, 
Adv. Lect. Math. 7, Int. Press, Somerville, MA, 2008. 
\bibitem[LS84]{LS}
\textsc{P. Li and R. Schoen},
$L^p$ and mean value properties of subharmonic functions on Riemannian manifolds, Acta Math. 153 (1984), 279-301.
\bibitem[LT91]{LT}
\textsc{P. Li and L. F. Tam},
The heat equation and harmonic maps of complete Riemannian manifolds, Invent. Math. 105 (1991), 1-46.
\bibitem[Lot02]{lott}
\textsc{J. Lott},
Collapsing and the differential form Laplacian: the case of a smooth limit space, Duke Math. J. 114 (2002), 267-306.
\bibitem[Lot03]{lott2}
\textsc{J. Lott},
Remark about the spectrum of the $p$-form Laplacian under a collapse with curvature bounded below, Proc. Amer. Math. Soc. 132 (2003), 911-918. 
\bibitem[LV09]{lv}
\textsc{J. Lott and C. Villani},
Ricci curvature for metric measure spaces via optimal transport, Ann. of Math. 169 (2009), 903-991.
\bibitem[MS95]{mas}
\textsc{P. Maheux and L. Saloff-Coste},
Analyse sur les boules d'un op\'erateur sous-elliptique. Math. Ann. 303 (1995), 713-740.
\bibitem[P97]{pere}
\textsc{G. Perelman},
Construction of manifolds of positive Ricci curvature with big volume and
large Betti numbers, Comparison Geometry, Math. Sci. Res. Inst. Publ. 30, Cambridge Univ.
Press, Cambridge (1997), 157-163. 
\bibitem[R19]{rad}
\textsc{H. Rademacher},
Uber partielle und totale Differenzierbarkeit I, Math. Ann. 79 (1919), 340-359.
\bibitem[S96]{sakai}
\textsc{T. Sakai},
Riemannian Geometry, Amer. Math. Soc. Transl. (1996). 
\bibitem[Sc84]{schoe}
\textsc{R. Schoen},
Conformal deformation of a Riemannian metric to constant scalar curvature, J. Differential Geom. 20 (1984), 479-495.
\bibitem[SY88]{sc-ya}
\textsc{R. Schoen and S. T. Yau},
Conformally flat manifolds, Kleinian groups and scalar curvature, Invent. Math. 92 (1988), 47-71.
\bibitem[Sh00]{shanm2}
\textsc{N. Shanmugalingam},
Newtonian spaces: An extension of Sobolev spaces to metric measure spaces, Rev. Mat. Iberoam. 16 (2000), 243-279.
\bibitem[Sh01]{shanm}
\textsc{N. Shanmugalingam},
Harmonic functions on metric spaces, Illinois J. Math. 45 (2001), 1021-1050.
\bibitem[Sim83]{le}
\textsc{L. M. Simon},
Lectures on Geometric Measure Theory, Proc. of the Center for Mathematical Analysis 3, Australian National University, 1983.
\bibitem[SW01]{sw1}
\textsc{C. Sormani and G. Wei},
Hausdorff convergence and universal covers, Trans. Amer. Math. Soc. 353 (2001), 3585-3602.
\bibitem[SW04]{sw1}
\textsc{C. Sormani and G. Wei},
Universal covers for Hausdorff limits of noncompact spaces, Trans. Amer. Math. Soc. 356 (2004), 1233-1270.
\bibitem[St06a]{sturm1}
\textsc{K. -T. Sturm},
On the geometry of metric measure spaces I, Acta. Math. 196 (2006), 65-131.
\bibitem[St06b]{sturm2}
\textsc{K. -T. Sturm},
On the geometry of metric measure spaces II, Acta. Math. 196 (2006), 133-177.
\bibitem[Tak02]{taka}
\textsc{J. Takahashi},
Small eigenbalues on $p$-forms for collapsings of the even-dimensional spheres, Manuscripta Math. 109 (2002), 63-71.
\bibitem[Tr68]{tru}
\textsc{N. Trudinger},
Remarks concerning the conformal deformation of Riemannian structures on compact manifolds, Ann. Scuola Norm. Sup. Pisa 22 (1968), 265-274.
\bibitem[T90]{tian}
\textsc{G. Tian},
On Calabi's conjecture for complex surfaces with positive first Chern class, Invent. Math. 101 (1990), 101-172.
\bibitem[V08]{vill}
\textsc{C. Villani},
Optimal transport, old and new, Springer-Verlag, 2008.
\bibitem[Yam60]{yamabe}
\textsc{H. Yamabe},
On a deformation of Riemannian structures on compact manifolds, Osaka J. Math. 12 (1960), 21-37.
\bibitem[Y75]{y1}
\textsc{S. T. Yau},
Harmonic functions on complete Riemannian manifolds, Commun. Pure and Appl. Math. 28 (1975), 201-228.
\bibitem[Y76]{y2}
\textsc{S. T. Yau},
Some function-theoretic properties of complete Riemannian manifold and their applications to geometry, Indiana Univ. Math. J. 25 (1976), 659-670.
\end{thebibliography}
\end{document}